\documentclass[12pt]{amsart}
\usepackage{amsmath,amsfonts,latexsym,graphicx,amssymb,url}
\usepackage{hyperref,pdfsync}
\usepackage{amsmath,txfonts,pifont,bbding,pxfonts}
\usepackage{manfnt}
\usepackage[active]{srcltx}
\usepackage{cases}
\usepackage{wasysym,pstricks, enumerate,subfigure}
\usepackage{lscape,color}
\usepackage{soul}
\usepackage[none]{hyphenat}
\usepackage[mathscr]{euscript}
\usepackage{wrapfig, subfigure,graphicx}

\newcommand{\R}{{\mathbb R}} 

\newcommand{\p}{\varphi}
\renewcommand{\(}{\left(}
\renewcommand{\)}{\right)}
\newtheorem{theorem}{Theorem}[section]
\newtheorem{claim}{Step} 
\newtheorem{step}{Step} 
\newtheorem{corollary}[theorem]{Corollary}
\newtheorem{lemma}[theorem]{Lemma}
\newtheorem{proposition}[theorem]{Proposition}
\newtheorem{definition}[theorem]{Definition}
\newtheorem{remark}[theorem]{Remark}

\begin{document}











\title[On the existence of dichromatic single element lenses ]{On the existence of dichromatic \\single element lenses}
\author[C. E. Guti\'errez and A. Sabra]{Cristian E. Guti\'errez and Ahmad Sabra}
\thanks{\today}
\address{Department of Mathematics\\Temple University\\Philadelphia, PA 19122}
\email{gutierre@temple.edu}
\address{Faculty of Mathematics, Informatics, and Mechanics,
University of Warsaw, Poland}
\email{sabra@mimuw.edu.pl}

\begin{abstract}
Due to dispersion, light with different wavelengths, or colors, is refracted at different angles. Our purpose is to determine when is it possible to design a lens made of a single homogeneous material so that it refracts light superposition of two colors into a desired fixed final direction. Two problems are considered: one is when light emanates in a parallel beam and the other is when light emanates from a point source. For the first problem, and when the direction of the parallel beam is different from the final desired direction, we show that such a lens does not exist; otherwise we prove the solution is trivial, i.e., the lens is confined between two parallel planes. For the second problem we prove that is impossible to design such a lens when the desired final direction is not among the set of incident directions. Otherwise, solving an appropriate system of functional equations we show that a local solution exists.   
\end{abstract}


\maketitle

\tableofcontents

\section{Introduction}
\setcounter{equation}{0}
We showed in \cite{gutierrez-sabra:asphericallensdesignandimaging} that given a function $u$ in $\Omega\subset \R^2$ and a unit direction $w\in S^2$ there exists a surface parametrized by a function $f$ such that the lens sandwiched by $u$ and $f$, made of an homogeneous material and denoted by $(u,f)$, refracts monochromatic light emanating vertically from $\Omega$ into the direction $w$. In the earlier paper \cite{gutierrez:asphericallensdesign}, a similar result is proved when light emanates from a point source. 
The purpose of this paper is to study if it is possible to design simple lenses doing similar refracting jobs for non monochromatic light.
By a simple (or single element) lens we mean a domain in $\R^3$ bounded by two smooth surfaces that is filled with an homogeneous material. 

To do this we need to deal with dispersion: since the index of refraction of a material depends on the wavelength of the radiation, a non monochromatic light ray after passing through a lens splits into several rays having different refraction angles and wavelengths. 
Therefore, when white light is refracted by a single lens each color comes to a focus at a different distance from the objective. 
This is called chromatic aberration and plays a vital role in lens design,  
see \cite[Chapter 5]{kingslake:lensdesignfundamentals}.  
Materials have various degrees of dispersion, and
low dispersion ones are used in the manufacturing of photographic lenses, see \cite{Cannon}.
%
A way to correct chromatic aberration is to build lenses composed of various simple lenses made of different materials.  
Also chromatic aberration has recently being handled numerically using demosaicing algorithms, see \cite{demosaicingalgorithms}.
The way in which the refractive index depends of the wavelength is given by a formula for the dispersion of light due to A. Cauchy: the refractive index $n$ in terms of the wavelength $\lambda$ is given by
$n(\lambda)=A_1+\dfrac{A_2}{\lambda^2}+\dfrac{A_4}{\lambda^4}+\cdots $, where $A_i$ are constants depending on the material \cite{cauchy-memoire-sur-la-dispersio-de-la-luminere}. The validity of this formula is in the visible wavelength range, see \cite[pp. 99-102]{book:born-wolf} for its accurateness in various materials.
A more accurate formula was derived by Sellmeier, see \cite[Section 23.5]{jenkins-white:fundamentalsofoptics}. 

A first result related to our question is 
 that there is no single lens bounded by two spherical surfaces that refracts non monochromatic radiation from a point into a fixed direction; this was originally stated by K. Schwarzschild \cite{schwarzschild:1905-telescope}.
The question of designing a single lens, non spherical, that focuses one point into a fixed direction for light containing only two colors, i.e., for two refractive indices $n\neq \bar n$, is considered in \cite{schultz:1983achromaticsinglelens} in the plane; but no mathematically rigorous proof is given. In fact, by tracing back and forth rays of both colors, the author describes how a finite number of points should be located on the faces of the desired lens and he claims, without proof, that the desired surfaces can be completed by interpolating adjacent points with third degree polynomials. Such an interpolation will give an undesired refracting behavior outside the fixed number of points considered.
For the existence of rotationally symmetric lenses capable of focusing one point into two points for two different refractive indices see \cite{vanbrunt-ockendon:lenstwowavelengths}, \cite{vaanbrunt:refinements}. 
The results of all these authors require size conditions on $n,\bar n$.
 The monochromatic case is due to Friedman and MacLeod  \cite{friedmanmacleod:optimaldesignopticalens} and  Rogers \cite{rogers1988:picardtypethemlensdesign}.
The solutions obtained are analytic functions.
These results are all two dimensional and therefore concern only to rotationally symmetric lenses.

In view of all this, we now state precisely the problems that are considered and solved in this paper.

Problem A: is there a single lens sandwiched by a surface $L$ given by the graph of a function $u$ in $\Omega$, the lower surface of the lens, and a surface $S$, the top part of the lens, such that each ray of dichromatic light (superposition of two colors) emanating in the vertical direction $e$ from $(x,0)$ for $x\in \Omega$  is refracted by the lens into the direction $w$?
We denote such a lens by the pair $(L,S)$.
Notice that when a dichromatic ray enters the lens, due to chromatic dispersion, it splits into two rays having different directions and colors, say red and blue, that they both travel inside the lens until they exit it at points on the surface $S$ and then both continue traveling in the direction $w$; see Figure \ref{fig:Problems A and B}(a).
\begin{figure}[htp]
\begin{center}
    \subfigure[$ $]{\label{fig:Problem A}\includegraphics[height=2.9in]{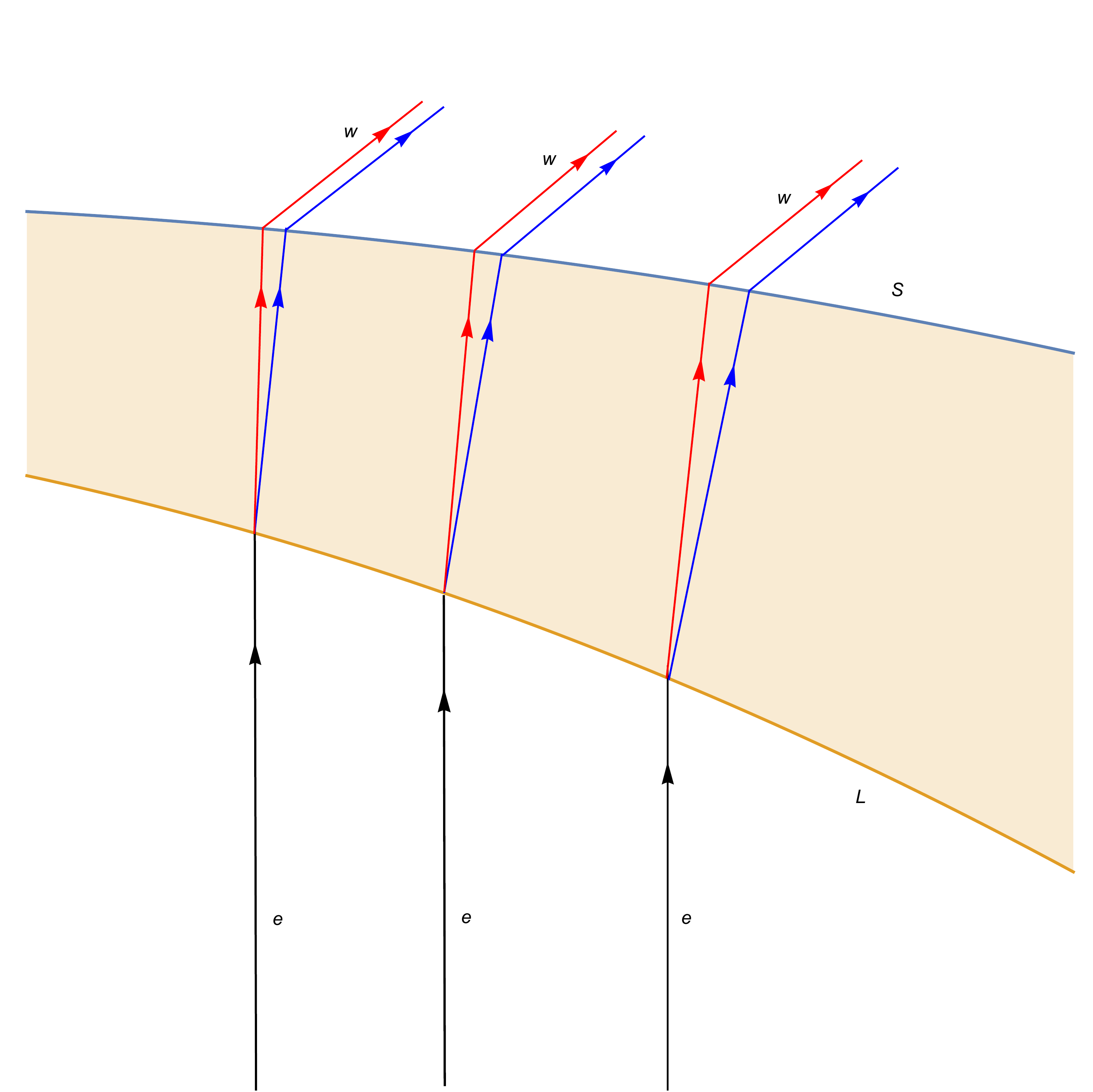}}\,
    \subfigure[$ $]{\label{fig:Problem B}\includegraphics[width=2.8in]{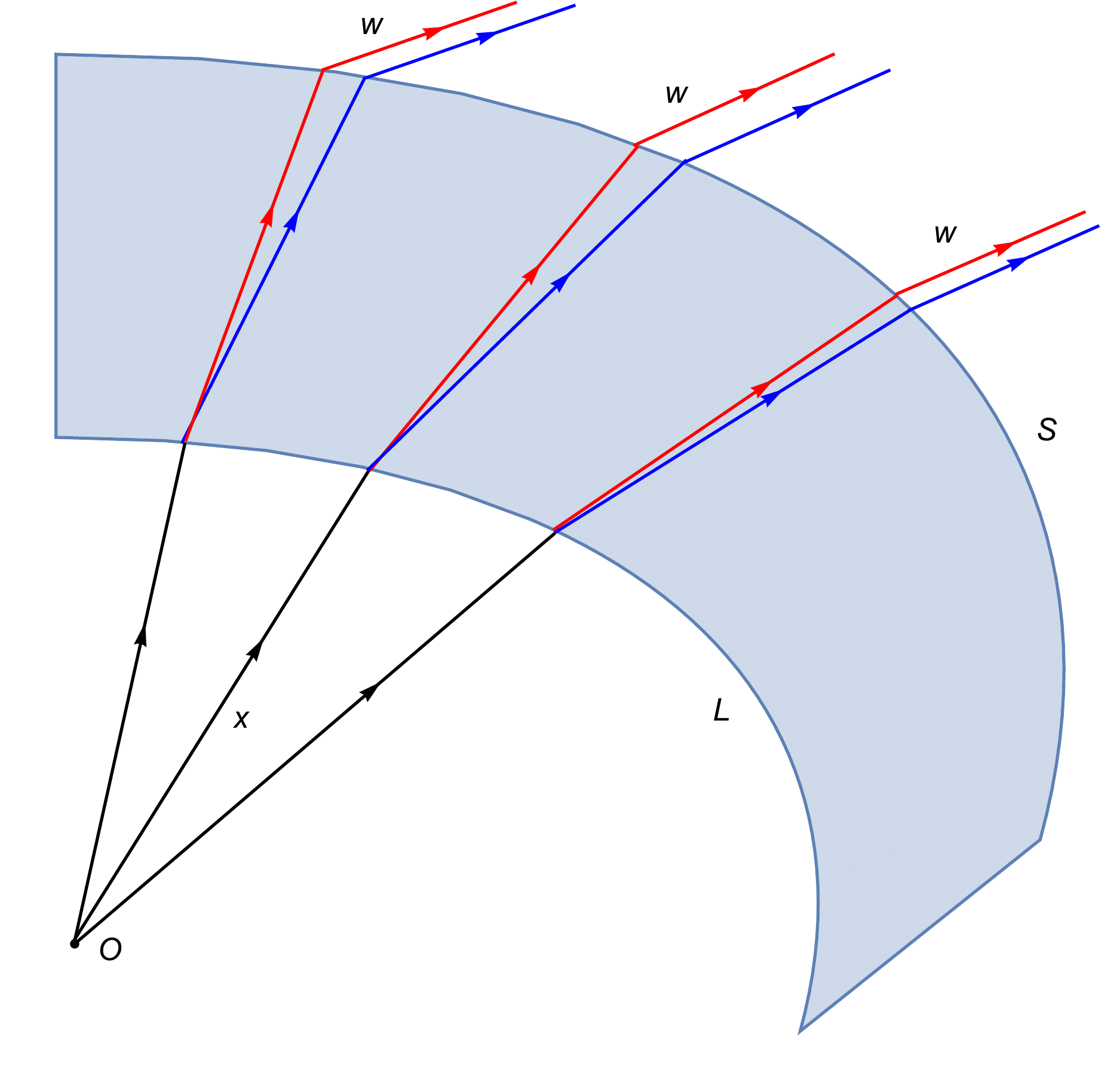}} 
\end{center}
  \caption{Problems A and B}
  \label{fig:Problems A and B}
\end{figure}


Problem B: a similar question is when the rays emanate from a point source $O$ and we ask if a a single lens $(L,S)$ exists such that all rays are refracted into a fixed given direction $w$. Now $L$ is given parametrically by $\rho(x)x$ for $x\in \Omega\subset S^2$;
see Figure \ref{fig:Problems A and B}(b).
\\

We will show in Section \ref{The collimated case} using Brouwer fixed point theorem that Problem A has no solution if $w\neq e$. In case $w=e$, the unique solution to Problem A is the trivial solution: $L$ and $S$ are contained in two horizontal planes.
This is the contents of Theorem \ref{thm:monotonicity of phi}. 
On the other hand, since Problem A is solvable for monochromatic light and for each given lower surface $L$, we obtain two single lenses, one for each color.
%
%
We then show in Section \ref{subsec:difference of the upper surface} that the difference between the upper surfaces of these two lenses can be estimated by the difference between the refractive indices for each color. 

Concerning Problem B, we prove in Theorem \ref{thm:nonexistence point source}, also using Brouwer fixed point theorem, that if $w\notin \Omega$ then Problem $B$ has no solution. 
The case when $w\in \Omega$ requires a more elaborate and long approach. 
In fact we show in Sections \ref{subsec:Problem B implies System} and \ref{subsec:converse functional implies optic}, in dimension two, that the solvability of Problem B is equivalent to solve a system of first order functional differential equations. 
For this we need an existence theorem for these type of equations that was introduced by Rogers in \cite{rogers1988:picardtypethemlensdesign}.
We provide in Section \ref{sec:FOD} a simpler proof of this existence and uniqueness result of local solutions using the Banach fixed point theorem, Theorems \ref{thm: Existence} and \ref{thm:Uniqueness}. Section \ref{sec:FOD} is self contained and has independent interest.
%
%
The existence of local solutions to Problem B in the plane is then proved in Section \ref{sec:existence of local solution} by application of Theorem \ref{thm: Existence}. For this it is necessary to assume conditions on the ratio between the thickness of the lens and its distance to the point source, Theorem \ref{thm:Existence last}. 
We also derive a necessary condition for the solvability of Problem B, see Corollary \ref{cor: Sufficient Condition}.  
To sum up our result: for $w=e\in \Omega$ and fixing two points $P$ and $Q$ on the positive $y$-axis, with $|Q|>|P|$ and letting $k=|P|/|Q-P|$, we show that if $k$ is small then there exists a unique lens $(L,S)$ local solution to Problem $B$ such that $L$ passes through the point $P$ and $S$ through $Q$; otherwise, for $k$ large no solution exists.  For intermediate values of $k$ see Remark \ref{rmk:final remark}.

The analogue of Problem B for more than two colors has no solution, i.e., if rays emitted from the origin are superposition of three or more colors, there is no simple lens refracting these rays into a unique direction $w$, see Remark \ref{rmk:three colors}.

We close this introduction mentioning that a number of results have been recently obtained for refraction of monochromatic light, these include the papers \cite{gutierrez-huang:farfieldrefractor}, \cite{gutierrez-huang:near-field-refractor},  \cite{gutierrez:cimelectures}, \cite{2016-karakhanyan:parallel-beam-refractor}, \cite{deleo-gutierrez-mawi:numerical-refractor}, and \cite{gutierrez-sabra:freeformgeneralfields}.

\section{Preliminaries}
\setcounter{equation}{0}
In this section we mention some consequences from the Snell law that will be used later.
In an homogeneous and isotropic medium, the index of refraction depends on the wavelength of light.
Suppose $\Gamma$ is a surface in $\R^3$ separating two media
I and II that are homogeneous and isotropic. 
If a
ray of monochromatic light
having unit direction $x$ and traveling
through the medium I hits $\Gamma$ at the point $P$, then this ray
is refracted in the unit direction $m$ through medium II
according with the Snell law in vector form, \cite{luneburg:maththeoryofoptics},
\begin{equation}\label{snellwithcrossproduct}
n_{1}(x\times \nu)=n_{2}(m\times \nu),
\end{equation} 
where $\nu$ is the unit normal to the surface to $\Gamma$ at $P$ going towards the medium II,
and $n_1,n_2$ are the refractive indices for the corresponding monochromatic light.
This has several consequences:
\begin{enumerate}
\item[(a)] the vectors $x,m,\nu$ are all on the same plane, called plane of incidence;
\item[(b)] the well known Snell law in scalar form
$$n_1\sin \theta_1= n_2\sin
\theta_2,$$ 
where $\theta_1$ is the angle between $x$ and $\nu$
(the angle of incidence),
$\theta_2$ the angle between $m$ and $\nu$ (the angle of refraction).
\end{enumerate}

From \eqref{snellwithcrossproduct}, 
with $\kappa=n_2/n_1$, 
\begin{equation}\label{eq:snellvectorform}
x-(n_2/n_1) \,m =\lambda \nu,
\end{equation}
with
\begin{equation}\label{formulaforlambda}
\lambda=x\cdot \nu -\sqrt{\kappa^2-1+(x\cdot \nu)^{2}}=\Phi_\kappa(x\cdot \nu).
\end{equation}
Notice that $\lambda>0$ when $\kappa<1$, and $\lambda<0$ if $\kappa>1$.
When $\kappa<1$ total reflection occurs, unless $x\cdot \nu\geq \sqrt{1-\kappa^2}$, see \cite{book:born-wolf} and \cite[Sec. 2]{gutierrez:cimelectures}.

The following lemmas will be used in the remaining sections of the paper.

\begin{lemma}\label{lm:lemma equality of three normals}
Assume we have monochromatic light.
 Let $\Gamma_1$ and $\Gamma_2$ be two surfaces enclosing a lens with refractive index $n_2$, 
and the outside of the lens is a medium with refractive index $n_1$ with $n_1\neq n_2$.

Suppose an incident ray with unit direction $x$ strikes $\Gamma_1$ at $P$, the ray propagates inside the lens and is refracted at $Q\in \Gamma_2$ into the unit direction $w$.
Then $w=x$ if and only if the unit normals $\nu_1(P)=\nu_2(Q)$.

\end{lemma}
\begin{proof}
From the Snell law applied at $P$ and $Q$
\[
x-(n_2/n_1)\,m=\lambda_1\,\nu_1(P), \qquad m-(n_1/n_2)\,w=\lambda_2\,\nu_2(Q),
\]
then
\begin{equation}\label{eq:x minus w equals normals}
x-w=\lambda_1\,\nu_1(P)+(n_2/n_1)\,\lambda_2\,\nu_2(Q).
\end{equation}
If $x=w$, since $\lambda_1$ and $-(n_2/n_1)\,\lambda_2$ have the same sign and the normals are unit vectors, we conclude
\[
\nu_1(P)=\nu_2(Q).
\]
Conversely, if $\nu_1(P)=\nu_2(Q):=\nu$, then from \eqref{eq:x minus w equals normals} $x-w=\(\lambda_1+(n_1/n_2)\,\lambda_2\, \)\,\nu$. Notice that $m\cdot \nu=(n_1/n_2)\(x\cdot \nu-\lambda_1\)=(n_1/n_2)\sqrt{(n_2/n_1)^2-1+(x\cdot \nu)^2}$. Hence from \eqref{formulaforlambda} 
\begin{align*}
\lambda_1+(n_2/n_1)\,\lambda_2
&=x\cdot \nu-(n_2/n_1)\,m\cdot \nu +(n_2/n_1)\(m\cdot \nu-\sqrt{(n_1/n_2)^2-1+(m\cdot \nu)^2}\)=0.
\end{align*}

\end{proof}

Let us now consider the case of dichromatic light, i.e., a mix of two colors b and r. That is, if a ray with direction $x$ in vacuum strikes a surface $\Gamma$ at $P$ separating two media, then the ray splits into two rays one with color b and direction $m_b$, and another with color r and direction $m_r$. Here  $m_r$ satisfies \eqref{snellwithcrossproduct} with $n_1=1$ and $n_2=n_r$ (the refractive index for the color r) and $m_b$ satisfies \eqref{snellwithcrossproduct} with $n_1=1$ and $n_2=n_b$ (the refractive index for the color b).
Notice $m_b,m_r$ are both in the plane of incidence containing $P$, the vector $x$, and $\nu(P)$ the normal to $\Gamma$ at $P$. 
Assuming $n_b>n_r$, i.e., rays with color r travel faster than rays with color b, then for a given incidence angle $\theta$ by the Snell law the angles of refraction satisfy $\theta_b\leq \theta_r$. In fact
$$\sin \theta=n_b\sin \theta_b=n_r\sin \theta_r,$$
obtaining the following Lemma.

\begin{lemma}\label{item:m vectors are equal} 
Suppose a dichromatic ray with unit direction $x$ strikes a surface $\Gamma$ at a point $P$ having normal $\nu$.
Then $m_b=m_r$ if and only if $x=\nu$.
%
\end{lemma}

\section{The collimated case: Problem A}\label{The collimated case}
\setcounter{equation}{0}


In this section we consider the following set up. We are given $\Omega\subseteq\R^2$ a compact and convex set with nonempty interior, and $w$ a unit vector in $\R^3$. Dichromatic rays with colors b and r are emitted from $(t,0)$, with $t\in \Omega$, into the vertical direction $e=(0,0,1)$. 
By application of the results from  \cite{gutierrez-sabra:asphericallensdesignandimaging} with $n_1=n_3=1$ and $n_2=n_r$, we have the following.
Given $u\in C^2$, there exist surfaces parametrized by 
$f_r(t)=(t,u(t))+d_r(t)m_r(t)$, with $m_r(t)=\dfrac{1}{n_r}\(e-\lambda_r \nu_u(t)\)$ where $\lambda_r
=\Phi_{n_r}\(e\cdot \nu_u(t)\)$ from \eqref{formulaforlambda}, $\nu_u(t)=\dfrac{\(-\nabla u(t),1\)}{\sqrt{1+|\nabla u(t)|^2}}$ the unit normal at $(t,u(t))$, and
\begin{equation}\label{eq:formula for dbt}
d_r(t)=\dfrac{C_r-(e-w)\cdot (t,u(t))}{n_r-w\cdot m_r(t)} 
\end{equation}
from \cite[Formula (3.14)]{gutierrez-sabra:asphericallensdesignandimaging}, such that lens bounded between $u$ and $f_r$ refracts the rays with color r into $w$. Here the constant $C_r$ is chosen so that $d_r(t)>0$ and $f_r$ has a normal vector at every point. This choice is possible from \cite[Theorem 3.2 and Corollary 3.3]{gutierrez-sabra:freeformgeneralfields}.

 Likewise there exist surfaces parametrized by 
$f_b(t)=(t,u(t))+d_b(t)m_b(t)$, with similar quantities as before with r replaced by b,
such that lens bounded between $u$ and $f_b$ refracts the rays with color b into $w$. 

We assume that $n_b>n_r>1$, where $n_b, n_r$ are the refractive indices of the material of the lens corresponding to monochromatic light having colors b or r, and the medium surrounding the lens is vacuum. 

To avoid total reflection compatibility conditions between $u$ and $w$ are needed, see \cite[condition (3.4)]{gutierrez-sabra:asphericallensdesignandimaging} which in our case reads
\[
\lambda_r\,\nu_u(t)\cdot w\leq e\cdot w-1,\text{ and } \lambda_b\,\nu_u(t)\cdot w\leq e\cdot w-1.
\]
If $w=e$, these two conditions are automatically satisfied because $\lambda_r, \lambda_b$ are both negative and $\nu_u(t)\cdot e=\dfrac{1}{\sqrt{1+|\nabla u(t)|^2}}>0$.

The problem we consider in this section is to determine if there exist $u$ and corresponding surfaces $f_r$ and $f_b$ for each color  such that $f_r$ can be obtained by a reparametrization of $f_b$. That is, if there exist a positive function $u\in C^2(\Omega)$, real numbers $C_r$ and $C_b$, and a continuous map $\p:\Omega \to \Omega$ such that the surfaces $f_r$ and $f_b$, corresponding to $u$, $C_r,C_b$, have normals at each point and 
\begin{equation}\label{eq:equality by re parametrization}
f_r(t)=f_b(\p(t))\qquad \forall t\in \Omega,
\end{equation} 
we refer to this as {\it Problem A}. Notice that if a solution exists 
then $f_r(\Omega) \subseteq f_b(\Omega).$ From an optical point of view, this means that the lens sandwiched between $u$ and $f_b$ refracts both colors into $w$.
Notice that there could be points in $f_b(\Omega)$ that are not reached by red rays. 

The answer to Problem A is given
in the following theorem.

\begin{theorem}\label{thm:monotonicity of phi}
If $w\neq e$, then Problem A has no solution, and if $w=e$ the only solutions to Problem A are lenses enclosed by two horizontal planes.
\end{theorem}
To prove this theorem we need the following lemma.
\begin{lemma}\label{lem:fixed point}
Given a surface described by $u\in C^2(\Omega)$ and the unit direction $w$, let $f_r$ and $f_b$ be the surfaces parametrized as above. If $f_r(t)=f_b(t)$ for some $t\in \Omega$, then $\nu_u(t)=e$, the unit normal vector to $u$ at $(t,u(t))$.

\end{lemma}

\begin{proof}
Since
\[
f_b(t)=(t,u(t))+d_b(t)\,m_b(t)=f_r(t)=(t,u(t))+d_r(t)\,m_r(t)
\] 
we get $d_b(t)\,m_b(t)=d_r(t)\,m_r(t)$, and since $m_r,m_b$ are unit, $d_b(t)=d_r(t)$. Therefore
$m_b(t)=m_r(t)$ which from Lemma \ref{item:m vectors are equal} implies that $\nu_u(t)=e$.
\end{proof}

\begin{proof}[Proof of Theorem \ref{thm:monotonicity of phi}]

To show the first part of the theorem, suppose by contradiction that Problem A has a solution with $w\neq e$. 
Since $\Omega$ is compact and convex, by Brouwer fixed point theorem \cite[Sect. 2, Chap. XVI]{dugundji:topology} there is $t_0\in \Omega$ such that $\p(t_0)=t_0$, and so from \eqref{eq:equality by re parametrization} $f_r(t_0)=f_b(t_0)$. Hence from Lemma \ref{lem:fixed point} $\nu_u(t_0)=e$. By Snell's law at $(t_0,u(t_0))$, $m_b(t_0)=m_r(t_0)=e$. Since $n_r\neq n_b$, and both colors with direction $e$ are refracted at $f_b(t_0)=f_r(t_0)$ into the direction $w$, it follows again from the Snell's law that $w=e$, a contradiction. 


To show the second part of the Theorem, assume there exist $u$ and $\varphi:\Omega\to \Omega$ such that problem A has a solution. Let $t\in \Omega$, and $Q=f_r(t)=f_b(\varphi(t))$. Since the ray emitted from $(t,0)$ with direction $e$ and color $r$ is refracted by $(u,f_r)$ into $e$ at $Q$, then by Lemma \ref{lm:lemma equality of three normals} $\nu_u(t)=\nu(Q)$, where $\nu(Q)$ denotes the normal to the upper face of the lens at $Q$. Similarly, applying Lemma \ref{lm:lemma equality of three normals} to the color b we have $\nu_u(\varphi(t))=\nu(Q)$. We conclude that for every $t\in \Omega$
\begin{equation}\label{eq:equality of normals}
\nu_u(t)=\nu_u(\varphi(t)).
\end{equation}

We will show that $u$ is constant, i.e., $\nabla u(t)=0$, for all $t\in \Omega$ connected. Suppose by contradiction that there exists $t_0\in \Omega$, with $\nabla u(t_0)\neq 0$. If $t_1=\p(t_0)$, then $t_1\neq t_0$. Otherwise, from Lemma \ref{lem:fixed point} $\nu_u(t_0)=\dfrac{\(-\nabla u(t_0),1\)}{\sqrt{1+\left|\nabla u(t_0)\right|^2}}= e$, so $\nabla u(t_0)=0$. Also from \eqref{eq:equality of normals}, $\nu_u(t_0)=\nu_u(t_1)$.
 Let $L_r(t_0)$ be the red ray from $(t_0,u(t_0))$ to $f_r(t_0)$, and let $L_b(t_1)$
be the blue ray from $(t_1,u(t_1))$ to $f_b(t_1)$. We have that $L_r(t_0)$ and $L_b(t_1)$ intersect at $Q_0:=f_r(t_0)=f_b(t_1)$.
If $\Pi_r$ denotes the plane of incidence passing through $(t_0,u(t_0))$ containing the directions $e$ and $\nu_u(t_0)$,
and $\Pi_b$ denotes the plane of incidence through $(t_1,u(t_1))$ containing the directions $e$ and $\nu_u(t_1)$, then $\Pi_r$ and $\Pi_b$ are parallel since $\nu_u(t_0)=\nu_u(t_1)$. 
Also by Snell's law $L_r(t_0)\subset \Pi_r$ and $L_b(t_1)\subset \Pi_b$, so  $Q_0\in \Pi_r \cap \Pi_b$. We then obtain $\Pi_r= \Pi_b:=\Pi$.
 
Let $\ell$ denote the segment $\Omega\cap \Pi$. We deduce from the above that  $t_0,t_1\in \ell$. Next, let $t_2=\p(t_1)$. As before, since $\nabla u(t_1)=\nabla u(t_0)\neq 0$, by Lemma \ref{lem:fixed point} $t_2\neq t_1$; and by \eqref{eq:equality of normals} $\nu_u(t_1)=\nu_u(t_2)$. Let $\Pi_2$ be the plane through $(t_2,u(t_2))$ and containing the vectors $e$ and $\nu_u(t_2)$.
We have $L_b(t_2)\subset \Pi_2$, $f_r(t_1)=f_b(t_2)$, and $f_r(t_1)\in \Pi$. 
Therefore $\Pi_2=\Pi$, in particular, $t_2\in \ell$.

Let $\ell_1$ denote the half line starting from $t_1$ and containing $t_0$. We claim that $t_2\notin \ell_1$. 
In fact, we first have that $L_r(t_0)$ and $L_b(t_1)$ intersect at $Q_0$. Since $\nu_u(t_0)=\nu_u(t_1):=\nu$, $L_r(t_0)$ is parallel to $L_r(t_1)$, and $L_b(t_0)$ is parallel to $L_b(t_1)$. And since $n_b>n_r$, it follows from the Snell law that the angle of refraction $\theta_b$ for the blue ray $L_b(t_0)$ and the angle of refraction $\theta_r$ of the red ray $L_r(t_1)$ satisfy $\theta_b<\theta_r$. Hence $L_b(t_0)$ and $L_r(t_1)$ diverge. Moreover, notice that all rays are on the plane $\Pi$, and $L_b(t_2)$ is parallel to $L_b(t_1)$. Then, if $t_2\in \ell_1$ , we have $L_b(t_2)$ and $L_r(t_1)$ diverge and cannot intersect, a contradiction since $f_b(t_2)=f_r(t_1)$ and the claim is proved;
%
%
%
%
see Figure \ref{fig:divergent rays} illustrating that $t_2$ cannot be on $\ell_1$. 
\begin{figure}[htp]
\includegraphics[width=3in]{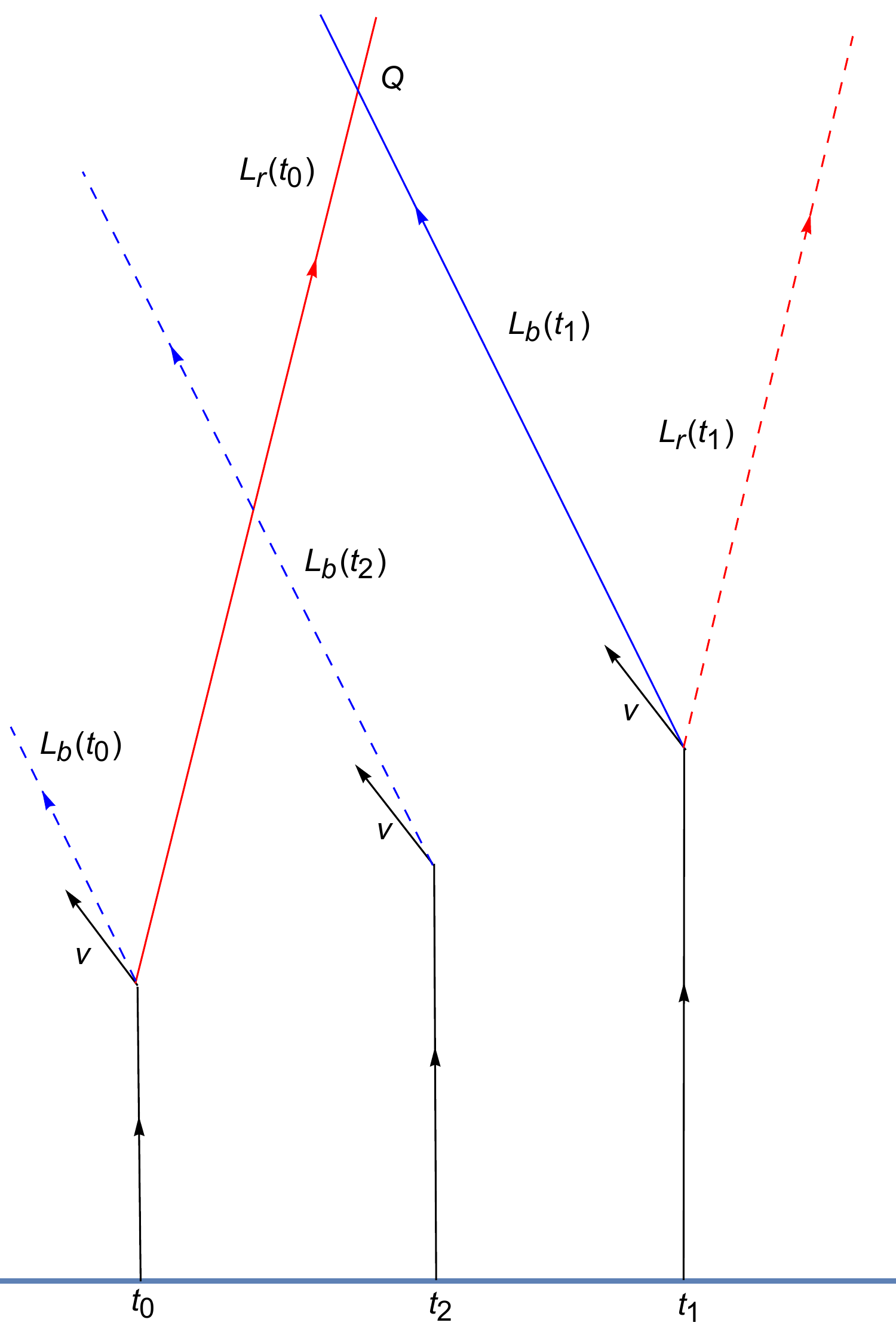}
\caption{$t_2\notin \ell_1$}
\label{fig:divergent rays}
\end{figure}

Continuing in this way we construct the sequence $t_k=\p(t_{k-1})$. 
By \eqref{eq:equality of normals} $\nu_u(t_k)=\nu_u(t_{k-1})=\cdots=\nu_u(t_0)\neq 0$, 
and again by Lemma \ref{lem:fixed point} $t_k\neq t_{k-1}$. 
By Snell's law $\{L_b(t_k)\}$ are all parallel, $\{L_r(t_k)\}$ are all parallel and arguing as before they are all contained in $\Pi$, and then $t_k\in \ell$. In addition, the angles between $L_r(t_k)$ and 
$L_b(t_k)$ are the same for all $k$. Also, for $k\geq 1$, if $\ell_k$ is the half line with origin $t_{k}$ and passing through $t_{k-1}$, then as above $t_{k+1}\notin \ell_k$. Hence the 
sequence $\{t_k\}$ is decreasing or increasing on the line $\ell$. Therefore $t_k$ converges to some $\hat t\in \ell$ so by continuity $\p(\hat t)=\hat t$. Hence by Lemma \ref{lem:fixed point} 
$\nabla u(\hat t)=0$, but since $\nabla u(t_k)=\nabla u(t_0)\neq 0$ for all $k$, and $u$ is $C^2$ we obtain a contradiction. Thus $u$ is constant in $\Omega$. Since the lower face is then contained in a horizontal plane, $m_r(t)=m_b(t)=e=(0,0,1)$. Hence from the form of the parameterizations of $f_b$ and $f_r$, and since from \eqref{eq:formula for dbt} $d_r$ and $d_b$ are constants, the upper face of the lens is also contained in a horizontal plane.
\end{proof}

\subsection{Estimates of the upper surfaces for two colors}\label{subsec:difference of the upper surface}

The purpose of this section is to measure how far the surfaces $f_r$ and $f_b$ can be when $w=e$. We shall prove the following.

\begin{proposition}\label{prop:estimate for fr and fb passing through a fixed point}
Suppose $f_r(t_0)=f_b(t_0)$  at some point $t_0$, and $w=e$. Then 
\[
|f_r(t)-f_b(t)|\leq \bar C\,|n_b-n_r|
\] 
for all $t$ with a constant $\bar C$ depending only $t_0$ and $n_r,n_b$.
\end{proposition}
\begin{proof}
We begin showing an upper estimate of the difference between $m_r(t)$ and $m_b(t)$.
To simplify the notation write $\nu=\nu_u$.

We have
\begin{align*}
m_b(t)&=\dfrac{1}{n_b}\left( e-\lambda_b\,\nu(t)\right);\qquad  \lambda_b=e\cdot \nu-\sqrt{n_b^2-1+(e\cdot \nu)^2};\\
m_r(t)&=\dfrac{1}{n_r}\left( e-\lambda_r\,\nu(t)\right);\qquad  \lambda_r=e\cdot \nu-\sqrt{n_r^2-1+(e\cdot \nu)^2}.
\end{align*}
So
\[
m_b(t)-m_r(t)
=
\left(\dfrac{1}{n_b}- \dfrac{1}{n_r}\right)\,e
+
\left(\dfrac{\lambda_r}{n_r}- \dfrac{\lambda_b}{n_b}\right)\,\nu(t):=A+B.
\]
Notice $|A|= \dfrac{|n_b-n_r|}{n_b\,n_r}$.
Next write
\begin{align*}
\left(\dfrac{\lambda_r}{n_r}- \dfrac{\lambda_b}{n_b}\right)
&=
\dfrac{1}{n_b\,n_r}\,\left(n_b\,\lambda_r-n_r\,\lambda_b \right)\\
&=
\dfrac{1}{n_b\,n_r}\,
\left\{\left(n_b-n_r \right)\,(e\cdot \nu(t)) 
+
n_r\,\sqrt{n_b^2-1+(e\cdot \nu)^2}-n_b\,\sqrt{n_r^2-1+(e\cdot \nu)^2}
\right\}.
\end{align*}
Now 
\[
n_r\,\sqrt{n_b^2-1+(e\cdot \nu)^2}-n_b\,\sqrt{n_r^2-1+(e\cdot \nu)^2}
=
\dfrac{(n_b^2-n_r^2)\,(1-(e\cdot \nu)^2)}{n_r\,\sqrt{n_b^2-1+(e\cdot \nu)^2}+n_b\,\sqrt{n_r^2-1+(e\cdot \nu)^2}}.
\]
Hence
\begin{align*}
|B|&\leq \dfrac{|n_b-n_r|}{n_b\,n_r}
+
\dfrac{1}{n_b\,n_r}\,\left( \dfrac{|n_b^2-n_r^2|}{n_r\,\sqrt{n_b^2-1}+n_b\,\sqrt{n_r^2-1}}\right)
\end{align*}
and therefore
\begin{equation}\label{eq:estimate of mr-mb}
\left|m_b(t)-m_r(t) \right|
\leq \dfrac{\left|n_b-n_r \right|}{n_b\,n_r}\,\left(2+\dfrac{n_b+n_r}{n_r\,\sqrt{n_b^2-1}+n_b\,\sqrt{n_r^2-1}} \right).
\end{equation}

We next estimate $d_r(t)-d_b(t)$, where
\begin{align*}
d_r(t)=\dfrac{C_r}{n_r-m_r(t)\cdot e}, \qquad d_b(t)=\dfrac{C_b}{n_b-m_b(t)\cdot e},
\end{align*}
by \eqref{eq:formula for dbt}.
From Lemma \ref{lem:fixed point}, since $f_r(t_0)=f_b(t_0)$, $\nu(t_0)=e=(0,0,1)$. Then by the Snell law $m_r(t_0)=m_b(t_0)=e$, and from the parametrization of $f_r$ and $f_b$, $d_r(t_0)=d_b(t_0):=d_0$. Hence
\[
\dfrac{C_b}{n_b-1}=\dfrac{C_r }{n_r-1}=d_0.
\]
We then obtain 
\begin{align*}
&d_r(t)-d_b(t)\\
&=\dfrac{d_0}{(n_r-m_r(t)\cdot e)(n_b-m_b(t)\cdot e)}\left((n_b-m_b(t)\cdot e)(n_r-1)-(n_r-m_r(t)\cdot e)(n_b-1)\right)\\
&=\dfrac{d_0}{(n_r-m_r(t)\cdot e)(n_b-m_b(t)\cdot e)}\,\Delta.
\end{align*}
Now
\begin{align*}
\Delta&=
n_r-n_b+\(m_b(t)-m_r(t)\)\cdot e+n_b\,m_r(t)\cdot e-n_r\,m_b(t)\cdot e\\
&=
n_r-n_b+\(m_b(t)-m_r(t)\)\cdot e+(n_b-n_r)\,m_r(t)\cdot e+n_r\,\(m_r(t)-m_b(t)\)\cdot e,
\end{align*}
so from \eqref{eq:estimate of mr-mb}
\[
|\Delta|\leq  C(n_r,n_b)\,|n_r-n_b|.
\]
Since $(n_r-m_r(t)\cdot e)(n_b-m_b(t)\cdot e)\geq (n_r-1)(n_b-1)$, we obtain
\begin{equation}\label{eq:estimate difference dr minus db}
|d_r(t)-d_b(t)|\leq C' \,|n_r-n_b|,
\end{equation}
with $C'$ depending on $d_0,n_r,n_b$.

Finally write
\[
f_r(t)-f_b(t)=d_r(t)m_r(t)-d_b(t)m_b(t)=d_r(t)\(m_r(t)-m_b(t)\)-\(d_b(t)-d_r(t)\)m_b(t).
\]
Since $d_r(t)\leq C_r/(n_r-1)=d_0$, then the desired estimate follows from \eqref{eq:estimate of mr-mb} and \eqref{eq:estimate difference dr minus db}.
\end{proof}

We conclude this section analyzing the intersection of the upper surfaces of the single lenses $(u,f_r)$ and $(u,f_b)$.
\begin{proposition}\label{lm:grad u injective implies no lens exist}
Let $w=e$. If $\nu_u(t)\neq e$ for all $t\in \Omega$, and $\nabla u$ is injective in $\Omega$, then $S_r\cap S_b=\emptyset$, where $S_r=f_r(\Omega)$ and $S_b=f_b(\Omega)$.
This means that, the upper surfaces of the single lenses $(u,f_r)$ and $(u,f_b)$ are disjoint.
\end{proposition}
\begin{proof}
Suppose $P\in S_r\cap S_b$, i.e. there exists $t_0, t_1\in \Omega$ such that $f_r(t_0)=f_b(t_1)$, then as in the proof of \eqref{eq:equality of normals} we get $\nu_u(t_0)=\nu_u(t_1)$, and therefore $\nabla u(t_0)=\nabla u(t_1)$.
Since $\nabla u$ is injective, then $t_0=t_1$ and so $f_r(t_0)=f_b(t_0)$. Therefore from the proof of Lemma \ref{lem:fixed point} we conclude that $\nu_u(t_0)=e$.
\end{proof}

\begin{remark}\label{rmk:parallel planes}\rm
When $\nabla u$ is not injective the upper surfaces of the single lenses $(u,f_r)$ and $(u,f_b)$ may or may not be disjoint.
We illustrate this with lenses bounded by parallel planes.
In fact, by Lemma \ref{lm:lemma equality of three normals}
such a lens refracts all rays blue and red into the vertical direction $e$.
Notice that if the planes are sufficiently far apart, depending on the refractive indices, then $f_r(\Omega)\cap f_b(\Omega)=\emptyset$. This is illustrated in Figure \ref{fig:parallel-planes}: if the lens in between the planes $A$ and $B$, then $f_r(\Omega)\cap f_b(\Omega)\neq\emptyset$; and if the lens is between the planes $A$ and $C$, then $f_r(\Omega)\cap f_b(\Omega)=\emptyset$.
\begin{figure}
\includegraphics[width=3in]{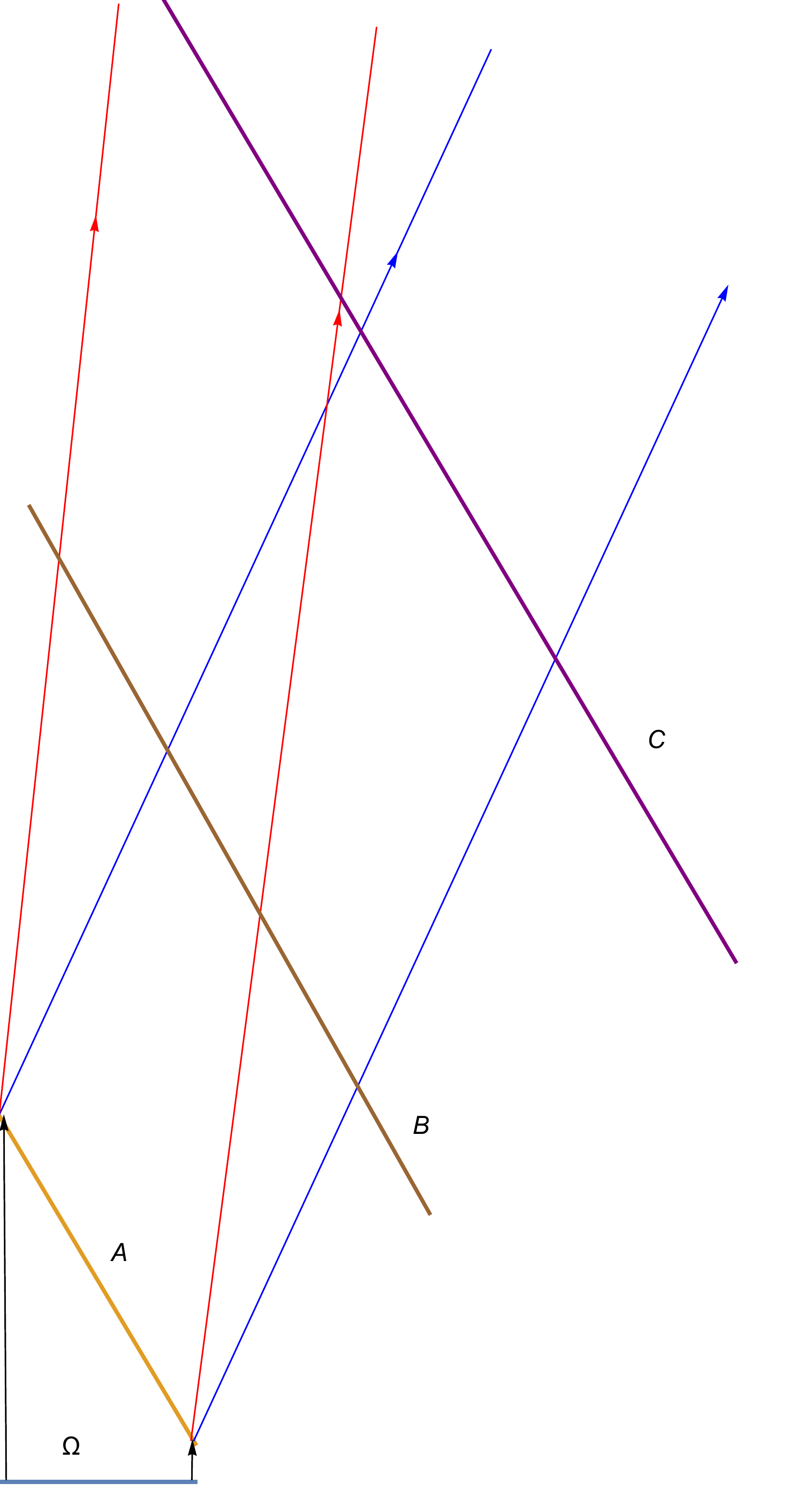}
\caption{The lens is between $A$ and $B$ or between $A$ and $C$.}
\label{fig:parallel-planes}
\end{figure}

\end{remark}

\section{First order functional differential equations}\label{sec:FOD}
\setcounter{equation}{0}

In this section we give a new and simpler proof of an existence theorem for functional differential equations 
originally due to J. Rogers \cite[Sec. 2]{rogers1988:picardtypethemlensdesign}. It will be used in Section \ref{sec:one point source problem} to show the existence of a dichromatic lens when rays emanate from one point source.
Developing an extension of Picard's iteration method for functional equations, Van-Brunt and Ockendon gave another proof of that theorem, \cite{vanbrunt-ockendon:lenstwowavelengths}.
We present a topological proof that we believe has independent interest and uses the Banach fixed point theorem.

Let $H$ be a continuous map defined in an open domain in $\R^{4n+1}$ with values in $\R^n$ given by
\begin{equation}\label{eq:definition map H}
H:=H(X)=\(h_1(X),h_2(X),\cdots, h_{n}(X)\),
\end{equation}
with $X:=\(t;\zeta^0,\zeta^1;\xi^0,\xi^1\),  t\in \R; \zeta^0,\zeta^1,\xi^0,\xi^1\in \R^n.$\\


We are interested in solving the following system of functional differential equations
\begin{align}
Z'(t)&=H\(t;Z(t),Z(z_1(t));Z'(t),Z'(z_1(t))\)\label{eq:System}\\
Z(0)&=0,\notag
\end{align}
with $Z(t)=(z_1(t),\cdots,z_n(t))$.

\begin{theorem}\label{thm: Existence}
Let $\|\cdot\|$ be a norm in $\R^n$.
Assume that the system
\begin{equation}\label{eq:system P}
P=H\(0;{\bf 0,0};P,P\),
\end{equation} 
has a solution $P=\(p_1,p_2,\cdots,p_n\)$ such that
\begin{equation}\label{eq:First Component}
|p_1|\leq 1.
\end{equation}

Let $\mathcal P=\left(0;{\bf 0,0};P,P\right)\in R^{4n+1}$, and let
$$
N_{\varepsilon}(\mathcal P)=\left\{\(t; \zeta^0,\zeta^1;\xi^0,\xi^1\): |t|+ \|\zeta^0\|+\|\zeta^1\|+\|\xi^0-P\|+\|\xi^1-P\|\leq\varepsilon\right\}
$$
be a neighborhood of $\mathcal P$ such that 
\begin{enumerate}[(i)]
\item $H$ is uniformly Lipschitz in the variable $t$, i.e., there exists $\Lambda>0$ such that
\begin{equation}\label{eq:Lip in t}
\left\|H\(\bar t;{ \zeta^0, \zeta^1;  \xi^0,  \xi^1}\)-H\( t;{\zeta^0, \zeta^1;  \xi^0, \xi^1}\)\right\|\leq \Lambda\,|\bar t-t|.
\end{equation}
for all $\(\bar t;{ \zeta^0, \zeta^1;  \xi^0,  \xi^1}\), \( t;{ \zeta^0, \zeta^1; \xi^0,  \xi^1}\)\in  N_{\varepsilon}(\mathcal P)$;
\item $H$ is uniformly Lipschitz in the variables $\zeta^0$ and $\zeta^1$, i.e., there exist positive constants $L_0$ and $L_1$ such that
\begin{equation}\label{eq:Lip in zeta0,zeta1}
\left\|H\(t;{ \bar\zeta^0, \bar \zeta^1;  \xi^0,  \xi^1}\)-H\(t;{ \zeta^0, \zeta^1;  \xi^0, \xi^1}\)\right\|\leq L_0\left\|\bar \zeta^0-\zeta^0\right\|+L_1\left\|\bar\zeta^1-\zeta^1\right\|,
\end{equation}
for all $\(t;{ \bar\zeta^0, \bar\zeta^1;  \xi^0,  \xi^1}\), \(t;{ \zeta^0, \zeta^1; \xi^0,  \xi^1}\)\in  N_{\varepsilon}(\mathcal P)$;
\item $H$ is a uniform contraction in the variables $\xi^0$ and $\xi^1$, i.e., there exists constants $C_0$ and $C_1$ such that
\begin{equation}\label{eq:Lip in xi0,xi1}
\left\|H\(t;{ \zeta^0, \zeta^1;  \bar \xi^0, \bar \xi^1}\)-H\( t;{ \zeta^0, \zeta^1;  \xi^0, \xi^1}\)\right\|\leq C_0\left\|\bar \xi^0-\xi^0\right\|+C_1\left\|\bar\xi^1-\xi^1\right\|,
\end{equation}
for all $\(t;{ \zeta^0, \zeta^1;  \bar\xi^0,  \bar\xi^1}\), \(t;{ \zeta^0, \zeta^1; \xi^0,  \xi^1}\)\in  N_{\varepsilon}(\mathcal P)$,
with
\begin{equation}\label{eq:Contraction}
C_0+C_1<1;
\end{equation}
\item For all $X\in N_{\varepsilon}(\mathcal P)$
\begin{equation}\label{eq:bd on h1}
|h_1(X)|\leq 1.
\end{equation}
\end{enumerate}
Under these assumptions, there exists $\delta>0$ and $Z\in C^1[-\delta,\delta]$ with 
$$\(t;Z(t),Z(z_1(t));Z'(t),Z'(z_1(t))\)\in N_{\varepsilon}(\mathcal P)$$ 
and $Z$ 
solving the system
\begin{equation}\label{eq:functional equations}
\begin{cases}
Z'(t)=H\(t;Z(t),Z(z_1(t));Z'(t),Z'(z_1(t))\)\\
Z(0)=0,
\end{cases}
\end{equation}
for $|t|\leq \delta$ and satisfying in addition that $Z'(0)=P$.

\end{theorem}
%
%
\begin{proof}
Since $H$ is continuous, let 
\begin{equation}\label{eq:Bound on H}
\alpha=\max\{\left\|H(X)\right\|:X\in N_{\varepsilon}(\mathcal P)\}.
\end{equation}
From \eqref{eq:system P}
\begin{equation}\label{eq:bound on P}
\|P\|=\|H(\mathcal P)\|\leq \alpha.
\end{equation}
Let $\mu$ be a constant such that
\begin{equation}\label{eq:lower bound on mu}
\mu\geq\dfrac{\Lambda+(L_0+L_1)\,\alpha}{1-C_0-C_1}.
\end{equation}


For any map $Z:\R\to\R^n$, we define the vector
$$V_Z(t)=\(t;Z(t), Z\(z_1(t)\); \(Z\)'(t),\(Z\)'\(z_1(t)\)\).$$

\begin{definition}\label{def:set C(delta)}
Let $C^1[-\delta,\delta]$ denote the class of all functions $Z:[-\delta,\delta]\to \R^n$ that are $C^1$ equipped with the norm $\|Z\|_{C^1[-\delta,\delta]}=\max_{[-\delta,\delta]}\|Z(t)\|+\max_{[-\delta,\delta]}\|Z'(t)\|$.
We define the set $\mathcal C=\mathcal C(\delta)$ as follows:
$Z\in \mathcal C$ if and only if 
\begin{enumerate}
\item $Z\in C^{1}[-\delta,\delta]$,
\item $Z(0)=0$, $Z'(0)=P,$
\item $|z_1(t)|\leq |t|,$
\item $\|Z(t)-Z(\bar t)\|\leq \alpha\, |t-\bar t|,$
\item $|z_1(t)-z_1(\bar t)|\leq |t-\bar t|,$
\item $\|Z'(t)-Z'(\bar t)\|\leq \mu\, |t-\bar t|$, 
\item $V_Z(t)\in N_{\varepsilon}(\mathcal P)$, for all $|t|\leq \delta$.
\end{enumerate}
\end{definition}
Define a map $T$ on $\mathcal C$ as follows:
$$
T\,Z(t)=\int_0^t H(s;Z(s),Z(z_1(s));Z'(s),Z'(z_1(s)))\,ds.
$$
Our goal is to show that $T:\mathcal C\to \mathcal C$, for $\delta$ sufficiently small, and therefore  
from Banach's fixed point theorem, $T$ has a unique fixed point $Z\in \mathcal C$ and so $Z$ solves \eqref{eq:System}.

We will prove the theorem in a series of steps.

\begin{claim}\label{clm:Non empty}
There exists $\delta_0>0$ such that $\mathcal C(\delta)$ is non empty for $\delta\leq \delta_0$; in fact, the function $Z^0(t)=t P\in \mathcal C$.
\end{claim}

\begin{proof}
Obviously, $Z^0(0)=0$, $\(Z^0\)'=P$, and from  \eqref{eq:First Component} 
$\left|z_1^0(t)\right|=|p_1|\,|t|\leq |t|.$
Also from \eqref{eq:system P} and \eqref{eq:First Component}
$$
\|Z^0(t)-Z^0(\bar t)\|\leq\|P\|\, |t-\bar t|= \|H(\mathcal P)\|\, |t-\bar t|\leq \alpha \,|t-\bar t|,
$$ 
$$|z_1^0(t)-z_1^0(\bar t)|=|p_1|\,|t-\bar t|\leq |t-\bar t|,$$
and
$
\|\(Z^0\)'(t)-\(Z^0\)'(\bar t)\|=0\leq \mu\,|t-\bar t|.
$

It remains to show that $V_{Z^0}(t)\in  N_{\varepsilon}(\mathcal P)$.
By definition 
$
V_{Z^0}(t)=\( t;t\,P,t\,p_1\,P;P,P\),
$
and so $V_{Z^0}(t)\in N_\varepsilon \(\mathcal P\)$ if and only if
$
|t|+|t|\,\|P\|+|t|\,|p_1|\,\|P\|\leq \varepsilon,
$
which is equivalent to 
\begin{equation}\label{eq:definition of delta zero}
|t|\leq \dfrac{\varepsilon}{1+\|P\|+|p_1|\,\|P\|}:=\delta_0.
\end{equation}
%
%
%
%
%
\end{proof}

\begin{claim}\label{clm:Completeness}
$\mathcal C(\delta)$ is complete for every $\delta>0$.
\end{claim}
\begin{proof}
Let $Z^k$ be a Cauchy sequence in $\mathcal C$. Since $C^1[-\delta,\delta]$ is complete, $Z^k$ converges uniformly to a function $Z\in C^{1}[-\delta,\delta]$, and $\(Z^{k}\)'$ converges uniformly to $Z'$. Since $Z^k$ satisfy properties (1)-(7) in Definition \ref{def:set C(delta)}, then Step \ref{clm:Completeness} follows by uniform convergence.
\end{proof}

\begin{claim}\label{clm:First Component map}
If $Z\in \mathcal C$ and $W=TZ=(w_1,\cdots ,w_n)$, then
$$|w_1(t)|\leq |t|.$$
\end{claim}
\begin{proof}
From \eqref{eq:bd on h1} 
and since $V_Z(t)\in N_{\varepsilon}(\mathcal P)$ for $|t|\leq \delta$, it follows that  
$$|w_1(t)|=\left |\int_0^t h_1(V_Z(s))\, ds\right| \leq |t|.$$
\end{proof}

\begin{claim}\label{clm:Lip Continuity}
If $Z\in \mathcal C$ and $W=TZ$, then
$$\left\|W(t)-W(\bar t)\right\|\leq \alpha |t-\bar t|,$$
and
$$|w_1(t)-w_1(\bar t)|\leq |t-\bar t|$$
for every $t,\bar t\in[-\delta,\delta].$
\end{claim}

\begin{proof}
Since $V_{Z}(t)\in N_\varepsilon(\mathcal P)$, by \eqref{eq:Bound on H} for all $|t|\leq \delta$ 
\[
\left\|W(t)-W(\bar t)\right\|=\left\|\int_{\bar t}^{t} H(V_Z(s))\,ds\right\|\leq \alpha \,|t-\bar t|,
\]
and from \eqref{eq:bd on h1}, we get similarly the desired estimate for $w_1(t)-w_1(\bar t)$.
\end{proof}

\begin{claim}\label{clm:Lip Derivative for W}
If $Z\in \mathcal C(\delta)$ and $W=TZ$, then
$$\left\|W'(t)-W'(\bar t)\right\|\leq \mu |t-\bar t|,$$
for every $t,\bar t\in[-\delta,\delta]$.
\end{claim}

\begin{proof}
From the Lipschitz estimates for $H$
\begin{align*}
\left\|W'(t)-W'(\bar t)\right\|&=\left\|H(V_{Z}(t))-H(V_Z(\bar t))\right\|\\
&\leq \Lambda\, |t-\bar t|+L_0\,\|Z(t)-Z(\bar t)\|+L_1\,\|Z(z_1(t))-Z(z_1(\bar t))\|\\
&\qquad +C_0\,\|Z'(t)-Z'(\bar t)\|+C_1\,\|Z'(z_1(t))-Z'(z_1(\bar t))\|.
\end{align*}
Since $|z_1(t)|\leq |t|\leq \delta$ and $|z_1(\bar t)|\leq |\bar t|\leq \delta$, we get from the Lipschitz properties of $Z$, $z_1$, and $Z'$ 
\begin{align*}
\left\|W'(t)-W'(\bar t)\right\|&\leq \Lambda\,|t-\bar t|+L_0\,\alpha |t-\bar t|+L_1\,\alpha |z_1(t)-z_1(\bar t)|+C_0\,\mu |t-\bar t|+C_1\,\mu |z_1(t)-z_1(\bar t)|\\
&\leq \(\Lambda+(L_0+L_1)\alpha+(C_0+C_1)\mu\)\,|t-\bar t|.
\end{align*}
From \eqref{eq:lower bound on mu},
$\Lambda+(L_0+L_1)\alpha\leq  (1-C_0-C_1)\mu,$
then Step \ref{clm:Lip Derivative for W} follows.
\end{proof}

\begin{claim}\label{clm:automorphism}
For $\delta$ sufficiently small,
$W=TZ\in \mathcal C$ for each $Z\in \mathcal C$.
\end{claim}

\begin{proof}
From the previous steps, it remains to show that $V_{W}(t)\in N_{\varepsilon}(\mathcal P)$. Define
\begin{align*}
S_{Z}(t)&=|t|+\left\|Z(t)\right\|+\left\|Z(z_1(t))\right\|+\left\|Z'(t)-P\right\|+\left\|Z'(z_1(t))-P\right\|\\
S_{W}(t)&=|t|+\left\|W(t)\right\|+\left\|W(w_1(t))\right\|+\left\|W'(t)-P\right\|+\left\|W'(w_1(t))-P\right\|.
\end{align*}
Since $V_{Z}(t)\in N_{\varepsilon}(\mathcal P)$, we have $S_{Z}(t)\leq \varepsilon$. We shall prove that $S_{W}(t)\leq \varepsilon$ by choosing $\delta$ sufficiently small.
In fact, from \eqref{eq:Bound on H} for every $|t|\leq \delta$
\begin{equation}\label{eq:bound on W}
\left\|W(t)\right\|= \left\|\int_{0}^t H(V_Z(s))\,ds\right\|\leq \alpha \,|t|\leq \alpha\, \delta,
\end{equation}
and from \eqref{eq:bd on h1} 
$$
\left|w_1(t)\right|= \left|\int_0^t h_1(V_Z(s))\,ds\right|\leq |t|\leq \delta.
$$
Hence
$
\left\|W(w_1(t))\right\|\leq \alpha\, \delta.
$
We next estimate $\left\|W'(t)-P\right\|$. Using the Lipschitz properties of $H$ we write
\begin{align*}
\left\|W'(t)-P\right\|&=\left\|H(V_{Z}(t))-H(\mathcal P)\right\|\\
&\leq \Lambda|t|+L_0\left\|Z(t)\right\|+L_1\left\|Z(z_1(t))\right\|+C_0\left\|Z'(t)-P\right\|+C_1\left\|Z'(z_1(t))-P\right\|.
\end{align*}
Notice that $\left\|Z(t)\right\|=\left\|Z(t)-Z(0)\right\|\leq \alpha |t|\leq \alpha \delta$,
and since $|z_1(t)|\leq |t|\leq \delta$ then 
$\left\|Z(z_1(t))\right\|\leq \alpha \delta$.
Also
$$\left\|Z'(t)-P\right\|=\left\|Z'(t)-Z'(0)\right\|\leq \mu|t|\leq \mu\delta .$$
and 
$$\left\|Z'(z_1(t))-P\right\|\leq \mu\delta .$$
Therefore, for $|t|\leq \delta$
$$\left\|W'(t)-P\right\|\leq \left[\Lambda+\alpha(L_0+L_1)\delta+\mu(C_0+C_1)\right]\delta,$$
and since $|w_1(t)|\leq \delta $, we also get 
$$\left\|W'(w_1(t))-P\right\|\leq \left[\Lambda+\alpha(L_0+L_1)\delta+\mu(C_0+C_1)\right]\delta.$$
We conclude that
$$S_W(t)\leq \delta\,\(1+2\Lambda+2\alpha(1+L_0+L_1)+2\mu(C_0+C_1)\),
$$
so choosing $\delta\leq \dfrac{\varepsilon}{1+2\Lambda+2\alpha(1+L_0+L_1)+2\mu(C_0+C_1)}$ Step \ref{clm:automorphism} follows.
\end{proof}

It remains to show that $T$ is a contraction.
\begin{claim}\label{stp:T is contraction}
If $Z^1,Z^2\in \mathcal C(\delta)$ with $\delta$ small enough from the previous steps, then
$$\left\|TZ^1-TZ^2\right\|_{C^1[-\delta,\delta]}\leq q \,\left\|Z^1-Z^2\right\|_{C^1[-\delta,\delta]},$$
for some $q<1$.
\end{claim}

\begin{proof}
Let $W^1(t)=T Z^1(t),$ $W^2(t)=TZ^2(t)$. By the fundamental theorem of calculus we have for every $|t|\leq\delta$
\begin{align*}
\left\|W^1(t)-W^2(t)\right\| &\leq \left |\int_0^t \left\|\(W^1\)'(s)-\(W^2\)'(s)\right\|\,ds \right| \leq \delta \sup_{|t|\leq \delta} \left\|\(W^1\)'(t)-\(W^2\)'(t)\right\|\\
& \leq \delta \|W^1-W^2\|_{C^1[-\delta,\delta]},
\end{align*}
and similarly $\left\|Z^1(t)-Z^2(t)\right\| \leq \delta\, \|Z^1-Z^2\|_{C^1[-\delta,\delta]}$.
From the Lipschitz properties of $H$, for every $|t|\leq \delta$,
{
\begin{align*}
\left\|\(W^1\)'(t)-\(W^2\)'(t)\right\|&\leq \left\|H\(V_{Z^1}(t)\)-H\(V_{Z^2}(t)\)\right\|\\
&\leq L_0 \left\|Z^1(t)-Z^2(t)\right\|+L_1\left\|Z^1\(z_1^1(t)\)-Z^2\(z_1^2(t)\)\right\|\\
&\quad+C_0\left\|\(Z^1\)'(t)-\(Z^2\)'(t)\right\|+C_1\left\|\(Z^1\)'\(z_1^1(t)\)-\(Z^2\)'\(z_1^2(t)\)\right\|.
\end{align*}
}
We have
\begin{align*}
\left\|Z^1(t)-Z^2(t)\right\|
&\leq \delta\,\|Z^1-Z^2\|_{C^1[-\delta,\delta]}\\
\left\|Z^1(z_1^1(t))-Z^2(z_1^2(t))\right\|
&\leq \left\|Z^1(z_1^1(t))-Z^2(z_1^1(t))\right\|+\left\|Z^2(z_1^1(t))-Z^2(z_1^2(t))\right\|\\
&\leq \delta\,\|Z^1-Z^2\|_{C^1[-\delta,\delta]}+\alpha\, |z_1^1(t)-z_1^2(t)|\\
&\leq \delta\,\|Z^1-Z^2\|_{C^1[-\delta,\delta]}+\alpha C_{\|\cdot\|} \|Z^1(t)-Z^2(t)\|\\
&\leq \delta\,(\alpha C_{\|\cdot\|}+1)\,\|Z^1-Z^2\|_{C^1[\delta,\delta]}\\
\left\|\(Z^1\)'(t)-\(Z^2\)'(t)\right\|&\leq \|Z^1-Z^2\|_{C^1[-\delta,\delta]}\\
\left\|\(Z^1\)'\(z_1^1(t)\)-\(Z^2\)'\(z_1^2(t)\)\right\|&\leq \left\|\(Z^1\)'\(z_1^1(t)\)-\(Z^2\)'\(z_1^1(t)\)\right\|+\left\|\(Z^2\)'\(z_1^1(t)\)-\(Z^2\)'\(z_1^2(t)\)\right\|\\
&\leq \left\|Z^1-Z^2\right\|_{C^1[-\delta,\delta]}+\mu |z_1^1(t)-z_1^2(t)|\\
&\leq \left\|Z^1-Z^2\right\|_{C^1[-\delta,\delta]}+\mu C_{\|\cdot\|} \left\|Z^1(t)-Z^2(t)\right\|\\
&\leq \(\mu \, C_{\|\cdot\|} \,\delta+1\)\left\|Z^1-Z^2\right\|_{C^1[-\delta,\delta]};
\end{align*}
here $C_{\|\cdot\|}$ is a constant larger than $1$, depending only on the choice of the norm in $\R^n$, since all norms in $\R^n$ are equivalent such constant exists.
Combining the above inequalities, we obtain
\begin{align*}
&\left\|\(W^1\)'(t)-\(W^2\)'(t)\right\|\\
&\leq \left(L_0\,\delta+L_1\,\delta\,(\alpha C_{\|\cdot\|}+1)+C_0+C_1\(\mu C_{\|\cdot\|}\,\delta+1\)\right) \,\left\|Z^1-Z^1\right\|_{C^1}:=\(M\,\delta+C_0+C_1\)\left\|Z^1-Z^2\right\|_{C^1},
\end{align*}
and from the fundamental theorem of calculus
\[
\left\|W^1(t)-W^2(t)\right\|\leq \delta  \sup_{|t|\leq \delta} \left\|\(W^1\)'(t)-\(W^2\)'(t)\right\|\leq \delta \,\(M\,\delta+C_0+C_1\)\left\|Z^1-Z^2\right\|_{C^1}.
\]
We conclude that 
$$
\left\|W^1-W^2\right\|_{C^1[-\delta,\delta]}\leq \(1+\delta\)\,\left(M\,\delta+C_0+C_1\right)\,\left\|Z^1-Z^2\right\|_{C^1[-\delta,\delta]}
$$
Since $C_0+C_1<1$, then choosing $\delta$ sufficiently small Step \ref{stp:T is contraction} follows.
\end{proof}
We conclude that there exists $\delta^*>0$ small such that for $0<\delta\leq \delta^*$, the map $T:\mathcal C(\delta)\to \mathcal C(\delta)$ is a contraction and hence by the Banach fixed point theorem there is a unique $Z\in \mathcal C(\delta)$ such that 
$$
Z(t)=TZ(t)=\int_0^t H(V_Z(s))\,ds.
$$
Differentiating with respect to $t$, we get that $Z$ solves \eqref{eq:functional equations} for $|t|\leq \delta$.
\end{proof}

We make the following observations about the assumptions in Theorem \ref{thm: Existence}.
\begin{remark}\label{rmk:non existence}\rm
We show that even for $H$ smooth, satisfying \eqref{eq:system P} with \eqref{eq:First Component}, and \eqref{eq:bd on h1}, the system \eqref{eq:System} might not have any real solutions in a neighborhood of $t=0$.
%
%
In fact, consider for example the following ode:
\begin{equation}\label{eq:Counterexample System}
z'(t)= z'(t)^2+\dfrac{1-t}{4},\qquad z(0)=0.
\end{equation}
In this case $n=1$ and $H:\R^5\to \R$ with
$$H\(t;\zeta^0,\zeta^1;\xi^0,\xi^1\)=\(\xi^0\)^2+\dfrac{1-t}{4}$$
analytic.
The system $P=H(0;0,0;P,P)$ has a unique solution $P=1/2$, and so \eqref{eq:system P} and \eqref{eq:First Component} hold. 
Let $\mathcal P=(0;0,0;1/2,1/2)$. Since $H(\mathcal P)<1$, \eqref{eq:bd on h1} holds in a small neighborhood of $\mathcal P$.
On the other hand, \eqref{eq:Counterexample System} cannot have real solutions for $t<0$ and so in any neighborhood of $t=0$.
This shows that there cannot exist a norm $\|\cdot \|$ in $\R$ so that the contraction condition \eqref{eq:Contraction} is satisfied.
In particular, this shows that the conclusion of \cite[Lemma 2.2]{vanbrunt-ockendon:lenstwowavelengths} is in error.
\end{remark}

\begin{remark}\label{rmk:More C0+C1}\rm Let $H$ be a map from a domain in $\R^{4n+1}$ with values in  $\R^n$, and let $P$ be a solution to the system $P=H(0;{\bf 0,0};P,P)$ satisfying \eqref{eq:First Component}. Assume $H$ is $C^1$ in a neighborhood of $\mathcal P=(0;{\bf 0,0};P,P)$. Given a norm $\|\cdot\|$ in $\R^n$, let $\left\|\left|\cdot\right|\right\|$ be the induced norm on the space of $n\times n$ matrices, i.e., for a $n\times n$ matrix $A$ 
$$\left\|\left|A\right|\right\|=\max\{\|Av\|:v\in \R^n, \|v\|=1\},$$
see \cite[Section 5.6]{horn-johnson-matrix-analysis}.
Since $H$ is $C^1$ then there exists a neighborhood $N_{\varepsilon}(\mathcal P)$ as defined in Theorem \ref{thm: Existence} such that \eqref{eq:Lip in t}, \eqref{eq:Lip in zeta0,zeta1}, \eqref{eq:Lip in xi0,xi1} are satisfied.

The following proposition shows estimates for $C_0+C_1$ in \eqref{eq:Lip in xi0,xi1}. 
\begin{proposition}\label{eq:prop Bound on C0+C1}
We define the $n\times n$ matrices $$\nabla_{\xi^0}H=\left(\dfrac{\partial h_i}{\partial \xi^0_j}\right)_{1\leq i,j\leq n},\quad \nabla_{\xi^1}H=\left(\dfrac{\partial h_i}{\partial \xi^1_j}\right)_{1\leq i,j\leq n}.$$ 
If \eqref{eq:Lip in xi0,xi1} holds for some $C_0, C_1$, then 
\begin{equation}\label{eq:lower bound on C0+C1}
C_0\geq \left\|\left|\nabla_{\xi^0}H(\mathcal P)\right|\right\|,\text{ and } C_1\geq\left\|\left|\nabla_{\xi^0}H(\mathcal P)\right|\right\|.
\end{equation}
In addition, \eqref{eq:Lip in xi0,xi1} holds with $C_0=\max_{N_{\varepsilon}(\mathcal P)}\left|\left\|\nabla_{\xi^0} H\right\|\right|$ and $C_1=\max_{N_{\varepsilon}(\mathcal P)}\left|\left\|\nabla_{\xi^1} H\right\|\right|$.
\end{proposition}

\begin{proof}
We first prove \eqref{eq:lower bound on C0+C1}. 
Let $v\in \R^n$ with $\|v\|=1$; for $s>0$ small, the vector $(0;{\bf 0,0};P+s\,v,P)\in N_{\varepsilon}(\mathcal P)$. From \eqref{eq:Lip in xi0,xi1}, we get
$$
\left\|H(0;{\bf 0,0};P+s\,v,P)-H(0;{\bf 0,0};P,P)\right\|\leq C_0\, |s|.
$$
Dividing by $|s|$ and letting $s\to 0$ we obtain, by application of the mean value theorem on each component, that for every $v\in \R^n$ with $\|v\|=1$
$$C_0\geq \left\| \nabla_{\xi^0} H(\mathcal P)\, v^t\right\|.$$
Taking the supremum over all $v\in \R^n$ we obtain
$C_0\geq \left\|\left|\nabla_{\xi^0}H(\mathcal P)\right|\right\|.$
Similarly we get $C_1\geq \left\|\left|\nabla_{\xi^1}H(\mathcal P)\right|\right\|$.

To show the second part of the proposition,
%
%
by the fundamental theorem of calculus,
we have for $\(t;\zeta^0,\zeta^1;\xi^0,\xi^1\),\(t;\zeta^0,\zeta^1;\bar\xi^0,\bar\xi^1\)\in N_{\varepsilon}(\mathcal P)$ that
\begin{align*}
&H\(t;\zeta^0,\zeta^1;\xi^0,\xi^1\)-H\(t;\zeta^0,\zeta^1;\bar\xi^0,\bar\xi^1\)\\
&=\int_{0}^1 DH\left( (1-s) \(t;\zeta^0,\zeta^1;\xi^0,\xi^1\)+s \(t;\zeta^0,\zeta^1;\bar\xi^0,\bar\xi^1\)\right)\, \(0;\bf{0,0};\xi^0-\bar\xi^0,\xi^1-\bar \xi^1\)^t\,ds\\
&= \int_{0}^1 \nabla_{\xi^0}H\left( (1-s) \(t;\zeta^0,\zeta^1;\xi^0,\xi^1\)+s \(t;\zeta^0,\zeta^1;\bar\xi^0,\bar\xi^1\)\right)\, \(\xi^0-\bar \xi^0\)^t\,ds\\
&\qquad + \int_{0}^1 \nabla_{\xi^1}H\left( (1-s) \(t;\zeta^0,\zeta^1;\xi^0,\xi^1\)+s \(t;\zeta^0,\zeta^1;\bar\xi^0,\bar\xi^1\)\right)\, \(\xi^1-\bar \xi^1\)^t\,ds,
\end{align*}
where $DH$ is the $(4n+1)\times n$ matrix of the first derivatives of $H$ with respect to all variables.
Then
\begin{align*}
&\left\|H\(t;\zeta^0,\zeta^1;\xi^0,\xi^1\)-H\(t;\zeta^0,\zeta^1;\bar\xi^0,\bar\xi^1\)\right\|\\
&\leq \int_{0}^{1} \left\|\left|\nabla_{\xi^0}H\left( (1-s) \(t;\zeta^0,\zeta^1;\xi^0,\xi^1\)+s \(t;\zeta^0,\zeta^1;\bar\xi^0,\bar\xi^1\)\right)\right\|\right|\, \left\|\xi^0-\bar \xi^0\right\|\, ds  \\
&\qquad +\int_{0}^{1} \left\|\left|\nabla_{\xi^1}H\left( (1-s) \(t;\zeta^0,\zeta^1;\xi^0,\xi^1\)+s \(t;\zeta^0,\zeta^1;\bar\xi^0,\bar\xi^1\)\right)\right\|\right|\,\left\|\xi^1-\bar \xi^1\right\|\, ds \\
&\leq \max_{N_{\varepsilon}(\mathcal P)}\left|\left\|\nabla_{\xi^0} H\right\|\right|\,\,  \left\|\bar \xi^0-\xi^0\right\|+
\max_{N_{\varepsilon}(\mathcal P)}\left|\left\|\nabla_{\xi^1} H\right\|\right|\,\,  \left\|\bar \xi^1-\xi^1\right\|.
\end{align*}
The proof is then complete.
\end{proof}

Given an $n\times n$ matrix $A$, let $R_A$ be its spectral radius, i.e., $R_A$ is the largest  absolute value of the eigenvalues of $A$. From \cite[Theorem 5.6.9]{horn-johnson-matrix-analysis} we have $R_A\leq \left\|\left|A\right|\right\|$ for any matrix norm $\left\|\left|\cdot \right|\right\|$. Then from \eqref{eq:lower bound on C0+C1} we get the following corollary that shows that the possibility of choosing a norm in $\R^n$ for which the contraction property \eqref{eq:Contraction} holds depends on the spectral radii of the matrices $\nabla_{\xi^0} H(\mathcal P),  \nabla_{\xi^1} H(\mathcal P)$.

\begin{corollary}\label{cor:spectral radii}
Let $H$ be as above. Denote by $R_{\xi^0}$ and  $R_{\xi^1}$ the spectral radii of the matrices  $\nabla_{\xi^0} H(\mathcal P),$ and $\nabla_{\xi^1} H(\mathcal P)$ respectively. Then for any norm in $\R^n$ and any $C_0,C_1$ satisfying \eqref{eq:Lip in xi0,xi1} we have
$$C_0+C_1\geq R_{\xi^0}+R_{\xi^1}.$$
\end{corollary}

To apply Theorem \ref{thm: Existence}, we need $H$ to satisfy \eqref{eq:Lip in xi0,xi1} together with the contraction condition \eqref{eq:Contraction} for some norm $\|\cdot \|$ in $\R^n$.
It might not be possible to find such a norm. In fact, if an $n\times n$ matrix $A$ has spectral radius $R_A\geq 1$, then for each norm $\|\cdot \|$ in $\R^n$ the induced matrix norm satisfies $\||A|\|\geq 1$; see [Theorem 5.6.9 and Lemma 5.6.10]\cite{horn-johnson-matrix-analysis}.
So if the sum of the spectral radii of the Jacobian matrices $\nabla_{\xi^0}H(\mathcal P),$ $\nabla_{\xi^1}H(\mathcal P)$ is bigger than one, then from Corollary \ref{cor:spectral radii} it is not possible to find a norm $\|\cdot \|$ in $\R^n$ for which \eqref{eq:Contraction} holds.



\end{remark}

\subsection{Uniqueness of solutions} In this section we show the following uniqueness theorem.
\begin{theorem}\label{thm:Uniqueness}
Under the assumptions of Theorem \ref{thm: Existence}, the local solution to \eqref{eq:functional equations} with $Z'(0)=P$ is unique.
\end{theorem}

The theorem is a consequence of the following lemma.
\begin{lemma}\label{lm:uniqueness lemma}
Let $\mathcal C(\delta)$ be the set in Definition \ref{def:set C(delta)}.
If there exists $\delta>0$ such that $W$ solves \eqref{eq:functional equations} for $|t|\leq \delta$,  with $W'(0)=P$, and $V_W(t)\in N_{\varepsilon}(\mathcal P)$ for $|t|\leq \delta$, then $W\in \mathcal C(\delta)$. 
\end{lemma}
%

\begin{proof}
Since $w_1'(t)=h_1(V_{W}(t))$, then \eqref{eq:bd on h1} implies
\begin{equation}\label{eq:first component W}
|w_1(t)|=\left|\int_{0}^t h_1(V_W(s))\,ds\right|\leq |t|.
\end{equation}
Also for $|t|,|\bar t|\leq \delta$ since $W$ is a solution to \eqref{eq:functional equations} then from \eqref{eq:Bound on H}
\begin{equation}\label{eq:lip W}
\left\|W(t)-W(\bar t)\right\|=\left\|\int_{\bar t}^t H(V_W(s))\, ds\right\|\leq \alpha |t-\bar t|
\end{equation}
and by \eqref{eq:bd on h1} 
\begin{equation}\label{eq:Lip w1}
\left|w_1(t)-w_1(\bar t)\right|=\left|\int_{\bar t}^t h_1(V_W(s))\,ds\right|\leq |t-\bar t|.
\end{equation}

It remains to show the Lipschitz estimate on $W'$. Let $|t|,|\bar t|\leq \delta$, then by the Lipschitz properties of $H$
\begin{align*}
\left\|W'(t)-W'(\bar t)\right\|&= \left\|H(V_{W}(t))-H(V_{W}(\bar t))\right\|\\
&\leq \Lambda |t-\bar t|+L_0 \left\|W(t)-W(\bar t)\right\|+L_1\left\|W(w_1(t))-W(w_1(\bar t))\right\|\\
&\qquad +C_0 \left\|W'(t)-W'(\bar t)\right\|+C_1\left\|W'(w_1(t))-W'(w_1(\bar t))\right\|.
\end{align*}
Using \eqref{eq:lip W}, \eqref{eq:first component W}, and \eqref{eq:Lip w1}, we get that for every $|t|,|\bar t|\leq \delta$
\begin{equation}\label{eq:Estimate}
\left\|W'(t)-W'(\bar t)\right\|\leq (\Lambda+(L_0+L_1)\alpha)|t-\bar t|+C_0 \left\|W'(t)-W'(\bar t)\right\|+C_1\left\|W'(w_1(t))-W'(w_1(\bar t))\right\|.
\end{equation}
Fix $t$ and $\bar t$ and let $r=|t-\bar t|$. Let $\tau$, and $\bar \tau$ be such that $|\tau|,|\bar \tau|\leq \delta$ and $|\tau-\bar \tau|\leq r$, then by \eqref{eq:Lip w1}
$$\left|w_1(\tau)-w_1(\bar \tau)\right|\leq |\tau-\bar \tau|\leq r.$$
Hence applying \eqref{eq:Estimate} for $\tau$ and $\bar \tau$ we get for $|\tau|,|\bar \tau|\leq \delta$, $|\tau-\bar \tau|\leq r$ that
\begin{align*}
\left\|W'(\tau)-W'(\bar \tau)\right\|
&\leq (\Lambda+\(L_0+L_1\)\alpha)|\tau-\bar \tau|+C_0 \left\|W'(\tau)-W'(\bar \tau)\right\|+C_1\left\|W'(w_1(\tau))-W'(w_1(\bar \tau))\right\|\\
&\leq (\Lambda+\(L_0+L_1\)\alpha)\,r+(C_0+C_1)\sup_{|\tau|,|\bar \tau|\leq \delta, |\tau-\bar \tau|\leq r} \left\|W'(\tau)-W'(\bar \tau)\right\|.
\end{align*}
Hence taking the supremum on the left hand side of the inequality, and using \eqref{eq:lower bound on mu} we get
$$
\sup_{|\tau|,|\bar \tau|\leq \delta, |\tau-\bar \tau|\leq r} \left\|W'(\tau)-W'(\bar \tau)\right\| \leq \dfrac{\Lambda+\(L_0+L_1\)\,\alpha}{1-C_0-C_1} \,r\leq \mu \,r
$$
so for every $|t|,|\bar t|\leq \delta$
$$
\left\|W'(t)-W'(\bar t)\right\|\leq \sup_{|\tau|,|\bar \tau|\leq \delta, |\tau-\bar \tau|\leq |t-\bar t|} \left\|W'(\tau)-W'(\bar \tau)\right\| \leq \mu\, |t-\bar t|,
$$
and the lemma follows.
\end{proof}
\begin{proof}[Proof of Theorem \ref{thm:Uniqueness}]
Let $\delta_1,\delta_2\leq \delta^*$ and let $W^i$ solving \eqref{eq:functional equations} for $|t|\leq \delta_i$, with $\(W^i\)'(0)=P$, and $V_{W^i}(t)\in N_{\varepsilon}(\mathcal P)$ for $|t|\leq \delta_i$ for $i=1,2$. From Lemma \ref{lm:uniqueness lemma} $W^i\in \mathcal C(\delta_i)$, and since they solve \eqref{eq:functional equations} we have $TW^i(t)=W^i(t)$ for $|t|\leq \delta_i$, $i=1,2$.
Let $\delta=\min\{\delta_1,\delta_2\}$. We have $\mathcal C(\delta_i)\subset \mathcal C(\delta)$, $i=1,2$.
Since $\delta\leq \delta^*$, from the proof of the existence theorem $T$ has a unique fixed point in $\mathcal C(\delta)$. But, $TW^i=W^i$ for $|t|\leq \delta$ and so $W^1=W^2$ for $|t|\leq \delta$.
\end{proof}

\begin{remark}\rm
If the vector $P$ solution to the system \eqref{eq:system P}
does not satisfy \eqref{eq:First Component}, then the system \eqref{eq:functional equations} may have infinitely many solutions. This goes back to the paper by Kato and McLeod \cite[Thm. 2]{katomcleod:functionaldifferentialequations} about single functional differential equations.
We refer to \cite[Equation (1.7)]{fox-mayers-ockendon-tayler:funcdiffeqexplicitexample} for representation formulas for infinitely many solutions.
\end{remark}

\begin{remark}\rm\label{rmk:theorem for general m}
Theorem \ref{thm: Existence} has an extension to more variables and the proof is basically the same. 
In fact, the set up in this case and the result are as follows.

We set $Z(t)=\(z_1(t),\cdots ,z_n(t)\)$ where $z_i(t)$ are real valued functions of one variable, and
$Z'(t)=\(z_1'(t),\cdots ,z_n'(t)\)$. 
Let $H=\(h_1,\cdots ,h_n\)$ where 
$$h_i=h_i\(t;\zeta^0,\zeta^1,\cdots ,\zeta^m;\xi^0,\xi^1,\cdots ,\xi^m\)$$ 
are real valued functions with $\zeta^j,\xi^j\in \R^n$ for $1\leq j\leq m$, so each $h_i$ has $1+2\,(m+1)\,n$ variables.
Let 
\[
X=\(t;\zeta^0,\zeta^1,\cdots ,\zeta^m;\xi^0,\xi^1,\cdots ,\xi^m\).
\] 
We assume $1\leq m\leq n$.
Set $Z(z_k(t))=\(z_1(z_k(t)),\cdots ,z_n(z_k(t))\)$ for $1\leq k\leq m$.
We want to find $Z(t)$ as above satisfying the following functional differential equation
\begin{align*}
Z'(t)&=H\(t;Z(t),Z(z_1(t)),\cdots ,Z(z_m(t));Z'(t),Z'(z_1(t)),\cdots ,Z'(z_m(t))\)\\
Z(0)&=(0,\cdots ,0)\in \R^n,
\end{align*}
for $t$ in a neighborhood of $0$.
We assume that there exists $P=(p_1,\cdots ,p_n)\in \R^n$ with $|p_i|\leq 1$ for $1\leq i\leq m$ solving
\[
P=H\(0;{\bf 0,\cdots ,0}; P,\cdots ,P\)
\]
where dots mean $m+1$ times. 

Let $\mathcal P=\left(0;{\bf 0,\cdots,0};P,\cdots,P\right)\in R^{2(m+1)n+1}$, and let $\|\cdot \|$ be a norm in $\R^n$ with
$$
N_{\varepsilon}(\mathcal P)=\left\{\(t;{ \zeta^0,\cdots,\zeta^m;\xi^0,\cdots,\xi^m}\); |t|+ \|\zeta^0\|+\cdots+\|\zeta^m\|+\|\xi^0-P\|+\cdots+\|\xi^m-P\|\leq\varepsilon \right\}
$$
a neighborhood of $\mathcal P$ such that 
\begin{enumerate}[(i)]
\item $H$ is uniformly Lipschitz in the variable $t$, i.e., there exists $\Lambda>0$ such that
\[
\left\|H\(\bar t;{ \zeta^0, \cdots,\zeta^m;  \xi^0,\cdots,  \xi^m}\)-H\( t;{ \zeta^0,\cdots, \zeta^m;  \xi^0, \cdots,\xi^m}\)\right\|\leq \Lambda|\bar t-t|.
\]
for all $\(\bar t;{ \zeta^0,\cdots, \zeta^m;  \xi^0,\cdots,  \xi^m}\), \(t;{ \zeta^0,\cdots, \zeta^m; \xi^0,\cdots,  \xi^m}\)\in  N_{\varepsilon}(\mathcal P)$;
\item $H$ is uniformly Lipschitz in the variables $\zeta^0,\cdots,\zeta^m$ i.e., there exist positive constants $L_0,\cdots,L_m$ such that
\[
\left\|H\(t;{ \bar\zeta^0,\cdots ,\bar \zeta^m;  \xi^0,\cdots , \xi^m}\)-H\(t;{\zeta^0, \cdots,\zeta^m;  \xi^0, \cdots,\xi^m}\)\right\|\leq L_0\left\|\bar \zeta^0-\zeta^0\right\|+\cdots+L_m\left\|\bar\zeta^m-\zeta^m\right\|,
\]
for all $\(t;{ \bar\zeta^0, \cdots,\bar\zeta^m;  \xi^0,\cdots ,  \xi^m}\), \(t;{ \zeta^0,\cdots, \zeta^m; \xi^0,\cdots , \xi^m}\)\in  N_{\varepsilon}(\mathcal P)$;
\item $H$ is a uniform contraction in the variables $\xi^0,\cdots,\xi^m,$ i.e., there exists constants $C_0,\cdots, C_m$ such that
\[
\left\|H\(t;{ \zeta^0,\cdots, \zeta^m;  \bar \xi^0,\cdots, \bar \xi^m}\)-H\( t;{ \zeta^0,\cdots, \zeta^m;  \xi^0, \cdots,\xi^m}\)\right\|\leq C_0\left\|\bar \xi^0-\xi^0\right\|+\cdots+C_m\left\|\bar\xi^m-\xi^m\right\|,
\]
for all $\(t;{ \zeta^0,\cdots, \zeta^m;  \bar\xi^0, \cdots ,\bar\xi^m}\), \(t;{\zeta^0,\cdots, \zeta^m; \xi^0, \cdots, \xi^m}\)\in  N_{\varepsilon}(\mathcal P)$,
with
\[
C_0+\cdots+C_m<1;
\]
\item For all $X\in N_{\varepsilon}(\mathcal P)$
\begin{equation*}
|h_1(X)|\leq 1.
\end{equation*}
\end{enumerate}
Under these assumptions, there exists $\delta>0$, such that the system
\begin{equation*}
\begin{cases}
Z'(t)=H\(t;Z(t),Z(z_1(t)),\cdots,Z(z_m(t));Z'(t),Z'(z_1(t)),\cdots,Z'(z_m(t))\)\\
Z(0)=0,
\end{cases}
\end{equation*}
has a unique solution defined for $|t|\leq \delta$ satisfying $Z'(0)=P$. 
\end{remark}

\section{One point source case: Problem B}\label{sec:one point source problem}
\setcounter{equation}{0}

The setup in this section is the following. We are given a unit vector $w\in \R^3$, and a compact domain $\Omega$ contained in the upper unit sphere $S^2$, such that $\Omega=x(D)$, where $D$ is a convex and compact domain in $\R^2$ with nonempty interior. Here  $x(t)$ are for example spherical coordinates, $t\in D$. Dichromatic rays with colors b and r are now emitted from the origin with unit direction $x(t)$, $t\in D$. 
From the results from \cite[Section 3]{gutierrez:asphericallensdesign} with $n_1=n_3=1$, $n_2=n_r$, and $e_1=w$ we have the following. 
Consider a $C^2$ surface with a given polar parametrization $\rho(t)x(t)$ for $t\in D$, and the surface parametrized by 
$f_r(t)=\rho(t)x(t)+d_r(t)m_r(t)$, with $m_r(t)=\dfrac{1}{n_r}\(x(t)-\lambda_r \nu_\rho(t)\)$, $\lambda_r(t)
=\Phi_{n_r}\(x(t)\cdot \nu_\rho(t)\)$ from \eqref{formulaforlambda}, $\nu_\rho(t)$ the outer unit normal at $\rho(t)x(t)$, and with 
\begin{equation}\label{eq:formula for dbt point source}
d_r(t)=\dfrac{C_r-\rho(t)\(1-w\cdot x(t)\)}{n_r-w\cdot m_r(t)}, 
\end{equation}
for some constant $C_r$. Then the lens bounded between $\rho$ and $f_r$ refracts the rays with color r into the direction $w$ provided that $C_r$ is chosen so that $d_r(t)>0$ and $f_r$ has a normal at each point. 
Likewise and for the color b the surface 
$f_b(t)=\rho(t)x(t)+d_b(t)m_b(t)$, with similar quantities as before with r replaced by b,
does a similar refracting job for rays with color b. 

As before, we assume $n_b>n_r>1$, and the medium surrounding the lens is vacuum. To avoid total reflection for each color, compatibility conditions between $\rho$ and $w$ are needed, see \cite[condition (3.8)]{gutierrez:asphericallensdesign} which in our case reads
\[
\lambda_r\,\nu_\rho(t)\cdot w\leq x(t)\cdot w-1,\text{ and } \lambda_b\,\nu_\rho(t)\cdot w\leq x(t)\cdot w-1.
\]

The problem we consider in this section is to determine if there exist $\rho$ and corresponding surfaces $f_r$ and $f_b$ for each color such that $f_r$ can be obtained by a re-parametrization of $f_b$. That is, if there exist a positive function $\rho\in C^2(D)$, real numbers $C_r$ and $C_b$, and a $C^1$ map $\p:D\to D$ such that the surfaces $f_r$ and $f_b$, corresponding to $\rho$, $C_r,C_b$, have normals at each point and 
\begin{equation}\label{eq:equality by re parametrization point source}
f_r(t)=f_b(\p(t))\qquad \forall t\in D.
\end{equation} 

We refer to this as {\it Problem B}. As in the collimated case, if a solution exists $f_r(D) \subseteq f_b(D)$. Again, from an optical point of view, this means that the lens sandwiched between $\rho$ and $f_b$ refracts both colors into $w$; however, there could be points in $f_b(D)$ that are not reached by red rays.

If $w\notin \Omega$, we will show in Theorem \ref{thm:nonexistence point source} that Problem B is not solvable. 
On the other hand, when $w\in \Omega$ we shall prove, in dimension two, that problem B is locally solvable, Theorem \ref{thm:Existence last}. 
Notice that by rotating the coordinates we may assume without loss of generality that $w=e$. Theorem \ref{thm:Existence last} will follow from Theorem \ref{thm: Existence} on functional differential equations, assuming an initial  size condition on the ratio between the thickness of the lens and its distance to the origin. By local solution we mean that there exists an interval $[-\delta,\delta]\subseteq D$, a positive function $\rho\in C^2[-\delta,\delta]$, real numbers $C_r$ and $C_b$, and  
$\varphi:[-\delta,\delta]\to [-\delta,\delta]$ $C^1$ such that the corresponding surfaces $f_r$ and $f_b$ have normals at every point and 
\begin{equation*}
f_r(t)=f_b(\p(t))\qquad \forall t\in [-\delta,\delta].
\end{equation*} 
We will also show a necessary condition for solvability of Problem B, Corollary \ref{cor: Sufficient Condition}.

We first state the following lemma whose proof is the same as that of Lemma \ref{lem:fixed point}.
\begin{lemma}\label{lem:fixed point normal}
Given a surface $\rho(t)x(t)$, $t\in D$, and $w$ a unit vector in $\R^3$, let $f_r$ and $f_b$ be the surfaces parametrized as above. If $f_r(t)=f_b(t)$ for some $t\in D$, then $\nu_\rho(t)=x(t)$. In addition $d_b(t)=d_r(t)$.
\end{lemma}

We next show nonexistence of solutions to Problem $B$ for $w\notin \Omega$.

\begin{theorem}\label{thm:nonexistence point source}
Let $w$ be a unit vector in $\R^3$. If Problem B is solvable, then $x(t)=w$ for some $t\in D$. Therefore, since $x(D)=\Omega$, Problem $B$ has no solutions for $w\notin \Omega$.
\end{theorem}

\begin{proof}
Suppose there exist $\rho$ and $\varphi:D\to D$ satisfying \eqref{eq:equality by re parametrization point 
source}. Since $D$ is a compact and convex domain, by Brouwer fixed point theorem $\varphi$ has a fixed point $t_0$, and from
\eqref{eq:equality by re parametrization point source} $f_r(t_0)=f_b(t_0)$. Therefore, by Lemma \ref{lem:fixed point normal} $\nu_
\rho(t_0)=x(t_0)$, and by the Snell's law at 
$\rho(t_0)x(t_0)$ we have $m_b(t_0)=m_r(t_0)=x(t_0)$. Using Snell's law again at $f_r(t_0)=f_b(t_0)$, since $n_r\neq n_b$ and both colors with direction $x(t_0)$ are refracted at 
$f_r(t_0)=f_b(t_0)$ into $w$, we obtain $x(t_0)=w$.
\end{proof}

From now on our objective is to show that problem B is locally solvable in dimension two when $w\in \Omega$, Theorem \ref{thm:Existence last}.

\subsection{Two dimensional case, $w\in \Omega$.} Let $w$ be a unit vector in $\R^2$, by rotating the coordinates we will assume that $w=e=(0,1)$. Let $\Omega$ be a compact domain of the upper circle, such that $\Omega=x(D)$ where $D$ is a closed interval in $(-\pi/2,\pi/2)$, and $x(t)=(\sin t,\cos t)$. 

We will use  the following expression for the normal to a parametric curve.
\begin{lemma}\label{lm:first normal}
If a curve is given by the polar parametrization $\rho(t)x(t)=\rho(t)\,(\sin t, \cos t)$, with $\rho\in C^1$, then the unit outer normal is
$$\nu(t)=\dfrac{1}{\sqrt{\rho^2(t)+\rho'(t)^2}}\left(\rho(t)\sin t-\rho'(t)\cos t, \rho'(t)\sin t+\rho(t)\cos t\right).$$

\end{lemma}

\begin{proof}
The tangent vector to the curve at the point $\rho(t)\,x(t)$, with $x(t)=(\sin t,\cos t)$, equals
$$
(\rho(t)\, x(t))'=\rho'(t)x(t)+\rho(t)x'(t)
=\left(\rho'(t)\sin t+\rho(t)\cos t, \rho'(t)\cos t-\rho(t)\sin t\right).
$$
Thus
\begin{align*}
|(\rho(t)\, x(t))'|^2
&=\rho(t)^2+\rho'(t)^2.
\end{align*}
Hence the normal
$$
\nu(t)=\pm\dfrac{1}{\sqrt{\rho^2(t)+\rho'(t)^2}}\left(\rho(t)\sin t-\rho'(t)\cos t, \rho'(t)\sin t+\rho(t)\cos t\right).
$$
Since $\nu(t)$ is outer, i.e. $x(t)\cdot \nu(t)\geq 0$, so we take the positive sign above and 
%
the lemma follows.
\end{proof}

As a consequence, we obtain the following important lemma.
\begin{lemma}\label{lem:initial condition point source}
Assume Problem $B$ is solvable in the plane when $w=e$. Then $0\in D$,
$$\varphi(0)=0,\qquad d_b(0)=d_r(0), \qquad \text{and } \rho'(0)=0.$$
\end{lemma}

\begin{proof}
Using the proof of Theorem \ref{thm:nonexistence point source}, there exists $t_0\in D$ such that $\varphi(t_0)=t_0$ and $x(t_0)=e=(0,1)$, then $(\sin t_0,\cos t_0)=(0,1),$ and $t_0=0$. By Lemma \ref{lem:fixed point normal}, we get $d_b(0)=d_r(0)$, and $\nu_{\rho}(0)=e$. Therefore, Lemma \ref{lm:first normal} yields
$$
(0,1)=\dfrac{1}{\sqrt{\rho^2(0)+\rho'(0)^2}}(-\rho'(0),\rho(0)).
$$
\end{proof}
\subsection{Derivation of a system of functional equations from the solvability of Problem B in the plane.}\label{subsec:Problem B implies System}
Assume Problem B has a solution refracting rays of both colors b and r into the direction $e$, and recall Lemma \ref{lem:initial condition point source}.
\footnote{We are assuming that $0$ is an interior point of $D$. Otherwise, in our statements the interval $[-\delta,\delta]$ has to be replaced by either $[-\delta,0]$ or $[0,\delta]$.}
%

We set $\rho(0)=\rho_0$ and $d_b(0)=d_r(0)=d_0$, and prove the following theorem.
\begin{theorem}\label{thm:Problem B implies system}
Suppose there exist $\rho$ and $\varphi$ solving Problem $B$ in an interval $D$. 
Let $Z(t)=\(z_1(t),z_2(t),z_3(t),z_4(t),z_5(t)\)\in \mathcal \R^5$ with
\begin{equation}\label{eq:change of variables phi to z}
z_1(t)=\varphi(t),\quad z_2(t)=v_1(t)+\rho_0,\quad z_3(t)=v_2(t),\quad z_4(t)=v_1'(t),\quad z_5(t)=v_2'(t)-\rho_0,
\end{equation}
where $v_1(t)=-\rho(t)\cos t$, and $v_2(t)=\rho(t)\sin t$.

Let $\mathcal Z=(0;{\bf 0,0;}Z'(0),Z'(0))\in \R^{21}$.\footnote{$Z'(0)=\(\p'(0),0,\rho_0,-\rho''(0)+\rho_0,0\)$.}
There exists a neighborhood of $\mathcal Z$ and a map $H$ defined and smooth in that neighborhood with
%
$$
H:=H(t;\zeta^0,\zeta^1;\xi^0,\xi^1)=(h_1,\cdots,h_5)
$$
where  $\zeta^0,\zeta^1,\xi^0,\xi^1\in\R^5$, 
$
\zeta^i=\(\zeta^i_1,\zeta^i_2,\cdots,\zeta^i_5\), \xi^i=\(\xi^i_1,\xi^i_2,\cdots,\xi^i_5\),
$
and with the functions $h_1,\cdots ,h_5$ given by  \eqref{eq:formula for h1}, \eqref{eq:formula for h2}, \eqref{eq:formula for h3}, \eqref{eq:formula for h4}, and
\eqref{eq:h5formula}, respectively, such that
$Z$ is a solution to the system of functional differential equations
\begin{align}
Z'(t)&=H\(t;Z(t),Z(z_1(t));Z'(t),Z'(z_1(t))\)\label{eq:Optic System}\\
Z(0)&={\bf 0}\notag
\end{align}
for $t$ in a neighborhood of $0$. The map $H$ depends on the values $\rho_0$ and $d_0$.

\end{theorem}

\begin{proof}
From Lemma \ref{lem:initial condition point source}, $Z(0)={\bf 0}$.
We will derive the expressions for $h_i(t;\zeta^0,\zeta^1;\xi^0,\xi^1)$, $1\leq i\leq 5$, so that 
$$z_i'(t)=h_i\(t;Z(t),Z(z_1(t)),Z'(t),Z'(z_1(t))\).$$

The first step is to express the quantities involved as functions of $$\(t;Z(t),Z(z_1(t));Z'(t),Z'(z_1(t))\).$$
From the Snell law, a ray emitted from the origin with color r and direction $x(t)=(\sin t,\cos t)$ refracts by a curve $\rho(t)x(t)$ into a medium with refractive index $n_r$ into the direction $m_r(t)$ such that
$
x(t)-\,n_r m_r(t)=\Phi_{n_r}(x\cdot\nu(t))\,\nu(t),
$
where $\nu(t)$ is the outward unit normal to $\rho$ at $\rho(t)x(t)$.
From \eqref{formulaforlambda}
\[
\Phi_{n_r}(s)=s-\sqrt{n_r^2-1+s^2}=\dfrac{1-n_r^2}{s+\sqrt{n_r^2-1+s^2}},
\]
and from Lemma \ref{lm:first normal} $\nu(t)=\dfrac{v'(t)}{|v'(t)|},$ and $x(t)\cdot \nu(t)=\dfrac{\rho(t)}{\sqrt{\rho^2(t)+\rho'(t)^2}}=\dfrac{|v(t)|}{|v'(t)|}$, with $v(t)=\(v_1(t),v_2(t)\)$. So
\begin{align*}
\Phi_{n_r}(x(t)\cdot \nu(t))&=\dfrac{1-n_r^2}{\dfrac{|v(t)|}{|v'(t)|}+\sqrt{n_r^2-1+\dfrac{|v(t)|^2}{|v'(t)|^2}}}=\dfrac{(1-n_r^2)|v'(t)|}{|v(t)|+\sqrt{|v(t)|^2+(n_r^2-1)|v'(t)|^2}}.
\end{align*}
Hence
\begin{align}
m_r(t)&=\dfrac{1}{n_r}\left[x(t)-\Phi_{n_r}(x(t)\cdot\nu(t))\nu(t)\right]=\dfrac{1}{n_r}\left[x(t)-A_r\(v(t),v'(t)\)v'(t)\right]\label{eq:Snelllaw}\\
&:=\(m_{1r}(t),m_{2r}(t)\)\notag
\end{align}
with
\begin{equation}\label{eq:Anr formula}
A_r(v(t),v'(t))=\dfrac{1-n_r^2}{|v(t)|+\sqrt{|v(t)|^2+(n_r^2-1)|v'(t)|^2}}.
\end{equation} 
Rewriting the last expressions in terms of the variables $z_i(t)$ introduced in \eqref{eq:change of variables phi to z}, and omitting the dependance in $t$ to simplify the notation, we obtain
\begin{align}
A_r(v(t),v'(t))&=\dfrac{1-n_r^2}{\left|(z_2-\rho_0,z_3)\right|+\sqrt{\left|(z_2-\rho_0,z_3)\right|^2+(n_r^2-1)\left|(z_4,z_5+\rho_0)\right|^2}}:=\mathcal A_r\left(Z(t)\right),\label{eq:Anr}\\
m_{1r}(t)&=\dfrac{1}{n_r}\left[\sin t-\mathcal A_r(Z)\,z_4\right]:=\mu_r\(t,Z(t)\),\label{eq:m1r}\\
m_{2r}(t)&=\dfrac{1}{n_r}\left[\cos t-\mathcal A_r(Z)\, (z_5+\rho_0)\right]:=\tau_r\(t,Z(t)\).\label{eq:m2r}
\end{align}
Notice that $\tau_r\(0,Z(0)\)=\tau_r\(0,\bf{0}\)=1$.

If for each $t$, the ray with direction $m_r(t)$ is refracted by the upper face of the lens into the direction $e=(0,1)$, then the upper face is parametrized by the vector
$$f_{r}(t)=\rho(t)x(t)+d_r(t)m_{r}(t):=\(f_{1r}(t),f_{2r}(t)\).$$ 
From \eqref{eq:formula for dbt point source}, and \eqref{eq:m2r}
\begin{align}
d_r(t)&=\dfrac{C_{r}-\rho(t)(1-\cos t)}{n_r-m_{2r}(t)}=\dfrac{C_{r}-|v(t)|-v_1(t)}{n_r-\tau_r(t,Z(t))}\label{eq:Dr}\\
&=\dfrac{C_{r}-\left|(z_2-\rho_0,z_3)\right|-z_2+\rho_0}{n_r-\tau_r\(t,Z(t)\)}:=D_r\(t,Z(t)\),\notag
\end{align}
with $C_r$ a constant. 
Since $\rho$ and $\varphi$ solve Problem B, from Lemma \ref{lem:initial condition point source} we have $\rho'(0)=0$, and $\nu(0)=(0,1)$. So $m_r(0)=(0,1)$ and from \eqref{eq:formula for dbt point source} we get $C_r=(n_r-1)\,d_0$.
Hence
\begin{align}
f_{1r}(t)&=z_3+D_r(t,Z)\,\mu_r(t,Z):=F_{1r}(t,Z(t))\label{eq:F1r}\\
f_{2r}(t)&=-z_2+\rho_0+D_r(t,Z)\,\tau_r(t,Z):=F_{2r}(t,Z(t))\label{eq:F2r}.
\end{align}
In addition, by the Snell law at $f_r(t)$, 
$
m_r(t)-\dfrac{1}{n_r}e=\lambda_{2,r}(t)\,\nu_r(t),
$
where $\nu_r$ is the normal to the upper surface at the point $f_r(t)$ (we are using here that the normal to $f_r$ exists since we are assuming Problem B is solvable).
Since $n_r>1$, then $\lambda_{2,r}>0$ and so taking absolute values in the last expression yields
\begin{align}
\lambda_{2,r}(t)&=\left|m_r(t)-\dfrac{1}{n_r}e\right|=\sqrt{1+\dfrac{1}{n_r^2}-\dfrac{2}{n_r}m_{2r}(t)}=\sqrt{1+\dfrac{1}{n_r^2}-\dfrac{2}{n_r}\tau_r(t,Z)}:=\Lambda_{2,r}(t,Z(t)).\label{eq:lambdar}
\end{align}
For $t\in \R$ and $\zeta=(\zeta_1,\cdots,\zeta_5)\in \R^5$ we let
\begin{equation}\label{eq:Formulas 1}
\begin{cases}
&\mathcal A_r(\zeta)=\dfrac{1-n_r^2}{\left|(\zeta_2-\rho_0,\zeta_3)\right|+\sqrt{\left|(\zeta_2-\rho_0,\zeta_3)\right|^2+(n_r^2-1)\left|(\zeta_4,\zeta_5+\rho_0)\right|^2}}\\
&\mu_r(t,\zeta)=\dfrac{1}{n_r}\left[\sin t-\mathcal A_r(\zeta)\,\zeta_4\right],\qquad
\tau_r(t,\zeta)=\dfrac{1}{n_r}\left[\cos t-\mathcal A_r(\zeta)\, (\zeta_5+\rho_0)\right]\\
&D_r(t,\zeta)=\dfrac{C_{r}-\left|(\zeta_2-\rho_0,\zeta_3)\right|-\zeta_2+\rho_0}{n_r-\tau_r\(t,\zeta\)},\qquad \text{with } C_r=(n_r-1)\,d_0\\
&F_{1r}(t,\zeta)= \zeta_3+D_r(t,\zeta)\,\mu_r(t,\zeta),\qquad
F_{2r}(t,\zeta)=-\zeta_2+\rho_0+D_r(t,\zeta)\,\tau_r(t,\zeta)\\
&\Lambda_{2,r}(t,\zeta)=\sqrt{1+\dfrac{1}{n_r^2}-\dfrac{2}{n_r}\tau_r(t,\zeta)}.
\end{cases}
\end{equation}
The functions $\mathcal A_r,\mu_r,\tau_r$ are well defined and smooth for all $t\in \R$ and for all $\zeta=(\zeta_1,\zeta_2,\zeta_3,\zeta_4,\zeta_5)$ with $(\zeta_2,\zeta_3)\neq (\rho_0,0)$.
Since $\tau_r\(0,\bf{0}\)=1$, then all functions in \eqref{eq:Formulas 1} are well defined and smooth in a neighborhood of $t=0$ and $\zeta=\bf 0$.
Notice that the definitions of $\mathcal A_r(\zeta),\mu_r(t,\zeta),\tau_r(t,\zeta)$ and $\Lambda_{2,r}(t,\zeta)$ depend on the value of $\rho_0$, and the definitions of $D_r(t,\zeta),F_{1r}(t,\zeta)$ and $F_{2r}(t,\zeta)$ depend on the values of $\rho_0$ and $d_0$.

To determine later the functions $h_i$ we 
next calculate the derivatives of $A_r,m_{1r}, m_{2r}, d_r, F_{1r}, F_{2r},$ and $\lambda_{2,r}$ with respect to $t$. Differentiating \eqref{eq:Anr}, \eqref{eq:m1r}, \eqref{eq:m2r}, \eqref{eq:Dr}, \eqref{eq:F1r}, and \eqref{eq:F2r} with respect to $t$ yields


%
{\small
\begin{align}
\dfrac{d}{d\,t}A_r(v(t),v'(t))
&=\dfrac{n_r^2-1}{\(\left|(z_2-\rho_0,z_3)\right|+\sqrt{\left|(z_2-\rho_0,z_3)\right|^2+(n_r^2-1)\left|(z_4,z_5+\rho_0)\right|^2}\)^2}\label{eq:tildeAnr}\\
&\qquad \qquad \left[\dfrac{(z_2-\rho_0,z_3)\cdot (z_2',z_3')}{|(z_2-\rho_0,z_3)|}+\dfrac{(z_2-\rho_0,z_3)\cdot (z_2',z_3')+(n_r^2-1)(z_4,z_5+\rho_0)\cdot (z_4',z_5')}{\sqrt{\left|(z_2-\rho_0,z_3)\right|^2+(n_r^2-1)\left|\(z_4,z_5+\rho_0\)\right|^2}}\right]\notag\\
&=\dfrac{\mathcal A_r(Z)^2}{n_r^2-1}\left[ \dfrac{(z_2-\rho_0,z_3)\cdot (z_2',z_3')}{\left|(z_2-\rho_0,z_3)\right|}+\dfrac{(z_2-\rho_0,z_3)\cdot (z_2',z_3') +(n_r^2-1)(z_4,z_5+\rho_0)\cdot (z_4',z_5')}{\sqrt{\left|(z_2-\rho_0,z_3)\right|^2+(n_r^2-1)\left|(z_4,z_5+\rho_0)\right|^2}}\right]\notag\\
&:=\widetilde{\mathcal A}_r(Z(t),Z'(t)),\notag\\
\dfrac{d}{d\,t}m_{1r}(t)&=\dfrac{1}{n_r}\left[\cos t -\mathcal A_r(Z)\,z_4'-\widetilde{\mathcal A}_r(Z,Z')\,z_4\right]:=\widetilde{\mu}_r\(t,Z(t),Z'(t)\),\label{eq:tildemur}\\
\dfrac{d}{d\,t}m_{2r}(t)&=\dfrac{1}{n_r}\left[-\sin t-\mathcal A_r(Z)\,z_5'-\widetilde{\mathcal A}_r(Z,Z')\,(z_5+\rho_0)\right]:=\widetilde{\tau}_r\(t,Z(t),Z'(t)\)\label{eq:tildetaur},
\end{align}
\begin{align}
\dfrac{d}{d\,t}d_r(t)&=\dfrac{\left(-\dfrac{\(z_2-\rho_0,z_3\)\cdot(z_2',z_3')}{\left|\(z_2-\rho_0,z_3\)\right|}-z_2'\right)(n_r-\tau_r(t,Z))+\widetilde{\tau}_r(t,Z,Z')\left(C_{r}-\left|\(z_2-\rho_0,z_3\)\right|-z_2+\rho_0\right)}{\(n_r-\tau_r(t,Z)\)^2}\label{eq:tildeDr}\\
&=\dfrac{-\dfrac{(z_2-\rho_0,z_3)\cdot (z_2',z_3')}{\left|(z_2-\rho_0,z_3)\right|}-z_2'+\tilde \tau_r(t,Z,Z')\,D_r(t,Z)}{n_r-\tau_r(t,Z)}:=\widetilde {D}_r(t,Z(t),Z'(t))\notag
\end{align}
\begin{align}
\dfrac{d}{d\,t}f_{1r}(t)&=z_3'+D_r(t,Z)\,\widetilde{\mu}_r(t,Z,Z')+\widetilde{D}_r(t,Z,Z')\,\mu_r(t,Z):=\widetilde{F}_{1r}(t,Z(t),Z'(t))\label{eq:tildeF1r}
\end{align}
\begin{align}
\dfrac{d}{d\,t}f_{2r}(t)&=-z_2'+D_r(t,Z)\,\widetilde{\tau}_r(t,Z,Z')+\widetilde{D}_r(t,Z,Z')\,\tau_r(t,Z):=\widetilde {F}_{2r}(t,Z(t),Z'(t))\label{eq:tildeF2r}\\
\dfrac{d}{d\,t}\lambda_{2,r}(t)&=-\dfrac{\dfrac{1}{n_r}\widetilde{\tau}_r(t,Z,Z')}{\Lambda_{2,r}(t,Z)}:=\widetilde{\Lambda}_{2,r}(t,Z(t),Z'(t)).\label{eq:tildelambda}
\end{align}
}
For $t\in \R$ and $\zeta,\xi\in \mathcal \R^5$ we let
{\small 
\begin{equation}\label{eq:derivativeformula}
\begin{cases}
&\widetilde{\mathcal A}_r(\zeta,\xi)=\dfrac{\mathcal A_r(\zeta)^2}{n_r^2-1}\left[ \dfrac{(\zeta_2-\rho_0,\zeta_3)\cdot (\xi_2,\xi_3)}{\left|(\zeta_2-\rho_0,\zeta_3)\right|}+\dfrac{(\zeta_2-\rho_0,\zeta_3)\cdot (\xi_2,\xi_3) +(n_r^2-1)(\zeta_4,\zeta_5+\rho_0)\cdot (\xi_4,\xi_5)}{\sqrt{\left|(\zeta_2-\rho_0,\zeta_3)\right|^2+(n_r^2-1)\left|(\zeta_4,\zeta_5+\rho_0)\right|^2}}\right]\\ 
&\widetilde{\mu}_{r}(t,\zeta,\xi)=\dfrac{1}{n_r}\left[\cos t -\mathcal A_r(\zeta)\xi_4-\widetilde{\mathcal A}_r(\zeta,\xi)\zeta_4\right]\\ 
&\widetilde{\tau}_{r}(t,\zeta,\xi)=\dfrac{1}{n_r}\left[-\sin t-\mathcal A_r(\zeta)\xi_5-\widetilde{\mathcal A}_r(\zeta,\xi)(\zeta_5+\rho_0)\right]\\ 
&\widetilde{D}_{r}(t,\zeta,\xi)=\dfrac{-\dfrac{(\zeta_2-\rho_0,\zeta_3)\cdot (\xi_2,\xi_3)}{\left|(\zeta_2-\rho_0,\zeta_3)\right|}-\xi_2+\tilde \tau_r(t,\zeta,\xi)\,D_r(t,\zeta)}{n_r-\tau_r(t,\zeta)}\\ 
&\widetilde{F}_{1r}(t,\zeta,\xi)=\xi_3+D_r(t,\zeta)\widetilde{\mu}_r(t,\zeta,\xi)+\widetilde{D}_r(t,\zeta,\xi)\mu_r(t,\zeta)\\ 
&\widetilde{F}_{2r}(t,\zeta,\xi)=-\xi_2+D_r(t,\zeta)\widetilde{\tau}_r(t,\zeta,\xi)+\widetilde{D}_r(t,\zeta,\xi)\tau_r(t,\zeta)\\ 
&\widetilde{\Lambda}_{2,r}(t,\zeta,\xi)=-\dfrac{\dfrac{1}{n_r}\widetilde\tau_r(t,\zeta,\xi)}{\Lambda_{2,r}(t,\zeta)}. 
\end{cases}
\end{equation}
As for \eqref{eq:Formulas 1}, all functions in \eqref{eq:derivativeformula} are well defined and smooth in a neighborhood of $t=0$ and $\zeta=\bf 0$, and for any $\xi$.

}

We mention the following remark that will be used later.
\begin{remark}\label{rmk:derivative}\rm
Using the formulas in \eqref{eq:Formulas 1} and \eqref{eq:derivativeformula}, notice that for any differentiable map $U:V\to R^5$ with $V$ a neighborhood of $t=0$ and with $U(0)=\bf 0$, we have the following formulas valid for $t$ in a neighborhood $V'$ of $t=0$ (possibly smaller than $V$):
$$
\begin{array}{llrr}
\dfrac{d}{dt}[\mathcal A_r(U(t))]&=\widetilde {\mathcal A}_r(U(t),U'(t)),& \dfrac{d}{dt}[\mu_r(t,U(t))]&=\widetilde {\mu}_r(t,U(t),U'(t)), \\ \\ \dfrac{d}{dt}[\tau_r(t,U(t))]&=\widetilde {\tau}_r(t,U(t),U'(t)),&
\dfrac{d}{dt}[D_r(t,U(t))]&=\widetilde {D}_r(t,U(t),U'(t)), \\ \\
\dfrac{d}{dt}[F_{1r}(t,U(t))]&=\widetilde {F}_{1r}(t,U(t),U'(t)) ,
&\dfrac{d}{dt}[F_{2r}(t,U(t))]&=\widetilde {F}_{2r}(t,U(t),U'(t)),\\ \\
\dfrac{d}{dt}[\Lambda_{2,r}(t,U(t))]&=\widetilde {\Lambda}_{2,r}(t,U(t),U'(t)).
\end{array}
$$
\end{remark}
We also obtain the same formulas for the color $b$ with $n_r$ replaced by $n_b$.

We are now ready to calculate $h_i$, one by one, for $i=1,\cdots, 5$ to obtain the system of functional differential equations satisfied by $Z(t)=(\varphi(t),v_1(t)+\rho_0,v_2(t),v_1'(t),v_2'(t)-\rho_0)$. 
Recall $h_i:=h_i\(t;\zeta^0,\zeta^1;\xi^0,\xi^1\)$, with $\zeta^0,\zeta^1,\xi^0,\xi^1\in \R^5$,
$\zeta^i=\(\zeta^i_1,\cdots,\zeta^i_5\)$, $\xi^i=\(\xi^i_1,\cdots,\xi^i_5\),$ $i=0,1$.
\\ \\
\textbf{Calculation of $h_1$.}
Since $z_1=\varphi$ satisfies $f_{r}(t)=f_{b}(\varphi(t))$, taking the first components and differentiating with respect to $t$ yields
\begin{equation}\label{eq:first components derivatives equal}
f_{1r}'(t)=\varphi'(t)\,f_{1b}'(\varphi(t)).
\end{equation}
We claim that $f_{1b}'(\varphi(t))\neq 0$ in a neighborhood of $t=0$. 
\begin{proof}[Proof of the claim.]
By continuity of $f_{1b}'\circ \varphi$ and since $\p(0)=0$ by Lemma \ref{lem:initial condition point source}, it is equivalent to show that $f_{1b}'(0)\neq 0$.

Recall that $f_{1r}(t)=\rho(t)\sin t+d_r(t)m_{1r}(t)$, $\rho(0)=\rho_0$, $d_r(0)=d_0$, and $m_r(0)=(0,1)$. Then 
$$f_{1r}'(0)=\rho_0+d_0\,m_{1r}'(0).$$
Also from \eqref{eq:tildemur}, since $Z(0)={\bf 0}$
$$m_{1r}'(0)=\dfrac{1}{n_r}\left(1-\mathcal A_{r}({\bf 0})z_4'(0)\right).$$
Notice that $z_4'(t)=v_1''(t)=\rho(t)\cos t-\rho''(t)\cos t+2\rho'(t)\sin t$, then $z_4'(0)=\rho_0-\rho''(0)$. Also from \eqref{eq:Anr}
$$
\mathcal A_{r}({\bf 0})=\dfrac{1-n_r^2}{(1+n_r)\rho_0}=\dfrac{1-n_r}{\rho_0}.
$$
We conclude that 
\begin{align*}
m_{1r}'(0)&=\dfrac{1}{n_r}\left(1+\dfrac{n_r-1}{\rho_0}(\rho_0-\rho''(0))\right)=1-\dfrac{n_r-1}{n_r}\dfrac{\rho''(0)}{\rho_0}.
\end{align*}
Similarly, $f_{1b}(t)=\rho(t)\sin t+d_b(t)m_b(t)$, $d_b(0)=d_0$, $m_b(0)=(0,1)$, and we get
$f_{1b}'(0)=\rho_0+d_0\,m_{1b}'(0)$, with
$$m_{1b}'(0)=1-\dfrac{n_b-1}{n_b}\dfrac{\rho''(0)}{\rho_0}.$$

Suppose by contradiction that  $f_{1b}'(0)=0$. Then by \eqref{eq:first components derivatives equal} $f_{1r}'(0)=0$. 
Hence from the calculations above $m_{1b}'(0)=m_{1r}'(0)=-\rho_0/d_0$. Since $n_r\neq n_b$, $m_{1b}'(0)=m_{1r}'(0)$ implies $\rho''(0)=0$. So $$m_{1b}'(0)=m_{1r}'(0)=1=-\rho_0/d_0<0,$$
a contradiction.
\end{proof}

Since $z_1=\p$ we then conclude that $f_{1b}'(z_1(t))\neq 0$ in a neighborhood of $t=0$, and 
$$z_1'(t)=\varphi'(t)=\dfrac{f_{1r}'(t)}{f_{1b}'(z_1(t))}.$$
Applying formula \eqref{eq:tildeF1r} for both $b$ and $r$ yields
\begin{align}
z_1'(t)=\varphi'(t)&=\dfrac{\widetilde{F}_{1r}(t,Z(t),Z'(t))}{\widetilde{F}_{1b}(z_1(t),Z(z_1(t)),Z'(z_1(t)))}:=h_1(t;Z(t),Z(z_1(t));Z'(t),Z'(z_1(t)))\notag
\end{align}
with
\begin{equation}\label{eq:formula for h1}
h_1(t;\zeta^0,\zeta^1;\xi^0,\xi^1)=\dfrac{\widetilde F_{1r}(t,\zeta^0,\xi^0)}{\widetilde F_{1b}(\zeta^0_1,\zeta^1,\xi^1)};
\end{equation}
$\widetilde{F}_{1r}$ is defined explicitly in \eqref{eq:derivativeformula}, and $\widetilde{F}_{1b}$ has a similar expression with $n_r$ replaced by $n_b$.

We next verify that $h_1$ is smooth in a neighborhood of $\mathcal Z=(0;{\bf 0,0};Z'(0),Z'(0))$; $Z'(0)=\(\p'(0),0,\rho_0,-\rho''(0)+\rho_0,0\)$. From \eqref{eq:Formulas 1}, $\mathcal A_r$ is smooth in a neighborhood of ${\bf 0}\in R^5$, $\mu_r,\tau_r, D_r, F_{1r}, F_{2r}, \Lambda_{2,r}$ are smooth in a neighborhood of $(0,{\bf 0})\in \R^6$. Also, from \eqref{eq:derivativeformula}, $\widetilde A_{r}$ is smooth in a neighborhood of $({\bf 0,Z'(0)})\in R^{10}$, and  $\widetilde{\mu}_r,\widetilde{\tau}_r, \widetilde{D}_r, \widetilde{F}_{1r}, \widetilde{F}_{2r}, \widetilde{\Lambda}_{2,r}$ are smooth in a neighborhood of $(0,\bf{0,Z'(0)})\in R^{11}$. Similarly, we have the same smoothness for the functions corresponding to $n_b$.
Therefore, from \eqref{eq:formula for h1},  to show that $h_1$ is smooth in a neighborhood of $\mathcal Z$, it is enough to prove that $\widetilde{F}_{1b}(0,{\bf 0},Z'(0))\neq 0$. 
In fact, since $Z(0)={\bf 0}$ we obtain from \eqref{eq:tildeF1r} for $n_b$ and the claim above that
$$\widetilde {F}_{1b}\(0,{\bf 0},Z'(0)\)=\widetilde{F}_{1b}\(0, Z(0),Z'(0)\)= f_{1b}'(0)\neq 0.$$

%
%
%
%
%

\textbf{Calculation of $h_2$ and $h_3$.}
We have 
\begin{equation}\label{eq:z2'}
z_2'(t)=v_1'(t)=z_4(t):=h_2(t;Z(t),Z(z_1(t));Z'(t),Z'(z_1(t)))
\end{equation}
with,
\begin{equation}\label{eq:formula for h2}
h_2(t;\zeta^0,\zeta^1;\xi^0,\xi^1)=\zeta^0_4.
\end{equation}

Similarly, 
\begin{equation}\label{eq:z3'}
z_3'(t)=v_2'(t)=z_5(t)+\rho_0:=h_3(t;Z(t),Z(z_1(t));Z'(t),Z'(z_1(t)))
\end{equation}
with
\begin{equation}\label{eq:formula for h3}
h_3(t;\zeta^0,\zeta^1;\xi^0,\xi^1)=\zeta^0_5+\rho_0.
\end{equation}
Trivially, $h_2$ and $h_3$ are smooth everywhere.
%

\textbf{Calculation of $h_4$.}
The rays $m_{r}(t)$ and $m_{b}(\varphi(t))$ are both refracted into $e=(0,1)$ at $f_r(t)$ (since Problem B is solvable $f_r$ has a normal vector), then by the Snell law
\begin{align*}
m_r(t)-\dfrac{1}{n_r}e&=\lambda_{2,r}(t)\nu_S(t)\\
m_b(\varphi(t))-\dfrac{1}{n_b}e&=\lambda_{2,b}(\varphi(t))\nu_S(t)
\end{align*}
with $\nu_S$ the outer unit normal to $S=\{f_r(D)\}$.
%
%
If $\alpha_S$ denotes the first component of $\nu_S$, then
\begin{align*}
m_{1r}(t)=\lambda_{2,r}(t)\alpha_S(t),\qquad 
m_{1b}(\varphi(t))=\lambda_{2,b}(\varphi(t))\alpha_S(t).
\end{align*}
Solving in $\alpha_S(t)$ and using \eqref{eq:m1r} and \eqref{eq:lambdar} yields
$$\mu_r(t,Z(t))\,\Lambda_{2,b}(z_1(t),Z(z_1(t)))=\mu_b(z_1(t),Z(z_1(t)))\,\Lambda_{2,r}(t,Z(t)).
$$
Differentiating the last expression with respect to $t$, Remark \ref{rmk:derivative} yields
\begin{align*}
&\tilde \mu_r(t,Z,Z')\,\Lambda_{2,b}(z_1,Z(z_1))
+z_1'\,\widetilde{\Lambda}_{2,b}(z_1,Z(z_1),Z'(z_1))\,\mu_r(t,Z)\\
&\qquad =z_1'\,\widetilde{\mu}_b(z_1,Z(z_1),Z'(z_1))\,\Lambda_{2,r}(t,Z)+\mu_b(z_1,Z(z_1))\,\widetilde{\Lambda}_{2,r}(t,Z,Z').
\end{align*}
Replacing \eqref{eq:tildemur} in the above expression
\begin{align*}
&\dfrac{1}{n_r}\left(\cos t -\mathcal A_r(Z)\,z_4'-\widetilde{\mathcal A}_r(Z,Z')\,z_4\right)\Lambda_{2,b}(z_1,Z(z_1))\\
&\quad=z_1'\left[\widetilde{\mu}_b(z_1,Z(z_1),Z'(z_1))\Lambda_{2,r}(t,Z)-\widetilde{\Lambda}_{2,b}(z_1,Z(z_1),Z'(z_1))\mu_r(t,Z)\right]+\mu_b(z_1,Z(z_1))\widetilde{\Lambda}_{2,r}(t,Z,Z').
\end{align*}
From \eqref{eq:Anr} and \eqref{eq:lambdar}, $\mathcal A_{r}(Z)\Lambda_{2,b}(z_1,Z(z_1))<0$. Then solving the last equation in $z_4'$ yields
{\tiny
\begin{align}
\hspace{-2.5cm} 
z_4'&=-\dfrac{z_1'\left[\widetilde{\mu}_b(z_1,Z(z_1),Z'(z_1))\Lambda_{2,r}(t,Z)-\widetilde{\Lambda}_{2,b}(z_1,Z(z_1),Z'(z_1))\mu_{r}(t,Z)\right]+\mu_b(z_1,Z(z_1))\widetilde {\Lambda}_{2,r}(t,Z,Z')-\dfrac{1}{n_r}\left(\cos t-\widetilde{\mathcal A}_r(Z,Z')\,z_4\right)\Lambda_{2,b}(z_1,Z(z_1))}{\dfrac{1}{n_r}\mathcal A_r(Z)\Lambda_{2,b}(z_1,Z(z_1))}\label{eq:z4'}\\
&:=h_4(t;Z(t),Z(z_1(t));Z'(t),Z'(z_1(t))),\notag
\end{align}
}
so
{\tiny
\begin{align}\label{eq:formula for h4}
h_4(t;\zeta^0,\zeta^1;\xi^0,\xi^1)&=-\dfrac{\xi^0_1\left[\widetilde \mu_b(\zeta^0_1,\zeta^1,\xi^1)\Lambda_{2,r}(t,\zeta^0)-\widetilde{\Lambda}_{2,b}(\zeta^0_1,\zeta^1,\xi^1)\mu_r(t,\zeta^0)\right]+\mu_b(\zeta^0_1,\zeta^1)\widetilde {\Lambda}_{2,r}(t,\zeta^0,\xi^0)-\dfrac{1}{n_r}\left(\cos t-\widetilde{\mathcal A}_r(\zeta^0,\xi^0)\zeta^0_4\right)\Lambda_{2,b}(\zeta^0_1,\zeta^1)}{\dfrac{1}{n_r}\mathcal A_r(\zeta^0)\Lambda_{2,b}(\zeta^0_1,\zeta^1)}.
\end{align}
}
%
$h_4$ is smooth in a neighborhood of $\mathcal Z$, since $\mathcal A_{r}({\bf 0})\Lambda_{2,b}(0,{\bf 0})<0,$ and as shown before all the functions appearing in the expression for $h_4$ are smooth in a neighborhood of $\mathcal Z$ from the comments after \eqref{eq:Formulas 1} and \eqref{eq:derivativeformula}.
\\ \\

\textbf{Calculation of $h_5$.} We have $v_2(t)=-(\tan t) \,v_1(t)$.
Differentiating twice we get
$$v_2''(t)=-(\tan t) \,v_1''(t)-2\,\dfrac{v_1'(t)}{\cos ^2 t}-2\dfrac{\sin t}{\cos ^3 t}\,v_1(t).$$
Then
\begin{align}
z_5'(t)&=-(\tan t) \,z_4'(t)-\dfrac{2}{\cos ^2 t}\,z_4(t)-\dfrac{2 \sin t}{\cos ^3 t}\,(z_2(t)-\rho_0)\label{eq:z5'}  \\
&:=h_5(t;Z(t),Z(z_1(t));Z'(t),Z'(z_1(t)))\notag,
\end{align}
and so
\begin{equation}\label{eq:h5formula}
h_5(t;\zeta^0,\zeta^1;\xi^0,\xi^1)=-(\tan t)\,\xi^0_4-\dfrac{2}{\cos ^2 t}\,\zeta^0_4-\dfrac{2\sin t}{\cos^3 t}\,(\zeta^0_2-\rho_0).
\end{equation}

Since $t\in D\subset(-\pi/2,\pi/2)$, then $h_5$ is smooth in $D\times \R^{20}$.

The proof of Theorem \ref{thm:Problem B implies system} is then complete.
\end{proof}

\subsection{Solutions of \eqref{eq:Optic System} yield local solutions to the optical problem.} \label{subsec:converse functional implies optic}
In this section, we show how to obtain a local solution to Problem B by solving the system of functional differential equations \eqref{eq:Optic System}.
\begin{theorem}\label{thm:Converse}
Let $\rho_0,d_0>0$ be given, and suppose that $H$ is the corresponding map defined in Theorem \ref{thm:Problem B implies system}. Assume $P=(p_1,\cdots,p_n)$ is a solution to the system 
$$P=H(0;{\bf 0,0};P,P),$$
such that $H$ is smooth in a neighborhood of $\mathcal P:=(0;{\bf 0,0};P,P)$, and
\begin{equation}\label{eq:additional}
0<|p_1|\leq 1.
\end{equation}

Let $Z(t)=(z_1(t),\cdots,z_5(t))$ be a $C^1$ solution to the system \eqref{eq:Optic System} in some open interval $I$ containing $0$, with $Z'(0)=P$ and 
\begin{equation}\label{eq:condition on z1 point source}
|z_1(t)|\leq |t|.
\end{equation}
Define 
\begin{equation}\label{eq:rho def}
\rho(t)=-\dfrac{z_2(t)-\rho_0}{\cos t},
\end{equation} 
and $\varphi(t)=z_1(t)$.

Then there is $\delta>0$ sufficiently small, so that $\varphi:[-\delta,\delta]\to [-\delta,\delta]$ and 
$$f_r(t)=f_b(\varphi(t)).$$ 
Here,
$f_r(t)=\rho(t)x(t)+d_r(t)m_r(t)$, and $f_b(t)=\rho(t)x(t)+d_b(t)m_b(t)$
where
\[d_r(t)=\dfrac{C_r-\rho(t)(1-\cos t)}{n_r-e\cdot m_r(t)},\qquad d_b(t)=\dfrac{C_b-\rho(t)(1-\cos t)}{n_b-e\cdot m_b(t)}
\]
with $C_r=(n_r-1)\,d_0$, $C_b=(n_b-1)\,d_0$; $m_r(t)$ and $m_b(t)$ are the refracted directions of the rays $x(t)$ by the curve $\rho(t)x(t)$ corresponding to each color r and b.
In addition, for $t\in [-\delta,\delta]$, $f_r$ and $f_b$ have normal vectors for every $t$ and
\begin{equation}\label{eq:positivity and internal reflection conditions}
\rho(t)>0,\quad d_r(t),d_b(t)>0,\quad m_r(t)\cdot e\geq 1/n_r,\quad m_b(t)\cdot e\geq1/n_b.
\end{equation}
\end{theorem}
Notice that \eqref{eq:positivity and internal reflection conditions} implies the lens defined with $\rho(t)x(t)$, $f_r$ and $f_b$ is well defined, and moreover total internal reflection is avoided.

\begin{proof}
We obtain the theorem by proving a series of steps.


\begin{step}\label{clm:z3 and rho}
If $\rho$ is from \eqref{eq:rho def}, then
\begin{equation}\label{eq:z3 form}
z_3(t)=\rho(t)\sin t.
\end{equation}
\end{step}

\begin{proof}
Since $z_2(0)=0$, from \eqref{eq:rho def} 
\begin{equation}\label{eq:rho at 0}
\rho(0)=\rho_0.
\end{equation}
Since $Z$ is a solution to \eqref{eq:Optic System}, $z_2'(t)=z_4(t)$ by the definition of $h_2$ in \eqref{eq:formula for h2}. Hence and from the definition of $h_5$ in \eqref{eq:h5formula}, 
\begin{align*}
z_5'(t)&=-(\tan t)\,z_4'(t)-\dfrac{2}{\cos ^2 t}\,z_4(t)-\dfrac{2\sin t}{\cos ^3 t}\,(z_2(t)-\rho_0)\\
&=-(\tan t)\,z_2''(t)-\dfrac{2}{\cos ^2 t}\,z_2'(t)-\dfrac{2\sin t}{\cos ^3 t}\,(z_2(t)-\rho_0).
\end{align*}
By \eqref{eq:rho def}, we have $z_2(t)=-\rho(t)\cos t+\rho_0$, then
\begin{align*}
z_2'(t)&=-\rho'(t)\cos t+\rho(t)\sin t\\
z_2''(t)&=-\rho''(t)\cos t+2\rho'(t)\sin t+\rho(t)\cos t.
\end{align*}
Replacing in the formula for $z_5'$ we get 
\begin{align*}
z_5'(t)&=(\tan t)(\rho''(t)\cos t-2\rho'(t)\sin t-\rho(t)\cos t)+\dfrac{2}{\cos ^2 t}(\rho'(t)\cos t-\rho(t)\sin t)+\dfrac{2\sin t}{\cos^3t}\rho(t)\cos t\\
&=\rho''(t)\sin t+2\rho'(t)\cos t -\rho(t)\sin t\\
&=\dfrac{d^2}{dt^2}\left[\rho(t)\sin t\right].
\end{align*}
Hence
$z_5(t)=\dfrac{d}{dt}\left[\rho(t)\sin t\right]+c=\rho'(t)\sin t+\rho(t)\cos t+c.$
Since $z_5(0)=0$, then by \eqref{eq:rho at 0}  $c=-\rho_0$, and therefore
\begin{equation}\label{eq:z5 form}
z_5(t)=\rho'(t)\sin t+\rho(t)\cos t-\rho_0.
\end{equation}
By the definition of $h_3$ in \eqref{eq:formula for h3}, we have 
$$z_3'(t)=z_5(t)+\rho_0=\rho'(t)\sin t+\rho(t)\cos t=\dfrac{d}{dt}\left[\rho(t)\sin t\right],$$
and since $z_3(0)=0$, we conclude \eqref{eq:z3 form}.
\end{proof}

\begin{step}\label{clm:First Composition}
For each $t\in I$
$$F_{1r}(t,Z(t))=F_{1b}(z_1(t),Z(z_1(t))),$$
with $F_{1r}, F_{1b}$ from \eqref{eq:Formulas 1}.
\end{step}

\begin{proof}
From the definition of $h_1$ in \eqref{eq:formula for h1}
$$\widetilde{F}_{1r}(t,Z(t),Z'(t))=z_1'(t)\widetilde{F}_{1b}(z_1(t),Z(z_1(t)),Z'(z_1(t))).$$
Hence by Remark \ref{rmk:derivative}, integrating the above equality yields 
$$F_{1r}(t,Z(t))=F_{1b}(z_1(t),Z(z_1(t)))+c.$$
Since $Z(0)={\bf 0}$, then from the formulas of $F_{1r}$ and $\mu_r$ in \eqref{eq:Formulas 1} we get
$F_{1r}(0,{\bf 0})=F_{1b}(0,{\bf 0})=0$. Hence $c=0$ and Step \ref{clm:First Composition} follows.
\end{proof}

\begin{step}\label{step:choice of delta for total reflection}
There is $\delta>0$ small so that \eqref{eq:positivity and internal reflection conditions} holds.
\end{step}
\begin{proof}
Since $\rho(0)=\rho_0>0$, by continuity there is $\delta>0$ with $[-\delta,\delta]\subseteq I$ so that $\rho(t)$ given by \eqref{eq:rho def} is positive for $t\in [-\delta,\delta]$. 
Rays with colors r and b emitted from the origin with direction $x(t)=(\sin t,\cos t)$ are refracted at $\rho(t)x(t)$ into the directions $m_r(t)$, and $m_b(t)$. 
For the upper faces $f_r$ and $f_b$ to be able to refract the rays $m_r$ and $m_b$ into $e$, they need to have a normal vector for each $t$, $d_r(t)$ and $d_b(t)$ must be positive for $t\in [-\delta,\delta]$, and $m_r$ and $m_b$ must satisfy the conditions $m_r(t)\cdot e\geq  \dfrac{1}{n_r}$, $m_b(t)\cdot e\geq  \dfrac{1}{n_b}$ to avoid total reflection. 
Notice that from \eqref{eq:rho def} 
$$\rho'(t)=-\dfrac{z_2'(t)\cos t+(z_2(t)-\rho_0)\sin t}{\cos ^2 t}.$$
From the definition of $h_2$ in \eqref{eq:formula for h2}, and since $z_4(0)=0$ we conclude that $\rho'(0)=0$. Hence from Lemma \ref{lm:first normal}, the normal to $\rho$ at $\rho(0)(0,1)$ is $\nu(0)=(0,1)$. Since $x(0)=\nu(0)=(0,1)$ we obtain from Snell's law that 
\begin{equation}\label{eq:normal refraction}
m_r(0)=m_b(0)=(0,1).
\end{equation}
Thus, $e\cdot m_r(0)=e\cdot m_b(0)=1$, and so
\begin{equation}\label{eq:initial d}
d_r(0)=\dfrac{C_r}{n_r-1}=d_0=\dfrac{C_b}{n_b-1}=d_b(0)>0.
\end{equation}
So choosing $\delta$ small we get $d_r(t), d_b(t)>0$, $e\cdot m_r(t)\geq \dfrac{1}{n_r}$ and $e\cdot m_b(t)\geq \dfrac{1}{n_b}$  in $[-\delta,\delta]$.
Therefore, \eqref{eq:positivity and internal reflection conditions} holds for $\delta>0$ sufficiently small. The fact that $f_r$ and $f_b$ have a normals for every $t$ will be proved in Step \ref{clm:regular points}.
\end{proof}

\begin{step}\label{eq:Identity Composition}
For each $t\in [-\delta,\delta]$
$$f_r(t)=\left(F_{1r}(t,Z(t)),F_{2r}(t,Z(t))\right),$$ and similarly $f_b(t)=\left(F_{1b}(t,Z(t)),F_{2b}(t,Z(t))\right)$.
\end{step}

\begin{proof}
We first show that
\begin{equation}\label{eq:mr recover}
m_r(t)=\(\mu_r(t,Z(t)),\tau_r(t,Z(t))\),
\end{equation}
and similarly for $m_b$. Since the direction $x(t)=(\sin t, \cos t)$ is refracted by $\rho$ into the unit direction $m_r(t)$, 
by \eqref{eq:Snelllaw}
$$
m_r(t)=\dfrac{1}{n_r}\left[x(t)-A_r(v(t),v'(t))\,v'(t)\right]
$$
with $A_r$ given in \eqref{eq:Anr formula} and $v(t)=(-\rho(t)\cos t, \rho(t)\sin t)$. 
From \eqref{eq:rho def} and \eqref{eq:z3 form}, 
$$v(t)=\left(z_2(t)-\rho_0,z_3(t)\right),\qquad v'(t)=\left(z_2'(t),z_3'(t)\right).$$
However, by the definition of $h_2$ in \eqref{eq:formula for h2} and the definition of $h_3$ in \eqref{eq:formula for h3}, we have $z_2'=z_4$ and $z_3'=z_5+\rho_0$, then $v'(t)=\(z_4,z_5+\rho_0\)$.
Hence \eqref{eq:Anr formula} becomes
$$
A_r(v(t),v'(t))=\dfrac{1-n_r^2}{\left|(z_2(t)-\rho_0,z_3(t))\right|+\sqrt{\left|(z_2(t)-\rho_0,z_3(t))\right|^2+(n_r^2-1)\left|(z_4(t),z_5(t)+\rho_0)\right|^2}},
$$
and so by \eqref{eq:Formulas 1}, $A_r(v(t),v'(t))=\mathcal A_r(Z(t)).$  

We conclude that
$$
m_r(t)=\dfrac{1}{n_r}\left[(\sin t,\cos t)-\mathcal A_r(Z(t))\left(z_4(t),z_5(t)+\rho_0\right)\right],
$$
and hence once again from \eqref{eq:Formulas 1}, the identity \eqref{eq:mr recover} follows.

We now show that $d_r(t)=D_r(t,Z(t))$. In fact, by definition of $v$ we have 
$$\rho(t)=|v(t)|=\left|\left(z_2(t)-\rho_0,z_3(t)\right)\right|.$$
Then by \eqref{eq:mr recover}, \eqref{eq:rho def} and the formula for $D_r$ in \eqref{eq:Formulas 1} we get
\begin{align*}
d_r(t)=\dfrac{C_r-\rho(t)+\rho(t)\cos t}{n_r-m_{2r}(t)}=\dfrac{C_r-\left|(z_2(t)-\rho_0,z_3(t))\right|-z_2(t)+\rho_0}{n_r-\tau_r(t,Z(t))}=D_r(t,Z(t)).
\end{align*}
Hence from \eqref{eq:rho def}, \eqref{eq:z3 form}, \eqref{eq:mr recover}, and the formulas of $F_{1r},F_{2r}$ in \eqref{eq:Formulas 1}, we conclude 
\begin{align*}
f_r(t)&=\rho(t)(\sin t,\cos t)+d_{r}(t)m_r(t)\\
&=\left(z_3(t),-z_2(t)+\rho_0\right)+D_r(t,Z(t))\left(\mu_r(t,Z(t)),\tau_r(t,Z(t))\right)\\
&=\left(F_{1r}(t,Z(t)),F_{2r}(t,Z(t))\right).
\end{align*}
A similar result holds for $f_b$.
\end{proof}

%


By assumption $|z_1(t)|\leq |t|$ for $t\in I$, then $z_1(t)\in [-\delta,\delta]$ for every $t\in [-\delta, \delta]$.

\begin{step}\label{clm:regular points}
For $\delta$ small, $f_{1r}'(t)\neq 0$ and $f_{1b}'(t)\neq 0$, and hence $f_r$ and $f_b$ have normal vectors for each $t\in [-\delta,\delta]$.
\end{step}

\begin{proof}
By continuity of $f_{1r}'$ and $f_{1b}'$, it is enough to show that $f_{1r}'(0)\neq 0$ and $f_{1b}'(0)\neq 0$. 
Since $H$ is well defined at $\mathcal P$, then from the definition of $h_1$ in \eqref{eq:formula for h1}, $\widetilde {F}_{1b}(0,{\bf 0},P)\neq 0$. We have $Z'(0)=P$, then from Remark \ref{rmk:derivative} and Step \ref{eq:Identity Composition}
$$f_{1b}'(0)=\widetilde{F}_{1b}(0,{\bf 0},Z'(0))=\widetilde{F}_{1b}(0,{\bf 0},P)\neq 0.$$
We next prove that $f_{1r}'(0)\neq 0$. From Steps \ref{clm:First Composition} and \ref{eq:Identity Composition}, $f_{1r}(t)=f_{1b}(z_1(t))$. Differentiating and letting $t=0$ yields
$$f_{1r}'(0)=z_1'(0) f_{1b}'(0).$$
Since $f_{1b}'(0)\neq 0$ and from \eqref{eq:additional} $z_1'(0)=p_1\neq 0$, we obtain $f_{1r}'(0)\neq 0$.
\end{proof}

\begin{step}\label{eq:Colinearity}
The vectors
$m_r(t)-\dfrac{1}{n_r}e$ and $m_b(z_1(t))-\dfrac{1}{n_b}e$ are colinear for every $t\in [-\delta,\delta]$.
\end{step}

\begin{proof}
Since $z_4'(t)=h_4(t;Z(t),Z(z_1(t));Z'(t),Z'(z_1(t)))$, from \eqref{eq:formula for h4} it follows that 
\begin{align*}
\dfrac{1}{n_r}&\left(\cos t -\mathcal A_r(Z(t))z_4'(t)-\widetilde{\mathcal A}_r(Z(t),Z'(t))z_4(t)\right)\Lambda_{2,b}(z_1(t),Z(z_1(t)))=\\
 &\qquad z_1'(t)\left[\widetilde{\mu}_{b}(z_1(t),Z(z_1(t)),Z'(z_1(t)))\Lambda_{2,r}(t,Z(t))-\widetilde{\Lambda}_{2,b}(z_1(t),Z(z_1(t)),Z'(z_1(t)))\mu_{r}(t,Z(t))\right]\\
 &\qquad +\mu_{b}(z_1(t),Z(z_1(t)))\widetilde{\Lambda}_{2,r}(t,Z(t),Z'(t)).
\end{align*}
Hence from the formula of $\widetilde \mu_r$ in \eqref{eq:derivativeformula}, we obtain
\begin{align*}
\widetilde{\mu}_r&(t,Z(t),Z'(t))\,\Lambda_{2,b}(z_1(t),Z(z_1(t)))+z_1'(t)\,\widetilde{\Lambda}_{2,b}(z_1(t),Z(z_1(t)))\,\mu_r(t,Z(t))\\
&\qquad=z_1'(t)\,\widetilde{\mu}_b(z_1(t),Z(z_1(t)))\,\Lambda_{2,r}(t,Z(t))+\mu_b(z_1(t),Z(z_1(t)))\,\widetilde{\Lambda}_{2,r}(t,Z(t),Z'(t)).
\end{align*}
Integrating the resulting identity using Remark \ref{rmk:derivative}, and that $\mu_{r}(0,{\bf 0})=\mu_{b}(0,{\bf 0})=0$ from \eqref{eq:Formulas 1}, we obtain
\begin{equation}\label{eq:First Component bis}
\mu_r(t,Z(t))\,\Lambda_{2,b}\(z_1(t),Z(z_1(t))\)=\mu_b\(z_1(t),Z(z_1(t))\)\,\Lambda_{2,r}(t,Z(t)).
\end{equation}
On the other hand, from \eqref{eq:mr recover}
\begin{equation}\label{eq:unit} 
\mu_r(t,Z(t))^2+\tau_r(t,Z(t))^2=1,
\end{equation}
then from \eqref{eq:Formulas 1}, $\Lambda_{2,r}$ can be written as follows
{\small
\begin{align*}
\Lambda_{2,r}(t,Z(t))&=\sqrt{1+\dfrac{1}{n_r^2}-\dfrac{2}{n_r}\tau_r(t,Z(t))}=\sqrt{1-\tau_r(t,Z(t))^2+\left(\dfrac{1}{n_r}-\tau_r(t,Z(t))\right)^2}\\
&=\sqrt{\mu_r(t,Z(t))^2+\left(\dfrac{1}{n_r}-\tau_r(t,Z(t))\right)^2},
\end{align*}
and similarly for $\Lambda_{2,b}$.}
Squaring \eqref{eq:First Component bis} and using the above identity for $n_r$ and $n_b$ we get
\begin{align*}
&\mu_r(t,Z(t))^2\left[\mu_b(z_1(t),Z(z_1(t)))^2+\left(\dfrac{1}{n_b}-\tau_b(z_1(t),Z(z_1(t))\right)^2\right]\\
&=\mu_b(z_1(t),Z(z_1(t)))^2\left[\mu_r(t,Z(t))^2+\left(\dfrac{1}{n_r}-\tau_r(t,Z(t))\right)^2\right].
\end{align*}
Hence 
 $$
 \mu_r(t,Z(t))^2 \left(\dfrac{1}{n_b}-\tau_b(z_1(t),Z(z_1(t)))\right)^2=\mu_b(z_1(t),Z(z_1(t)))^2\left(\dfrac{1}{n_r}-\tau_r(t,Z(t))\right)^2,
$$
and taking square roots
$$ \left| \mu_r(t,Z(t)) \left(\dfrac{1}{n_b}-\tau_b(z_1(t),Z(z_1(t)))\right) \right| =\left |\mu_b(z_1(t),Z(z_1(t)))\left(\dfrac{1}{n_r}-\tau_r(t,Z(t))\right)\right|.$$
From \eqref{eq:mr recover}, since $\delta$ is chosen so that \eqref{eq:positivity and internal reflection conditions} holds in $[-\delta,\delta]$, and $z_1(t)\in [-\delta,\delta]$, we have that  $1/n_b-\tau_b(z_1(t),Z(z_1(t)))=1/n_b-e\cdot m_b(z_1(t))\leq 0$, and $1/n_r-\tau_r(t,Z(t))=1/n_r-e\cdot m_r(t)\leq 0$. Moreover, since the functions $\Lambda_{2,b}$ and $\Lambda_{2,r}$ are both positive, then by \eqref{eq:First Component bis} $\mu_r(t,Z(t))$ and $\mu_b(z_1(t),Z(z_1(t)))$ have the same sign obtaining
$$ \mu_r(t,Z(t)) \left(\tau_b(z_1(t),Z(z_1(t)))-\dfrac{1}{n_b}\right)  =\mu_b(z_1(t),Z(z_1(t)))\left(\tau_r(t,Z(t))-\dfrac{1}{n_r}\right).$$
We conclude that the vectors
$$\left( \mu_r(t,Z(t)),\tau_r(t,Z(t))-\dfrac{1}{n_r}\right),\quad \left(\mu_b\(z_1(t),Z(z_1(t))\),\tau_b\(z_1(t),Z(z_1(t))\)-\dfrac{1}{n_b}\right)$$
are colinear and the claim follows from \eqref{eq:mr recover}.
\end{proof}

By Steps \ref{clm:First Composition} and \ref{eq:Identity Composition}, we obtain that $f_{1r}(t)=f_{1b}(z_1(t))$, so to conclude the proof of Theorem \ref{thm:Converse},
it remains to show the following.

\begin{step}\label{eq:Second Component}
We have for $|t|\leq \delta$ that
$$f_{2r}(t)=f_{2b}(z_1(t)).$$
\end{step} 
\begin{proof}
From Step \ref{clm:regular points}, for $|t|\leq \delta$, $f_r$ and $f_b$ have normals at every point. From \eqref{eq:positivity and internal reflection conditions} and the definitions of $f_r$ and $f_b$, $m_r(t)$ and $m_b(t)$ are refracted into the direction $e$ at $f_r(t)$ and $f_b(t)$, for each $t$, respectively. So the Snell law at $f_r(t)$ and $f_b(t)$ implies that $m_r(t)-\dfrac{1}{n_r}e$ is orthogonal to the tangent vector $f_r'(t)$, and $m_b(z_1(t))-\dfrac{1}{n_b}e$ is orthogonal to $f_b'(z_1(t))$.
%
%
%
Hence by Step \ref{eq:Colinearity}, $f_r'(t)$ and $f_b'(z_1(t))$ are parallel.
From Remark \ref{rmk:derivative} and Steps  \ref{clm:First Composition} and \ref{eq:Identity Composition} we also obtain that $f_{1r}'(t)=z_1'(t)f_{1b}'(z_1(t))$. 
From Step \ref{clm:regular points} and assumption \eqref{eq:condition on z1 point source}, $f_{1r}'(t)\neq 0$ and $f_{1b}'(z_1(t))\neq 0$ for all $t\in [-\delta,\delta]$, then
$$f_{2r}'(t)=z_1'(t)f_{2b}'(z_1(t)).$$
%
%
Integrating the last identity, we obtain $f_{2r}(t)=f_{2b}(z_1(t))+c.$
On the other hand, $z_1(0)=0$, 
$f_{2r}(0)=\rho(0)+d_r(0)m_{2r}(0)$, and $f_{2b}(0)=\rho(0)+d_b(0)m_{2b}(0)$.
Hence from \eqref{eq:normal refraction} and \eqref{eq:initial d}, $f_{2r}(0)=f_{2b}(0)$ and so $c=0$, and Step \ref{eq:Second Component} follows.

\end{proof}
This completes the proof Theorem \ref{thm:Converse}.
\end{proof}

\begin{remark}\label{rmk:self intersection}\rm
Notice that $f_b$ has no self intersections in the interval $[-\delta,\delta]$.  Because if there would exist $t_1,t_2\in [-\delta,\delta]$ such that
$f_b(t_1)=f_b(t_2)$, then by Rolle's theorem applied to $f_{1b}$ we would get $f_{1b}'(t_0)=0$ for some $t_0\in [t_1,t_2]$, a contradiction with Step \ref{clm:regular points}. The issue of  self-intersections in the monochromatic case is discussed in detail in \cite{gutierrez-sabra:freeformgeneralfields}.
This implies that $f_b$ is injective, and similarly $f_r$ is injective.
We also deduce that $\varphi$ is injective: in fact, if $\varphi(t_1)=\varphi(t_2)$ then $f_r(t_1)=f_b(\varphi(t_1))=f_b(\varphi(t_2))=f_r(t_2)$ and so $t_1=t_2$.
\end{remark}

\subsection{On the solvability of the algebraic system \eqref{eq:system P} }
In this section, we analyze for what values of $\rho_0$, $d_0$ the algebraic system $P=H(0;{\bf 0,0};P,P)$ has a solution, where $H$
is given in Theorem \ref{thm:Problem B implies system}. This analysis will be used to apply Theorem \ref{thm: Existence}, and to decide when Problem B has a local solution. Denote 
\[
k_0=\dfrac{\rho_0}{d_0},\qquad \Delta_r=\dfrac{n_r}{n_r-1},\qquad \Delta_b=\dfrac{n_b}{n_b-1}.
\]

\begin{theorem}\label{thm:Algebraic system}
The algebraic system $P=H(0;{\bf 0,0};P,P)$ has a solution if and only if $k_0\leq \dfrac{(\Delta_r-\Delta_b)^2}{4\Delta_r\Delta_b}$\footnote{This expression is equal to $\(\sqrt{\dfrac{\Delta_r}{\Delta_b}}-\sqrt{\dfrac{\Delta_b}{\Delta_r}}\)^2$.}. 
In case of equality, the system has only one solution $P$ with $|p_1|=1$, and in case of strict inequality the system has two solutions $P$ and $P'$ with $0<|p_1|<1$ and $|p_1'|>1$.
\end{theorem}

\begin{proof}
Recall that $H=H\(t;\zeta^0,\zeta^1;\xi^0,\xi^1\)$.
Suppose $P=(p_1,\cdots,p_5)$ solves 
$P=H\left(0;{\bf 0,0}; P,P\right)$. Then 
from the definitions of $h_2$ in \eqref{eq:formula for h2}, of $h_3$ in \eqref{eq:formula for h3}, and of $h_5$ in
\eqref{eq:h5formula} we get
\begin{equation}
p_2=h_2\left(0;{\bf 0,0}; P,P\right)=0, \quad p_3=h_3\left(0;{\bf 0,0;} P,P\right)=\rho_0,\label{eq:p3}\quad
p_5=h_5\left(0;{\bf 0,0;} P,P\right)=0.
\end{equation}
Then, from \eqref{eq:Formulas 1} and \eqref{eq:derivativeformula} 
{\small
\begin{equation}\label{eq:values at 0}
\begin{cases}
&\mathcal A_r\left({\bf 0}\right)=-\dfrac{n_r}{\Delta_r\,\rho_0},\qquad 
\mathcal A_b\left({\bf 0}\right)=-\dfrac{n_b}{\Delta_b\,\rho_0}\\
&\mu_r\left(0,{\bf 0}\right)=\mu_b\left(0,{\bf 0}\right)=0\\
&\tau_r\left(0,{\bf 0}\right)=\tau_b\left(0,{\bf 0}\right)=1\\
&D_r\left(0,{\bf 0}\right)=D_b\left(0,{\bf 0}\right)=d_0\\
&F_{1r}\left(0,{\bf 0}\right)=F_{1b}\left(0,{\bf 0}\right)=0,\qquad 
F_{2r}\left(0,{\bf 0}\right)=F_{2b}\left(0,{\bf 0}\right)=\rho_0+d_0\\
&\Lambda_{2,r}\left(0,{\bf 0}\right)=\dfrac{1}{\Delta_r},\qquad
\Lambda_{2,b}\left(0,{\bf 0}\right)=\dfrac{1}{\Delta_b}\\
&\widetilde{\mathcal A}_r\left({\bf 0}, P\right)=\widetilde{\mathcal A}_b\left({\bf 0}, P\right)=0\\
&\widetilde{\mu}_r\left(0,{\bf 0}, P\right)=\dfrac{1}{\Delta_r}\left(\dfrac{\Delta_r}{n_r}+\dfrac{p_4}{\rho_0}\right),\qquad
\widetilde{\mu}_b\left(0,{\bf 0}, P\right)=\dfrac{1}{\Delta_b}\left(\dfrac{\Delta_b}{n_b}+\dfrac{p_4}{\rho_0}\right)\\
&\widetilde {\tau}_r\left(0,{\bf 0}, P\right)=\widetilde {\tau}_{n_r}\left(0,{\bf 0}, P\right)=0\\
&\widetilde {D}_r\left(0,{\bf 0}, P\right)=\widetilde {D}_b\left(0,{\bf 0}, P\right)=0\\
&\widetilde {F}_{1r}\left(0,{\bf 0}, P\right)=\rho_0+\dfrac{d_0}{\Delta_r}\left(\dfrac{\Delta_r}{n_r}+\dfrac{p_4}{\rho_0}\right),\qquad
\widetilde {F}_{1b}\left(0,{\bf 0}, P\right)=\rho_0+\dfrac{d_0}{\Delta_b}\left(\dfrac{\Delta_b}{n_b}+\dfrac{p_4}{\rho_0}\right)\\
&\widetilde {F}_{2r}\left(0,{\bf 0}, P\right)=\widetilde {F}_{2b}\left(0,{\bf 0}, P\right)=0\\
&\widetilde {\Lambda}_{2,r}\left(0,{\bf 0}, P\right)=\widetilde {\Lambda}_{2,b}\left(0,{\bf 0}, P\right)=0.
\end{cases}
\end{equation}
}
Hence by the definition of $h_1$ in \eqref{eq:formula for h1}, 
and since $h_1\left(0;{\bf 0,0}; P,P\right)$ is well defined, we get that $\widetilde{F}_{1b}\left(0,{\bf 0}, P\right)\neq 0$. That is, from \eqref{eq:values at 0}, $\rho_0+\dfrac{d_0}{\Delta_b}\left(\dfrac{\Delta_b}{n_b}+\dfrac{p_4}{\rho_0}\right)\neq 0$ and we have
\begin{equation}\label{eq:p1}
p_1=h_1\left(0;{\bf 0,0}; P,P\right)
=\dfrac{\widetilde {F}_{1r}\left(0,{\bf 0}, P\right)}{\widetilde{F}_{1b}\left(0,{\bf 0}, P\right)}
=\dfrac{\rho_0+\dfrac{d_0}{\Delta_r}\left(\dfrac{\Delta_r}{n_r}+\dfrac{p_4}{\rho_0}\right)}{\rho_0+\dfrac{d_0}{\Delta_b}\left(\dfrac{\Delta_b}{n_b}+\dfrac{p_4}{\rho_0}\right)}.
\end{equation}
From the definition of $h_4$ in \eqref{eq:formula for h4}
\begin{align}\label{eq:p4}
p_4=h_4\left(0;{\bf 0,0}; P,P\right)&=\dfrac{\dfrac{p_1}{\Delta_b\Delta_r}\(\dfrac{\Delta_b}{n_b}+\dfrac{p_4}{\rho_0}\)-\dfrac{1}{n_r\Delta_b}}{\dfrac{1}{\Delta_r\Delta_b\rho_0}}=\left[p_1\left(\dfrac{\Delta_b}{n_b}+\dfrac{p_4}{\rho_0}\right)-\dfrac{\Delta_r}{n_r}\right]\rho_0.
\end{align}
To solve \eqref{eq:p4} in $p_1$,
first notice that $\dfrac{\Delta_b}{n_b}+\dfrac{p_4}{\rho_0}\neq 0$. 
Otherwise, from \eqref{eq:p4} we would obtain $\dfrac{p_4}{\rho_0}=-\dfrac{\Delta_r}{n_r}$, and so $\dfrac{\Delta_r}{n_r}=\dfrac{\Delta_b}{n_b}$, hence $n_r=n_b$, a contradiction since $n_b>n_r$.
Then \eqref{eq:p4} yields
\begin{equation}\label{eq:second form p1}
p_1=\dfrac{\dfrac{\Delta_r}{n_r}+\dfrac{p_4}{\rho_0}}{\dfrac{\Delta_b}{n_b}+\dfrac{p_4}{\rho_0}}.
\end{equation}
Hence from  \eqref{eq:p1} and \eqref{eq:second form p1} we get that
$$
\dfrac{\dfrac{\Delta_r}{n_r}+\dfrac{p_4}{\rho_0}}{\dfrac{\Delta_b}{n_b}+\dfrac{p_4}{\rho_0}}=\dfrac{\rho_0+\dfrac{d_0}{\Delta_r}\left(\dfrac{\Delta_r}{n_r}+\dfrac{p_4}{\rho_0}\right)}{\rho_0+\dfrac{d_0}{\Delta_b}\left(\dfrac{\Delta_b}{n_b}+\dfrac{p_4}{\rho_0}\right)}
$$
Then 
\begin{equation}\label{eq:Trick 0}
\left(\dfrac{\Delta_r}{n_r}+\dfrac{p_4}{\rho_0}\right)\left(\rho_0+\dfrac{d_0}{\Delta_b}\left(\dfrac{\Delta_b}{n_b}+\dfrac{p_4}{\rho_0}\right)\right)=\left(\dfrac{\Delta_b}{n_b}+\dfrac{p_4}{\rho_0}\right)\left(\rho_0+\dfrac{d_0}{\Delta_r}\left(\dfrac{\Delta_r}{n_r}+\dfrac{p_4}{\rho_0}\right)\right).
\end{equation}
Simplifying,
$$
\rho_0\left(\dfrac{\Delta_r}{n_r}-\dfrac{\Delta_b}{n_b}\right)=\left(\dfrac{\Delta_r}{n_r}+\dfrac{p_4}{\rho_0}\right)\left(\dfrac{\Delta_b}{n_b}+\dfrac{p_4}{\rho_0}\right)\left(\dfrac{1}{\Delta_r}-\dfrac{1}{\Delta_b}\right)d_0.
$$
Notice that 
\begin{equation}\label{eq:trick}
\dfrac{\Delta_r}{n_r}-\dfrac{\Delta_b}{n_b}=\dfrac{1}{n_r-1}-\dfrac{1}{n_b-1}=-1+\Delta_r+1-\Delta_b=\Delta_r-\Delta_b.
\end{equation}
Then dividing by $\Delta_r-\Delta_b$\footnote{$\Delta_r-\Delta_b=\dfrac{n_b-n_r}{(n_r-1)(n_b-1)}> 0$}, and using the notation $k_0=\dfrac{\rho_0}{d_0}$ yields
\begin{equation}\label{eq:Useful}
k_0\Delta_r\Delta_b+\left(\dfrac{\Delta_b}{n_b}+\dfrac{p_4}{\rho_0}\right)\left(\dfrac{\Delta_r}{n_r}+\dfrac{p_4}{\rho_0}\right)=0.
\end{equation}
Expanding we obtain that $p_4/\rho_0$ satisfies the following quadratic equation:
\begin{equation}\label{eq:quadratic}
\(\dfrac{p_4}{\rho_0}\)^2+\(\dfrac{\Delta_r}{n_r}+\dfrac{\Delta_b}{n_b}\)\dfrac{p_4}{\rho_0}+k_0\Delta_r\Delta_b+\dfrac{\Delta_r\Delta_b}{n_rn_b}=0.
\end{equation}


The discriminant of \eqref{eq:quadratic} is 
\begin{equation}
\delta=\(\dfrac{\Delta_r}{n_r}+\dfrac{\Delta_b}{n_b}\)^2-4\dfrac{\Delta_r\Delta_b}{n_rn_b}-4k_0\Delta_r\Delta_b=\(\dfrac{\Delta_r}{n_r}-\dfrac{\Delta_b}{n_b}\)^2-4k_0\Delta_r\Delta_b,
\end{equation}
then \eqref{eq:trick} yields
\begin{equation}\label{eq:formula for delta}
\delta=(\Delta_r-\Delta_b)^2-4k_0\Delta_r\Delta_b.
\end{equation}
Hence \eqref{eq:quadratic} has a real solutions if and only if  $k_0\leq \dfrac{\(\Delta_r-\Delta_b\)^2}{4\Delta_r\Delta_b}$.
Therefore we have proved the necessity part in Theorem \ref{thm:Algebraic system}, and if $P$ solves the algebraic system, then 
\begin{align*}
p_4&=\dfrac{-\left(\dfrac{\Delta_r}{n_r}+\dfrac{\Delta_b}{n_b}\right)-
\sqrt{\delta}}{2}\rho_0,\qquad 
p_4'=\dfrac{-\left(\dfrac{\Delta_r}{n_r}+\dfrac{\Delta_b}{n_b}\right)
+\sqrt{\delta}}{2}\rho_0
\end{align*}
and by \eqref{eq:second form p1} and \eqref{eq:trick}, the corresponding $p_1$ and $p_1'$ are
\begin{align*}
p_1&=\dfrac{\Delta_r-\Delta_b- \sqrt{\delta}}{\Delta_b-\Delta_r-\sqrt{\delta}},\qquad 
p_1'=\dfrac{\Delta_r-\Delta_b+ \sqrt{\delta}}{\Delta_b-\Delta_r+\sqrt{\delta}}.
\end{align*}
Therefore if $P=H(0;{\bf 0,0};P,P)$, then the solutions are
\begin{align}
P&=\(\dfrac{\Delta_r-\Delta_b- \sqrt{\delta}}{\Delta_b-\Delta_r-\sqrt{\delta}},0,\rho_0,\dfrac{-\left(\dfrac{\Delta_r}{n_r}+\dfrac{\Delta_b}{n_b}\right)-\sqrt{\delta}}{2}\rho_0,0\right)\label{eq:P}\\
P'&= \(\dfrac{\Delta_r-\Delta_b+ \sqrt{\delta}}{\Delta_b-\Delta_r+\sqrt{\delta}},0,\rho_0,\dfrac{-\left(\dfrac{\Delta_r}{n_r}+\dfrac{\Delta_b}{n_b}\right)+\sqrt{\delta}}{2}\rho_0,0\right)\label{eq:P'}
\end{align}
with $\delta$ given in \eqref{eq:formula for delta}. 

Let us now prove that if $k_0\leq \dfrac{\(\Delta_r-\Delta_b\)^2}{4\Delta_r\Delta_b}$, then the system 
$P=H(0;{\bf 0,0};P,P)$ is solvable.
In fact, from the assumption on $k_0$ there is $p_4$ solving \eqref{eq:quadratic}. We claim that this implies $\rho_0+\dfrac{d_0}{\Delta_b}\left(\dfrac{\Delta_b}{n_b}+\dfrac{p_4}{\rho_0}\right)\neq 0$.
Assume otherwise, since \eqref{eq:quadratic} is equivalent to \eqref{eq:Trick 0} and $\rho_0>0$ then $p_4$ solves the system
\begin{align*}
\rho_0+\dfrac{d_0}{\Delta_b}\left(\dfrac{\Delta_b}{n_b}+\dfrac{p_4}{\rho_0}\right)&=0\\
\rho_0+\dfrac{d_0}{\Delta_r}\left(\dfrac{\Delta_r}{n_r}+\dfrac{p_4}{\rho_0}\right)&=0
\end{align*}
Subtracting both identities we get 
$$d_0\left(\dfrac{1}{n_b}-\dfrac{1}{n_r}\right)+\dfrac{p_4 d_0}{\rho_0}\left(\dfrac{1}{\Delta_b}-\dfrac{1}{\Delta_r}\right)=0.$$
Since $\dfrac{1}{\Delta_b}-\dfrac{1}{\Delta_r}=\dfrac{1}{n_r}-\dfrac{1}{n_b}$ and $n_r\neq n_b$, dividing by $d_0\left(\dfrac{1}{n_b}-\dfrac{1}{n_r}\right)$ we get
$1-\dfrac{p_4}{\rho_0}=0,$
and then $p_4=\rho_0$.
Replacing in the first equation of the system yields
$\rho_0=-\dfrac{d_0}{\Delta_b}\left(\dfrac{\Delta_b}{n_b}+1\right)$,
and since $\rho_0,d_0>0$, we get a contradiction.
Now let $p_1$ be the corresponding value to this $p_4$ given by \eqref{eq:p1}, and $p_2,p_3,p_5$ from \eqref{eq:p3}; and $P$ be the point with these coordinates. 
Hence by the formula for $\widetilde F_{1b}(0,{\bf 0},P)$ in \eqref{eq:values at 0}, we get that 
$\widetilde F_{1b}(0,{\bf 0},P)\neq 0$, and therefore $h_1\left(0;{\bf 0,0}; P,P\right)$ is well defined, and $P$ solves the algebraic system. Then the possible values of $P$ solving the algebraic system are $P$ and $P'$ given by \eqref{eq:P} and \eqref{eq:P'}.




We now prove the last part of the theorem. 
If $k_0=\dfrac{\(\Delta_r-\Delta_b\)^2}{4\Delta_r\Delta_b}$ then $\delta=0$ and $P=P'=\left(-1,0,\rho_0,\dfrac{-\left(\dfrac{\Delta_r}{n_r}+\dfrac{\Delta_b}{n_b}\right)}{2}\rho_0,0\right)$.
If $k_0<\dfrac{\(\Delta_r-\Delta_b\)^2}{4\Delta_r\Delta_b}$ then the solutions $P\neq P'$. Moreover, since $\Delta_r>\Delta_b$, and from \eqref{eq:formula for delta} $0<\sqrt{\delta}<\Delta_r-\Delta_b$, Therefore
\begin{align*}
0<\left|p_1\right|&=\dfrac{\left |\Delta_r-\Delta_b- \sqrt{\delta}\right|}{\left|\Delta_b-\Delta_r-\sqrt{\delta}\right|}= \dfrac{\Delta_r-\Delta_b- \sqrt{\delta}}{\Delta_r-\Delta_b+\sqrt{\delta}}<1\\
\left |p_1'\right|&=\dfrac{\left|\Delta_r-\Delta_b+ \sqrt{\delta}\right|}{\left|\Delta_b-\Delta_r+\sqrt{\delta}\right|}=\dfrac{\Delta_r-\Delta_b+ \sqrt{\delta}}{\Delta_r-\Delta_b-\sqrt{\delta}}>1.
\end{align*}
\end{proof}

The following corollary gives a necessary condition on $k_0$ for the existence of solutions to Problem $B$.
\begin{corollary}\label{cor: Sufficient Condition}
If $k_0>\dfrac{(\Delta_r-\Delta_b)^2}{4\Delta_r\Delta_b}$ then Problem $B$ has no local solutions.
\end{corollary}

\begin{proof}
If Problem B has a solution then by Theorem \ref{thm:Problem B implies system}, the vector $Z(t)=(\varphi(t),v_1(t)+\rho_0,v_2(t),v_1'(t),v_2'(t)-\rho_0)$ solves the functional system \eqref{eq:Optic System} for $t$ in a neighborhood of $0$. 
Plugging $t=0$ in \eqref{eq:Optic System} yields
$$\(0;{\bf 0, 0};Z'(0),Z'(0)\)=H(0;{\bf 0,0};Z'(0),Z'(0)).$$ Hence $Z'(0)$ is a solution to the algebraic system $P=H(0;{\bf 0,0};P,P)$, and from Theorem \ref{thm:Algebraic system} we have $k_0\leq \dfrac{(\Delta_r-\Delta_b)^2}{4\Delta_r\Delta_b}$. 
\end{proof}

\begin{corollary}\label{cor:consequence on rho''}
If $k_0\leq \dfrac{(\Delta_r-\Delta_b)^2}{4\Delta_r\Delta_b}$ and a solution $\rho$ and $\varphi$ to Problem B exists, then 
\begin{equation}\label{eq:necessary condition for existence Problem B}
\dfrac{\rho''(0)}{\rho_0}\in(\Delta_b,\Delta_r)
\end{equation}
In fact, 
\begin{equation}\label{eq:formula for rho''}
\dfrac{\rho''(0)}{\rho_0}=\dfrac{2+\left(\dfrac{\Delta_r}{n_r}+\dfrac{\Delta_b}{n_b}\right)\pm\sqrt{\delta}}{2},
\end{equation}
with $\delta$ given in \eqref{eq:formula for delta}.
\end{corollary}

\begin{proof}
If $\rho$ and $\varphi$ solve Problem B, then from the proof of Corollary \ref{cor: Sufficient Condition}, $Z'(0)$ solves the algebraic system, with
$Z(t)=(\varphi(t),v_1(t)+\rho_0,v_2(t),v_1'(t),v_2'(t)-\rho_0)$. Using the proof of Theorem \ref{thm:Algebraic system}, it follows that $z_4'(0)$ satisfies \eqref{eq:Useful} then
$$\(\dfrac{\Delta_r}{n_r}+\dfrac{z_4'(0)}{\rho_0}\)\(\dfrac{\Delta_b}{n_b}+\dfrac{z_4'(0)}{\rho_0}\)=-k_0\Delta_r\Delta_b< 0,$$
and therefore $\dfrac{z_4'(0)}{\rho_0}\in \left(-\dfrac{\Delta_r}{n_r},-\dfrac{\Delta_b}{n_b}\right)$.

On the other hand, $z_4(t)=v_1'(t)=\rho(t)\sin t-\rho'(t)\cos t$, obtaining
$z_4'(0)=\rho_0-\rho''(0)$, so
\begin{equation}\label{eq:rho''andz4'}
\dfrac{\rho''(0)}{\rho_0}=1-\dfrac{z_4'(0)}{\rho_0}.
\end{equation}
We conclude that
$$\dfrac{\rho''(0)}{\rho_0}\in\left(1+\dfrac{\Delta_b}{n_b},1+\dfrac{\Delta_r}{n_r}\right)=(\Delta_b,\Delta_r).$$

Finally, from \eqref{eq:P}, \eqref{eq:P'}, and \eqref{eq:rho''andz4'}, we obtain \eqref{eq:formula for rho''}. 
\end{proof}

\begin{remark}\label{rmk:three colors}\rm 
\,The analogue of Problem B for three or more colors has no solution. In fact, if rays, superposition of three colors, are emitted from $O$ and $n_r<n_j<n_b$ are the refractive indices inside the lens for each color, then \eqref{eq:necessary condition for existence Problem B} must be satisfied for the pairs $n_r,n_j$ and $n_j$, $n_b$. Hence $$\dfrac{\rho''(0)}{\rho(0)}\in\left(\Delta_j,\Delta_r\right)\cap\left(\Delta_b,\Delta_j\right),$$
which is impossible since the last intersection is empty.
\end{remark}

\subsection{Existence of local solutions to \eqref{eq:Optic System}}\label{sec:existence of local solution}
In order to prove existence of solutions to the system \eqref{eq:Optic System}, we will apply Theorem \ref{thm: Existence}. From Theorem \ref{thm:Algebraic system}, we must assume that $k_0\leq\dfrac{(\Delta_r-\Delta_b)^2}{4\Delta_r\Delta_b}$. In this case, the algebraic system $P=H(0;{\bf 0,0};P,P)$ has a solution given by \eqref{eq:P} with $|p_1|\leq 1$, and therefore \eqref{eq:First Component} holds. Let $\mathcal P=(0;{\bf 0,0};P,P)$.
To show that $H$ satisfies all the hypotheses of Theorem \ref{thm: Existence}, it remains to show there is a norm in $\R^5$ so that $H$ satisfies conditions (i)-(iv) with respect to this norm in a neighborhood $N_\varepsilon(\mathcal P)$.
Our result is as follows.

\begin{theorem}\label{thm:Existence last}
There exists a positive constant $C(r,b)<\dfrac{(\Delta_r-\Delta_b)^2}{4\Delta_r\Delta_b}$, depending only on $n_r$ and $n_b$, such that for $0<k_0<C(r,b)$, the system \eqref{eq:Optic System} has a unique local solution $Z(t)=(z_1(t),\cdots,z_5(t))$ with $Z'(0)=P$, and $|z_1(t)|\leq |t|$, with $P$ given in \eqref{eq:P}. 
Hence from Theorem \ref{thm:Converse} and for those values of $k_0$, there exist unique $\rho$ and $\varphi$ solving Problem B.
\end{theorem}

\begin{proof}
Since by construction $h_i$ are all smooth in a small neighborhood of $\mathcal P$,
%
$H$ is Lipschitz in that neighborhood for any norm. 
To apply Theorem \ref{thm: Existence}, we need to find a norm $\|\cdot \|$ in $\R^5$ and a neighborhood $N_{\varepsilon}(\mathcal P)$ so that $|h_1|\leq 1$, and $H$ satisfies the contraction condition \eqref{eq:Contraction}.
 

To prove the contraction property of $H$, we first calculate the following matrices $\nabla_{\xi^0}H=\left(\dfrac{\partial h_i}{\partial \xi^0_j}\right)_{1\leq i,j\leq 5}$, and  $\nabla_{\xi^1}H=\left(\dfrac{\partial h_i}{\partial \xi^1_j}\right)_{1\leq i,j\leq 5}$ at the point $\mathcal P$.
\\
\textbf{Calculation of $\nabla_{\xi^0}H(\mathcal P)=\left(\dfrac{\partial h_i}{\partial \xi^0_j}(\mathcal P)\right)_{1\leq i,j\leq 5}$.} \\
From \eqref{eq:formula for h2} and \eqref{eq:formula for h3}, $h_2$ and $h_3$ do not depend on $\xi^0$, then
$$\partial_{\xi^0_j}h_2(\mathcal P)=\partial_{\xi^0_j}h_3(\mathcal P)=0,\qquad 1\leq j\leq 5.$$
Also from the definition of $h_5$ in \eqref{eq:h5formula}
$$\partial_{\xi^0_j}h_5(\mathcal P)=-\delta_{j}^5\,(\tan t)|_{\mathcal P}= 0,\qquad 1\leq j\leq 5.$$

We next calculate $\nabla_{\xi^0} h_1(\mathcal P)$. From \eqref{eq:formula for h1},
$$
\nabla_{\xi^0} h_1(\mathcal P)=\dfrac{1}{\widetilde {F}_{1b}\left(0,{\bf 0},P\right)}\nabla_{\xi^0} \widetilde {F}_{1r}(0,{\bf 0},P).
$$
Recall from \eqref{eq:values at 0}
$$\widetilde{ F}_{1b}\left(0,{\bf 0},P\right)=\rho_0+\dfrac{d_0}{\Delta_b}\left(\dfrac{\Delta_b}{n_b}+\dfrac{p_4}{\rho_0}\right),\qquad \mu_r(0,{\bf 0})=0,\qquad D_r(0,{\bf 0})=d_0.$$
Differentiating $\widetilde F_{1r}$ given in \eqref{eq:derivativeformula} with respect to $\xi_j^0$, we then get
$$
\dfrac{\partial\widetilde F_{1r}}{\partial \xi_j^0}(0,{\bf 0},P)=\delta_{j}^3+d_0\dfrac{\partial \widetilde{\mu}_r}{\partial\xi_j^0}(0,{\bf 0},P),\qquad 1\leq j\leq 5.
$$
From \eqref{eq:values at 0}, $\mathcal A_r({\bf 0)}=\dfrac{-n_r}{\Delta_r\rho_0}$, then differentiating $\widetilde \mu_r$ in \eqref{eq:derivativeformula} with respect to $\xi_j^0$ at $(0,{\bf 0},P)$ yields
\begin{equation}\label{eq:derivative of tilde mur}
\dfrac{\partial \widetilde{\mu}_r}{\partial\xi_j^0}(0,{\bf 0},P)=\dfrac{1}{\Delta_r \rho_0}\delta_j^4.
\end{equation}
Therefore
\begin{align}
\nabla_{\xi^0} \widetilde{F}_{1r}(0,{\bf 0}, P)&=\left(0,0,1,\dfrac{1}{k_0\Delta_r},0\right)\label{eq:grad tilde F1r}.
\end{align}
We conclude that
\begin{equation}\label{eq:grad h1 xi0}
\nabla_{\xi^0} h_1(\mathcal P)=\dfrac{1}{\rho_0+\dfrac{d_0}{\Delta_b}\left(\dfrac{\Delta_b}{n_b}+\dfrac{p_4}{\rho_0}\right)}\left(0,0,1,\dfrac{1}{k_0\Delta_r},0\right).
\end{equation}

We next calculate $\nabla_{\xi^0} h_4(\mathcal P)$. Recall from \eqref{eq:values at 0} that
{\small
$$
\mu_b(0,{\bf 0})=\mu_r(0,{\bf 0})=0,\quad  \widetilde\mu_b(0,{\bf 0}, P)=\dfrac{1}{\Delta_b}\left(\dfrac{\Delta_b}{n_b}+\dfrac{p_4}{\rho_0}\right),\quad \Lambda_{2,r}(0,{\bf 0})=\dfrac{1}{\Delta_r}, \quad\Lambda_{2,b}(0,{\bf 0})=\dfrac{1}{\Delta_b},\quad
\mathcal A_r({\bf 0})=-\dfrac{n_r}{\Delta_r\rho_0}.
$$
}
Then from \eqref{eq:formula for h4} it follows that
$$
\dfrac{\partial h_4}{\partial \xi_j^0}=\dfrac{\dfrac{1}{\Delta_b}\left(\dfrac{\Delta_b}{n_b}+\dfrac{p_4}{\rho_0}\right)\dfrac{1}{\Delta_r}}{\dfrac{1}{\Delta_r \rho_0}\dfrac{1}{\Delta_b}}\delta_j^1=\rho_0\left(\dfrac{\Delta_b}{n_b}+\dfrac{p_4}{\rho_0}\right)\delta_j^1,\quad 1\leq j\leq 5.
$$
Hence
\begin{equation}\label{eq:grad h4 xi0}
\nabla_{\xi^0} h_4(\mathcal P)=\rho_0\left( \dfrac{\Delta_b}{n_b}+\dfrac{p_4}{\rho_0},0,0,0,0\right).
\end{equation}
We then conclude that
\begin{equation}\label{eq:Total derivative xi0}
\nabla_{\xi^0} H(\mathcal P)=
\begin{bmatrix}
    0 & 0 & \dfrac{1}{\rho_0+\dfrac{d_0}{\Delta_b}\left(\dfrac{\Delta_b}{n_b}+\dfrac{p_4}{\rho_0}\right)}& \dfrac{\dfrac{1}{k_0\Delta_r}}{\rho_0+\dfrac{d_0}{\Delta_b}\left(\dfrac{\Delta_b}{n_b}+\dfrac{p_4}{\rho_0}\right)}& 0  \\
    0 & 0 & 0 & 0 & 0\\    
    0 & 0 & 0 & 0 & 0\\    
    \rho_0\left(\dfrac{\Delta_b}{n_b}+\dfrac{p_4}{\rho_0}\right) & 0 & 0 & 0 & 0\\
0 & 0 & 0 & 0 & 0\\    
  \end{bmatrix}.
\end{equation}

To calculate the spectral radius of the matrix $\nabla_{\xi^0} H(\mathcal P)$, 
set 
\begin{equation}\label{eq:form of a c}
a=\dfrac{1}{\rho_0+\dfrac{d_0}{\Delta_b}\left(\dfrac{\Delta_b}{n_b}+\dfrac{p_4}{\rho_0}\right)},\qquad 
c=\rho_0\left(\dfrac{\Delta_b}{n_b}+\dfrac{p_4}{\rho_0}\right).
\end{equation}
The eigenvalues of $\nabla_{\xi^0} H(\mathcal P)$ are $0$ (with multiplicity $3$), and $\pm\sqrt{ \dfrac{ac}{k_0\Delta_r}}$.
Notice that from \eqref{eq:P}, and \eqref{eq:trick}
\begin{equation}\label{eq:trick 2}
\dfrac{\Delta_r}{n_r}+\dfrac{p_4}{\rho_0}=\dfrac{\Delta_r-\Delta_b-\sqrt{(\Delta_r-\Delta_b)^2-4k_0\Delta_r\Delta_b}}{2}=\dfrac{2k_0\Delta_r\Delta_b}{\Delta_r-\Delta_b+\sqrt{(\Delta_r-\Delta_b)^2-4k_0\Delta_r\Delta_b}}>0,
\end{equation}
then by \eqref{eq:p1} and \eqref{eq:second form p1}
\begin{equation}\label{eq:trick 3}
\dfrac{ac}{k_0\Delta_r}=\dfrac{\dfrac{\rho_0}{k_0\Delta_r}\left(\dfrac{\Delta_b}{n_b}+\dfrac{p_4}{\rho_0}\right)}{\rho_0+\dfrac{d_0}{\Delta_b}\left(\dfrac{\Delta_b}{n_b}+\dfrac{p_4}{\rho_0}\right)}=\dfrac{\dfrac{\rho_0}{k_0\Delta_r}\left(\dfrac{\Delta_r}{n_r}+\dfrac{p_4}{\rho_0}\right)}{\rho_0+\dfrac{d_0}{\Delta_r}\left(\dfrac{\Delta_r}{n_r}+\dfrac{p_4}{\rho_0}\right)}>0.
\end{equation}
Therefore all the eigenvalues of $\nabla_{\xi^0} H(\mathcal P)$ are real and the spectral radius of $\nabla_{\xi^0} H(\mathcal P)$ is 
$R_{\xi^0}=\sqrt{\dfrac{ac}{k_0\Delta_r}}.$ We estimate $R_{\xi^0}$.

From \eqref{eq:formula for delta}, \eqref{eq:trick 2}, and \eqref{eq:trick 3}
\begin{equation}\label{eq:ac}
\dfrac{ac}{k_0\Delta_r}
=
\dfrac{\dfrac{2\Delta_b }{\Delta_r-\Delta_b+\sqrt{\delta}}}{1+\dfrac{2\Delta_b}{\Delta_r-\Delta_b+\sqrt{\delta}}}
=
\dfrac{2\,\Delta_b}{\Delta_r+\Delta_b+\sqrt{\delta}}
\leq
\dfrac{2\,\Delta_b}{\Delta_r+\Delta_b}:=\delta_0<1.
\end{equation}
Since $\sqrt{\delta}<\Delta_r-\Delta_b$, we conclude that
\begin{equation}\label{eq:Rxi0 bound}
\sqrt{\dfrac{\Delta_b}{\Delta_r}}< R_{\xi^0}=\sqrt{\dfrac{2\,\Delta_b}{\Delta_r+\Delta_b+\sqrt{\delta}}}\leq\sqrt{\dfrac{2\Delta_b}{\Delta_b+\Delta_r}}=\sqrt{\delta_0}<1.
\end{equation}


\textbf{Calculation of $\nabla_{\xi^1}H(\mathcal P)=\left(\dfrac{\partial h_i}{\partial \xi^1_j}(\mathcal P)\right)_{1\leq i,j\leq 5}$.} Notice from \eqref{eq:formula for h2}, \eqref{eq:formula for h3} and \eqref{eq:h5formula} that $h_2$, $h_3$ and $h_5$ do not depend on $\xi^1$ then 
$$\nabla_{\xi^1}h_2(\mathcal P)=\nabla_{\xi^1}h_3(\mathcal P)=\nabla_{\xi^1}h_5(\mathcal P)={\bf 0}.$$

We calculate $\nabla_{\xi^1}h_1(\mathcal P)$. From \eqref{eq:formula for h1}, and the fact that $p_1=h_1(\mathcal P)=\dfrac{\widetilde{F}_{1r}(0,{\bf 0}, P)}{\widetilde{F}_{1b}(0,{\bf 0}, P)}$ we get
\begin{align*}
\nabla_{\xi^1}h_1(\mathcal P)&=-\dfrac{\widetilde{F}_{1r}(0,{\bf 0}, P)}{\widetilde{F}_{1b}(0,{\bf 0}, P)^2}\nabla_{\xi^1} \widetilde {F}_{1b}(0,{\bf 0},P)=-\dfrac{p_1}{\widetilde F_{1b}(0,{\bf 0},P)}\nabla_{\xi^1} \widetilde {F}_{1b}(0,{\bf 0},P).
\end{align*}
Recall from \eqref{eq:values at 0}
$$\widetilde{F}_{1b}(0,{\bf 0},P)=\rho_0+\dfrac{d_0}{\Delta_b}\left(\dfrac{\Delta_b}{n_b}+\dfrac{p_4}{\rho_0}\right).$$
Similarly as in \eqref{eq:grad tilde F1r}
$$\nabla_{\xi^1}\widetilde{F}_{1b}(\mathcal P)=\left(0,0,1,\dfrac{1}{k_0\Delta_b},0\right).$$
Hence
\[
\nabla_{\xi^1}h_1(\mathcal P)=\dfrac{-p_1}{\rho_0+\dfrac{d_0}{\Delta_b}\left(\dfrac{\Delta_b}{n_b}+\dfrac{p_4}{\rho_0}\right)}\left(0,0,1,\dfrac{1}{k_0\Delta_b},0\right).
\]
We next calculate $\nabla_{\xi^1}h_4(\mathcal P)$.
Recall from \eqref{eq:values at 0}
$$\mathcal A_b({\bf 0})=-\dfrac{n_b}{\Delta_b\rho_0},\quad
\Lambda_{2,b}(0,{\bf 0})=\dfrac{1}{\Delta_b},\quad
\Lambda_{2,r}(0,{\bf 0})=\dfrac{1}{\Delta_r},\quad
\mu_r(0,{\bf 0})=\mu_b(0,{\bf 0})=0,
$$
and as in \eqref{eq:derivative of tilde mur} $\nabla_{\xi^1}\widetilde{\mu}_b(0,{\bf 0},P)=\dfrac{1}{\Delta_b \rho_0}\left(0,0,0,1,0\right)$,
then from \eqref{eq:formula for h4}
\[
\nabla_{\xi^1} h_4(\mathcal P)=\dfrac{p_1\dfrac{1}{\Delta_r}}{\dfrac{1}{\Delta_r\rho_0}\dfrac{1}{\Delta_b}}\dfrac{1}{\Delta_b\rho_0}(0,0,0,1,0)=p_1(0,0,0,1,0)
\]
We conclude that
\begin{equation}\label{eq:Total derivative xi1}
\nabla_{\xi^1} H(\mathcal P)=
\begin{bmatrix}
    0 & 0 & \dfrac{-p_1}{\rho_0+\dfrac{d_0}{\Delta_b}\left(\dfrac{\Delta_b}{n_b}+\dfrac{p_4}{\rho_0}\right)}& \dfrac{\dfrac{-p_1}{k_0\Delta_b}}{\rho_0+\dfrac{d_0}{\Delta_b}\left(\dfrac{\Delta_b}{n_b}+\dfrac{p_4}{\rho_0}\right)}& 0  \\
    0 & 0 & 0 & 0 & 0\\    
    0 & 0 & 0 & 0 & 0\\    
    0 & 0 & 0 & p_1 & 0\\
0 & 0 & 0 & 0 & 0\\    
  \end{bmatrix}.
\end{equation}

Notice that $\nabla_{\xi^1} H(\mathcal P)$ is an upper triangular matrix with eigenvalues $0$ (with multiplicity $4$) and $p_1$, the spectral radius of $\nabla_{\xi^1} H(\mathcal P)$ is
\begin{equation}\label{eq:spect of grad xi1}
R_{\xi^1}=|p_1|.
\end{equation}

%


\textbf{Choice of the norm.}
We are now ready to construct a norm for which $H$ is a contraction in the last two variables.
In fact, 
we will construct a norm denoted by $\|\cdot\|_{k_0}$ in $\R^5$ depending on $k_0$ such that for small $k_0$ 
\begin{equation}\label{eq:sum of norms less than one}
\left|\left\|\nabla_{\xi^0}H(\mathcal P)\right\|\right|_{k_0}+\left|\left\|\nabla_{\xi^1}H(\mathcal P)\right\|\right|_{k_0}<1
\end{equation}
where $\left|\left\|\cdot\right\|\right|_{k_0}$ is the matrix norm in $R^{5\times 5}$ induced by $\|\cdot\|_{k_0}$. Recall that $k_0=\rho_0/d_0$.

We will show that for each $0<k_0\leq \dfrac{\(\Delta_r-\Delta_b\)^2}{4\Delta_r\Delta_b}$, there exist $\lambda_1,\cdots ,\lambda_5$ positive depending on $k_0$ such that the norm in $\R^5$ having the form 
$$\|x\|_{k_0}=\max\left(\lambda_1|x_1|,\lambda_2|x_2|,\lambda_3|x_3|,\lambda_4|x_4|,\lambda_5|x_5|\right),$$
satisfies
\[
\left|\left\|\nabla_{\xi^0}H(\mathcal P)\right\|\right|_{k_0}<1.
\]

We first choose $\lambda_1=\lambda_2=\lambda_5=1$.
Assume $x\in \R^5$ with $\|x\|_{k_0}=1$, which implies $|x_i|\leq \dfrac{1}{\lambda_i}$. Then from \eqref{eq:Total derivative xi0}, and \eqref{eq:ac}
\begin{align*}
\left\|\nabla_{\xi^0} H(\mathcal P)x\right\|_{k_0}&=\max\left(\left|a\,x_3+\dfrac{a}{k_0\Delta_r}\,x_4\right|,\lambda_4 \,|c\, x_1|\right)\\
&\leq \max\left(\dfrac{1}{\lambda_3}\,|a|+\dfrac{1}{\lambda_4}\dfrac{|a|}{k_0\Delta_r},\lambda_4\,|c|\right)\leq 
\max\left(\dfrac{|a|}{\lambda_3}+\dfrac{\delta_0}{\lambda_4\,|c|},\lambda_4\,|c|\right)
\end{align*}
with $a$ and $c$ defined in \eqref{eq:form of a c}. Hence
\[
\left|\left\|\nabla_{\xi^0}H(\mathcal P)\right\|\right|_{k_0}=\max_{\|x\|_{k_0}= 1}\left\|\nabla_{\xi^0} H(\mathcal P)x\right\|_{k_0}
\leq 
\max\left\{\dfrac{|a|}{\lambda_3}+\dfrac{\delta_0}{\lambda_4\,|c|},\lambda_4\,|c|\right\}.
\]
We will choose $\lambda_3$ and $\lambda_4$ so that the last maximum is less than one.
Let $\delta_0<\delta_1<\delta_2<1$, with $\delta_0$ defined in \eqref{eq:ac}, $\lambda_4=\delta_2/|c|$ and $\lambda_3=N\,\lambda_4$, with $N$ to be determined depending only on $n_r$ and $n_b$.
Then 
\[
\max\left\{\dfrac{|a|}{\lambda_3}+\dfrac{\delta_0}{\lambda_4\,|c|},\lambda_4\,|c|\right\}=\max \left\{\dfrac{1}{\lambda_4\,|c|}\left(\dfrac{a\,c}{N}+\delta_0 \right),\lambda_4\,|c| \right\},
\]
notice that $ac>0$ from \eqref{eq:trick 3}.
From \eqref{eq:ac}, $a\,c\leq k_0\,\Delta_r\leq \Delta_r\,\dfrac{\(\Delta_r-\Delta_b\)^2}{4\Delta_r\Delta_b}:=B(r,b)$, we obtain
\[
\max\left\{\dfrac{|a|}{\lambda_3}+\dfrac{\delta_0}{\lambda_4\,|c|},\lambda_4\,|c|\right\}\leq \max \left\{\dfrac{1}{\delta_2}\left(\dfrac{B(r,b)}{N}+\delta_0 \right),\delta_2 \right\}.
\]
Now pick $N$, large depending only on $n_r$ and $n_b$, so that
$\dfrac{B(r,b)}{N}+\delta_0<\delta_1$. Therefore,  
\[
\left|\left\|\nabla_{\xi^0}H(\mathcal P)\right\|\right|_{k_0}\leq \max\left\{\dfrac{\delta_1}{\delta_2},\delta_2 \right\}:=s_0<1,
\]
for all $0<k_0\leq \dfrac{\(\Delta_r-\Delta_b\)^2}{4\Delta_r\Delta_b}$.

%

It remains to show that with the above norm $\|\cdot \|_{k_0}$ we also have $\eqref{eq:sum of norms less than one}$. To do this we need to choose $k_0$ sufficiently small.
In fact, from \eqref{eq:Total derivative xi1} and for $\|x\|_{k_0}= 1$ we have (since $\lambda_3=N\,\lambda_4$, $\lambda_4=\delta_2/|c|$, and $ac<k_0\Delta_r\leq B(r,b)$)
\begin{align*}
\left|\left\|\nabla_{\xi^1} H(\mathcal P) x\right\|\right|_{k_0}&=|p_1|\,\max\left(\left| a\,x_3+\dfrac{1}{k_0\Delta_b}a\,x_4\right|,\lambda_4\,|x_4|\right)\\
&\leq |p_1|\,\max\left(\dfrac{|a|}{\lambda_3}+\dfrac{1}{\lambda_4}\dfrac{|a|}{k_0\Delta_b},1\right)\\
&=
|p_1|\,\max\left(\dfrac{1}{\delta_2}\(\dfrac{a\,c}{N}+\dfrac{a\,c}{k_0\,\Delta_b}\),1\right)\\
&\leq
|p_1|\,\max\left(\dfrac{1}{\delta_2}\(\dfrac{B(r,b)}{N}+\dfrac{\Delta_r}{\Delta_b}\),1\right):=|p_1|\,s_1,
\end{align*}
for all $0<k_0\leq \dfrac{\(\Delta_r-\Delta_b\)^2}{4\Delta_r\Delta_b}$ with $s_1$ depending only on $n_r$ and $n_b$ ($s_1>1$).
Hence 
\[
\left|\left\|\nabla_{\xi^1}H(\mathcal P)\right\|\right|_{k_0}
\leq
|p_1|\,s_1.
\]
Therefore  
$$
\left|\left\|\nabla_{\xi^0} H(\mathcal P) \right\|\right|_{k_0}+\left|\left\|\nabla_{\xi^1} H(\mathcal P) \right\|\right|_{k_0}
\leq s_0+ |p_1|\,s_1.$$
From \eqref{eq:formula for delta} and \eqref{eq:P},  $p_1\to 0$ as $k_0\to 0$, and therefore we obtain \eqref{eq:sum of norms less than one} for $k_0$ close to $0$.

\textbf{Verification of \eqref{eq:Contraction}}.
Let us now show that $H$ satisfies the Lipschitz condition \eqref{eq:Lip in xi0,xi1} with constants satisfying \eqref{eq:Contraction} in a sufficiently small neighborhood of $\mathcal P$ and with respect to the norm chosen.
In fact, for $k_0$ sufficiently small, from \eqref{eq:sum of norms less than one}, there is $0<c_0<1$ so that 
\[
\left|\left\|\nabla_{\xi^0}H(\mathcal P)\right\|\right|_{k_0}+\left|\left\|\nabla_{\xi^1}H(\mathcal P)\right\|\right|_{k_0}\leq c_0<1.
\]
Since $H$ is $C^1$, there exists 
a norm-neighborhood $N_{\varepsilon}(\mathcal P)$ of $\mathcal P$ as in Theorem \ref{thm: Existence}, so that
\begin{align}\label{eq:contraction condition}
&\max_{\(t;\zeta^0,\zeta^1;\xi^0,\xi^1\)\in N_{\varepsilon}(\mathcal P)}\left|\left\|\nabla_{\xi^0} H\(t;\zeta^0,\zeta^1;\xi^0,\xi^1\) \right\|\right|_{k_0}\\
&\qquad \qquad 
+\max_{\(\hat t;\hat \zeta^0,\hat \zeta^1;\hat \xi^0,\hat \xi^1\)\in N_{\varepsilon}(\mathcal P)}\left|\left\|\nabla_{\xi^1} H\(\hat t;\hat \zeta^0,\hat \zeta^1;\hat \xi^0,\hat \xi^1\) \right\|\right|_{k_0}\leq c_1,\notag
\end{align}
for some $c_0<c_1<1$. Then by Proposition \ref{eq:prop Bound on C0+C1}, the inequalities \eqref{eq:Lip in xi0,xi1} and \eqref{eq:Contraction} hold with $C_0=\max_{N_{\varepsilon}(\mathcal P)}\left|\left\|\nabla_{\xi^0} H\right\|\right|$ and $C_1=\max_{N_{\varepsilon}(\mathcal P)}\left|\left\|\nabla_{\xi^1} H\right\|\right|$.

\textbf{Verification of \eqref{eq:bd on h1}.}
If $k_0<\dfrac{(\Delta_r-\Delta_b)^2}{4\Delta_r\Delta_b}$, then from Theorem \ref{thm:Algebraic system} $|p_1|<1$, and so $|h_1(\mathcal P)|<1$. Hence by continuity there exists a neighborhood of  $\mathcal P$ so that $\left|h_1(t;\zeta^0,\zeta^1;\xi^0,\xi^1)\right|< 1$ for all $(t;\zeta^0,\zeta^1;\xi^0,\xi^1)$ in that neighborhood. Therefore, condition \eqref{eq:bd on h1} in Theorem \ref{thm: Existence} is satisfied.

We conclude that for small values of $k_0$ Theorem \ref{thm: Existence} is applicable and therefore there exists a unique local solution to the system \eqref{eq:Optic System} with $Z'(0)=P$ and $P $ given in \eqref{eq:P}. Notice also that for such a solution $Z$, we have $z_1(0)=0$, and $\left| z_1'(0)\right|=\left| h_1(\mathcal P)\right|=\left|p_1\right|<1$, then by reducing the neighborhood of zero, if necessary, we have that the solution satisfies $|z_1(t)|\leq |t|$. Since we also showed in Theorem \ref{thm:Algebraic system} that $p_1\neq 0$,
Theorem \ref{thm:Converse} is applicable and 
the proof of Theorem \ref{thm:Existence last} is then complete.
\end{proof}

Summarizing the question of solvability of Problem B in the plane in a neighborhood of zero, we have: 
\begin{enumerate}
\item If $k_0>\dfrac{(\Delta_r-\Delta_b)^2}{4\Delta_r\Delta_b}$ then Problem $B$ is not locally solvable.

\item Let $A=(0,\rho_0)$, $B=(0,\rho_0+d_0)$, and $0<k_0<C(r,b)<\dfrac{(\Delta_r-\Delta_b)^2}{4\Delta_r\Delta_b}$.
Then there exists $\delta>0$
and a unique lens $(L,S)$ with lower face $L=\{\rho(t)x(t)\}_{t\in [-\delta,\delta]}$, $x(t)=(\sin t,\cos t)$, and upper face $S=\{f_b(t)\}_{t\in [-\delta,\delta]}$, $f_b$ defined in Theorem \ref{thm:Converse}, such that $L$ passes through $A$, $S$ passes through $B$, and so that $(L,S)$ refracts all rays emitted from $O$ with colors r and b and direction $x(t)$, $t\in [-\delta,\delta]$ into the vertical direction $e=(0,1)$. 
\end{enumerate}

\begin{remark}\label{rmk:final remark}\rm
In this final remark, we point out that Theorem \ref{thm: Existence} is not applicable to find solutions to the system \eqref{eq:Optic System} when $k_0\leq \dfrac{(\Delta_r-\Delta_b)^2}{4\Delta_r\Delta_b}$ and $k_0$ is away from $0$. 
In this case we claim that there is no norm in $\R^5$ for which we can obtain \eqref{eq:Lip in xi0,xi1} with $C_0$ and $C_1$ satisfying \eqref{eq:Contraction}.
In fact, from \eqref{eq:formula for delta}, \eqref{eq:P}, \eqref{eq:Rxi0 bound}, and \eqref{eq:spect of grad xi1}
$$R_{\xi^0}+R_{\xi^1}=\sqrt{\dfrac{2\Delta_b}{\Delta_r+\Delta_b+\sqrt{\delta}}}+|p_1|\to \sqrt{\dfrac{2\Delta_b}{\Delta_r+\Delta_b}}+1>1,$$
as $k_0\to \dfrac{(\Delta_r-\Delta_b)^2}{4\Delta_r\Delta_b}$ from below. 
Hence for $k_0$ close to $\dfrac{(\Delta_r-\Delta_b)^2}{4\Delta_r\Delta_b}$, we have $R_{\xi^0}+R_{\xi^1}>1$, and hence by Corollary \ref{cor:spectral radii} the claim follows.
\end{remark}

\section*{Acknowledgements}{\small C. E. G. was partially supported by NSF grant DMS--1600578, and A. S. was partially supported by Research Grant 2015/19/P/ST1/02618 from the National Science Centre, Poland, entitled "Variational Problems in Optical Engineering and Free Material Design".}\\ 
\vspace{-.6cm}
\begin{wrapfigure}{l}{0.2\textwidth} 
\begin{center} \includegraphics[width=0.2\textwidth]{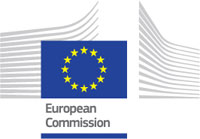} 
\end{center} 
\end{wrapfigure}\\
{\footnotesize This project has received funding from the European Union's Horizon 2020 research and innovation program under the Marie Sk\l{}odowska-Curie grant agreement No. 665778.}\\ \\


\end{document}